\documentclass[11pt,regno]{amsart}
\usepackage{overpic}
\usepackage{dsfont}
\usepackage{euscript,graphicx,color}
\usepackage{amsmath}
\usepackage{amssymb}
\usepackage{amsfonts}
\usepackage{mathrsfs}
\usepackage{epstopdf,rotating}
\newtheorem{theorem}{Theorem}
\newtheorem{theo}{Theorem}[section]
\newtheorem{theo-app}{Theorem}[section]
\newtheorem{corollary}[theo]{Corollary}
\newtheorem{proposition}[theo]{Proposition}
\newtheorem*{claim}{Claim}
\newtheorem{lemma}[theo]{Lemma}
\theoremstyle{definition}
\newtheorem{example}[theo]{Example}
\newtheorem{remark}[theo]{Remark}
\newtheorem*{remark*}{Remark}
\newtheorem{definition}[theo]{Definition}

\numberwithin{equation}{section}

\newcommand{\eqdef}{:=}
\renewcommand{\hat}{\widehat}

\newcommand{\sA}{\mathscr{A}}

\DeclareMathOperator{\BD}{BD}%Bounded Distortion
\DeclareMathOperator{\Indiff}{Indiff}

\newcommand{\R}{\mathbb{R}}
\newcommand{\N}{\mathbb{N}}

\DeclareMathOperator{\diam}{diam}

\DeclareMathOperator{\dist}{Dist}
\DeclareMathOperator{\interior}{int}

\DeclareMathOperator{\Lip}{Lip}

\DeclareMathOperator{\Crit}{{\rm Crit}}

\DeclareMathOperator{\Comp}{{\rm Comp}}

\def\bN{\mathbb{N}}\def\bC{\mathbb{C}}

\def\bR{\mathbb{R}}

\def\cA{\mathscr{A}}
\def\cL{\EuScript{B}}

\def\cU{\EuScript{U}}

\def\cM{\EuScript{M}}

\def\cL{\EuScript{L}}
\def\cO{\mathcal{O}}
\def\cB{\mathscr{B}}

\def\e{\varepsilon}

\DeclareMathSymbol{\varnothing}{\mathord}{AMSb}{"3F}
\renewcommand{\emptyset}{\varnothing}

\author[K. Gelfert]{Katrin Gelfert} \address{Instituto de Matem\'atica, Universidade Federal do Rio de Janeiro,
Cidade Universit\'aria - Ilha do Fund\~ao, Rio de Janeiro 21945-909,  Brazil}
\email{gelfert@im.ufrj.br}

\author[F. Przytycki]{Feliks Przytycki$^\dag$} \address{Institute of Mathematics, Polish Academy of Sciences, ul. \'{S}niadeckich 8, 00-956 Warszawa, Poland}
\email{feliksp@impan.pl}

\author[M. Rams]{{Micha\l} Rams$^\dag$} \address{Institute of Mathematics, Polish Academy of Sciences, ul. \'{S}niadeckich 8, 00-956 Warszawa, Poland}
\email{rams@impan.pl}

\thanks{
This paper was partially supported by CNPq, Faperj,  Pronex (Brazil), and BREUDS.
$\dag$Partially supported by Polish MNiSW Grant N N201 607640.
}
\begin{document}

\date{\today}

\title[Lyapunov spectrum]{Lyapunov spectrum for multimodal maps}

\begin{abstract}
We study the dimension spectrum of Lyapunov exponents for multimodal maps
of the interval and their generalizations.
We also present related results for rational maps on the Riemann sphere.
\end{abstract}

\keywords{}
\subjclass[2010]{%
37D25, %Nonuniformly hyperbolic systems (Lyapunov exponents, Pesin theory, etc.)
37D35, % Thermodynamic formalism, variational principles, equilibrium states
28A78, %   	Hausdorff and packing measures%28D20, % Entropy and other invariants
37E05,  %  	Maps of the interval (piecewise continuous, continuous, smooth)
37C45, % Dimension theory of dynamical systems
37F10% Rational maps
}

\maketitle
\tableofcontents

%--------------------------------------------------------------------------------------------------------
\section{Introduction}
%--------------------------------------------------------------------------------------------------------

This paper is an interval multimodal mapping counterpart of the authors' earlier paper on complex rational maps on the Riemann sphere, \cite{GelPrzRam:10}. Some new ideas yield also a progress in the complex setting (see Appendix A).

We consider a $C^3$ (or $C^2$ with bounded distortion)  multimodal map $f$ on a finite union $\hat I$ of pairwise disjoint closed intervals in $\bR$ and the iteration of $f$ on the maximal forward invariant subset $K$ (in other words: $K$ is  $f$-forward invariant and $f|_K$ satisfies the Darboux property\footnote{See \cite[Appendix A]{PrzRiv:13} and references therein.}). We assume that $f|_K$ is topologically transitive and has positive topological entropy. Given $\alpha\in\bR$, we consider the level sets of Lyapunov exponents of $f$
\[
	\cL(\alpha):=
	\big\{x\in K\colon \chi(x):=\lim_{n\to\infty} \frac1n \log\,\lvert (f^n)'(x)\rvert=\alpha\big\}
\]	
and study the function $\alpha\mapsto \dim_{\rm H}\cL(\alpha)$, which is called Hausdorff dimension spectrum of level sets $\cL(\alpha)$ for Lyapunov exponents, or shortly, the  \emph{Lyapunov spectrum}.
We prove that this spectrum can be expressed
in terms of the Legendre transform of the \emph{geometric pressure function}, that is, of the function $t\mapsto P(f|_K, -t\log\,\lvert f'\rvert)$.
We only allow non-flat critical points in $K$.

This is a classical result in the uniformly hyperbolic case, see e.g. \cite{Ols:95,Wei:99,Pes:97} and earlier
papers by mathematical physicists, see e.g. \cite{EckPro:86,ColLebPor:87,Ran:89}.
But our system is just `chaotic' in the sense that it is topologically transitive and has positive topological entropy.
Observe that our system has a uniformly hyperbolic factor obtained in the following way:
 Identifying end points of gaps (intervals complementary to $K$ in case that $K$ Cantor set; it can be {\it a priori} a finite union of closed intervals $\hat I\,$),  by a continuous change  of coordinates we get a piecewise continuous piecewise monotone map of the interval with constant slope $h_{\rm {top}}(f|_K)$. In particular, this map is piecewise uniformly hyperbolic, see \cite[Appendix A]{PrzRiv:13}. Unfortunately, the geometric potential $x\mapsto -t\log\,\lvert f'(x)\rvert$ does not survive this coordinate change. Hence this approach would not help studying the geometric properties of the original map.

We are also studying the irregular part of the spectrum, that is, the level sets of irregular points
\[
\cL(\alpha,\beta) \eqdef
 \left\{x\in K\colon \underline\chi(x)=\alpha, \, \overline\chi(x)=\beta\right\},\quad
 \alpha<\beta
\]
for the lower and the upper Lyapunov exponents.
In some cases we obtain a precise formula, conjectured in previous papers in both the real and the complex cases. Our main tools are:  hyperbolic approximation to show lower bounds and conformal measures and their generalizations to show upper bounds.
The existence of those conformal measures has been proved via an inducing scheme in \cite[Theorem A]{PrzRiv:13}.

In particular, we show that some of certain level sets are empty. Denoting by $\chi_{\inf}$ the infimum of exponents of  $f$-invariant probability measures on $K$, we obtain that $\chi(x)\ge \chi_{\inf}$ for every $x\in K$ for which $\chi(x)$ exists and  is different from $-\infty$. This was conjectured before  in
\cite{IomTod:11,IomTod:13}\footnote{After this paper was written, J. Rivera-Letelier explained to us that $\chi(x)\ge\chi_{\inf}$ follows easily from $-\lim_{n\to\infty} \frac1n\log\max\,\lvert W_n\rvert\ge\chi_{\inf}$, where the maximum is taken over all connected components $W_n$ of preimages $f^{-n}(W)$ and $W\subset \hat I$ a sufficiently small interval. The latter inequality was proved in \cite{R-L:} under additional, but not substantial, assumptions on $f$. See Appendix B for details.}.

Moreover, in the case that $K=\hat I$,  $f$ is non-exceptional, and has no indifferent periodic orbits in $K$, we conclude that $\underline\chi(x)^\sharp\ge \chi_{\inf}$ provided $\overline\chi(x)>0$. Here  $\underline\chi(x)^\sharp$ denotes a certain non-negative average of $\underline\chi(x)$ and $\overline\chi(x)$.

The novelty of this paper, compared to methods used in \cite{GelPrzRam:10}, is that we investigate not only noncritical pull-backs but also allow pull-backs of bounded criticality (Topological Collet-Eckmann type feature).

Our setting of multimodal maps has been elaborated recently by FP and J.~Rivera-Letelier~\cite{PrzRiv:13}. Substantial earlier contributions in this theory have been done in particular by F. Hofbauer, G. Iommi and M. Todd.
\smallskip

\noindent\textbf{Acknowlegement.}
The authors thank G.~Levin and J.~Rivera-Letelier for  inspiring discussions.

%--------------------------------------------------------------------------------------------------------
\subsection{Generalized Multimodal Maps}\label{ss:generalized}
%--------------------------------------------------------------------------------------------------------

In this paper we work in the {\it multimodal} setting elaborated in \cite{PrzRiv:13}.
We will consider a $C^r$ map $f$ on a certain open subset of the real line into the real line. We assume $r=2$ if not stated otherwise. A point $x$ is said to be \emph{critical} if  $f'(x)=0$.
We denote by $\Crit(f)$ the set of all critical points of $f$.
We denote by $\Crit_T(f)\subset\Crit(f)$ the set of all turning critical points and by $\Crit_I(f)\subset \Crit(f)$ the set of all inflection critical points.
We say that $f$ is \emph{non-flat} at a critical point $c$ if
 $f(x)=\pm\lvert\phi(x)\rvert^d +f(c)$ for some $d\ge 2$ and $\phi$ a $C^r$ diffeomorphism in a neighborhood of $c$ with $\phi(c)=0$, see \cite[Chapter IV]{deMvanS:}.
In this paper we assume that $f$ has only a finite number of all critical points and that all of them are non-flat.
We call such a map {\it multimodal}.

We denote by $\deg f$ the geometric degree of $f$, that is,
\[
	\deg f:=\sup_{x\in K}\#f^{-1}|_K(\{x\})\,.
\]

Let $K\subset\bR$ be an infinite compact set in the domain of $f$. We assume that $K$ is \emph{forward invariant} under $f$, that is $f(K)= K$. We also assume that $f\colon K\to K$ is \emph{topologically transitive}, that is, for all $V_1,V_2\subset K$ nonempty and open in $K$ there is $n\ge1$ such that $f^n(V_1)\cap V_2\ne\emptyset$.

We usually will assume that $f|_K$ has positive topological entropy.

We assume that there exists a covering of $K$ by a finite family of pairwise disjoint closed intervals $\hat I_j$ in the domain of $f$, with end points in $K$,
such that $K$ is the maximal compact forward invariant set in $\hat I_K\eqdef\bigcup_j\hat I_j$.

We assume that $\hat I_K\setminus K$ does not contain critical points.

Furthermore, we consider a $C^r$ extension of $f$ to some neighborhoods $U^j$ of $\hat I_j$ such that $U^j$ are pairwise disjoint and none of $U^j\setminus K$ contains a critical point. Let $U=\bigcup_j U^j$. Sometimes to distinguish $U$ we write $U_K$.

Following \cite{PrzRiv:13} we denote the space of the quadruples $(f,K,\hat I_K, U)$ satisfying the above properties by $\sA_+^r$. We simply write $\sA_+$ if $r=2$.
As in~\cite{PrzRiv:13}, the subscript $+$ indicates that $f|_K$ has positive topological entropy. We will omit it in case this assumption is not necessary. Sometimes we abbreviate to $(f,K, \hat I_K)$, $(f,K,U)$, or just $(f,K)$.
For such quadruples, triples, couples, we sometimes use the name {\it generalized $\sA$-multimodal maps} or {\it systems}.

If  $(f,K)\in\sA$, then $K$ is either the union of a finite collection of compact intervals $K=\hat I_K$ or a Cantor set, see \cite[Lemma 2.1]{PrzRiv:13}.

\subsection{Periodic orbits}

We call a point $p\in U$ \emph{periodic} if there exists $m\ge 1$ such that $f^m(p)=p$ and we call the smallest such number $m$ its \emph{period}. We denote by $\cO(p)$ its periodic orbit and define the \emph{basin of attraction} of $\cO(p)$ by
\[
	\cB(p)\eqdef
	\interior \{x\in I\colon  f^n(x)\to \cO(p)\text{ as } n\to\infty\}.
\]
The orbit $\cO(p)$ is called \emph{attracting} if $\cO(p)\subset \cB(\cO(p))$. The orbit is called \emph{repelling} if $\lvert (f^m)'(p)\rvert\ge1$ and some small neighborhood $W$ of $\cO(p)$ is forward invariant for $g={f|_W}^{-1}$ and satisfies $g^n(W)\to\cO(p)$ as $n\to\infty$. If $\lvert(f^m)'(p)\rvert>1$ then $p$ is said to be \emph{hyperbolic repelling}. In the remaining cases we call $\cO(p)$ \emph{indifferent}. We denote the set of indifferent periodic points by $\Indiff(f)$.

The multimodal systems we study in this paper are natural generalizations of the class of multimodal maps of the interval $f\colon I\to I$ (extended to a neighborhood $U$ of $I$), where $K=I\setminus \bigcup_p\cB(p)$ is called \emph{Julia set}, see
\cite[Example 1.8]{PrzRiv:13} or \cite[Chapter IV]{deMvanS:}. A wider class of examples are so-called {\it basic sets} in the spectral decomposition, see again \cite[Example 1.9]{PrzRiv:13} or \cite[Chapter III.4]{deMvanS:}.

A special case occurs if $K=\hat I_K$ and, in particular, if $K=I$. This will allow us to prove stronger forms of Theorem~\ref{main3} and Theorem~\ref{strangestrange}.

For $(f,K)\in\sA$ there are no periodic orbits in $\hat I_K\setminus K$. Further, periodic orbits in $K$ are either hyperbolic repelling (and there are infinitely many such orbits if $(f,K)\in\sA_+$) or non-hyperbolic repelling or indifferent (and of those two types there are at most finitely many orbits).
By changing $f$ outside $K$ and shrinking $U$ if necessary, we can also achieve that  periodic orbits outside $K$ are hyperbolic repelling (see \cite[Remark 1.6 and Appendix A]{PrzRiv:13}).

\begin{definition}\label{wi}
For $(f,K,U)\in \sA$ we call $K$  \emph{weakly isolated},  if there exists an open set $U'\subset U$ such that every $f$-periodic orbit $\cO(p)$ which is contained in $U'$ must be contained in $K$.
\end{definition}

Recall that in the definition of $\sA$ we required that $K$ is the maximal invariant set in $\hat I_K$, so the weak isolation condition  is not automatically satisfied, {\it a priori} there can be hyperbolic repelling periodic orbits close to $K$ intersecting $\bR\setminus \widehat I_K$, see \cite[Example 2.12]{PrzRiv:13}.

If $K$ is maximal in an open neighborhood of $K$ then it is weakly isolated, and even isolated, see the definition in Subsection~\ref{ss:pressure}, but observe that $\hat I_K$ does not contain an open neighborhood of $K$.

In the main theorems of this paper we need to assume that $K$ is weakly isolated. Notice however that if $K$ is Julia set (see the definition above), then the weak isolation condition is automatically satisfied.

%-----------------------------------------------------------------------------------------------
\subsection{Distortion} %-----------------------------------------------------------------------------------------------

Let us introduce some more notation following the terminology in~\cite{deMvanS:}.
We denote by
\[
	\dist g|_Z \eqdef \sup\limits_{x,y\in Z}\frac{\lvert g'(x)\rvert}{\lvert g'(y)\rvert}
\]
the maximal \emph{distortion} of a differentiable map $g$ on a set $Z$.
Given two intervals $U\subset V$, we say that $V$ is an \emph{$\varepsilon$-scaled neighborhood} of $U$ if $V\setminus U$ has two components of length $\varepsilon\lvert U\rvert$. In this case we also denote
\begin{equation}\label{notescaled}
	V \eqdef (1+2\varepsilon)\diamond U.
\end{equation}
Correspondingly, we write
\[
	U=(1+2\varepsilon)^{-1}\diamond V\,.
\]
We say that $(f,K)\in \sA$ satisfies a \emph{bounded distortion condition} if there exists $\delta>0$ such that for every $\varepsilon>0$ there is $C=C(\varepsilon)>0$ such that the following holds:
For every  pair of intervals $I_1, I_2 \subset U$ such that $I_1$ intersects $K$ and $\lvert I_2\rvert\le \delta$ and every $n>0$ for which $f^n$ maps an interval $I_1'$ containing $I_1$ diffeomorphically onto an interval $I_2'$ being the $\varepsilon$-scaled neighborhood of $I_2$ then
\begin{equation}\label{disttt}
	\dist f^n|_{I_1}\le C(\varepsilon).
\end{equation}
This easily implies that there exists $\tau(\varepsilon)$ such that $I_1'$ contains a $\tau(\varepsilon)$-neighbourhood of $I_1$. If $(f,K)\in \sA$ satisfies bounded distortion condition, we write $(f,K)\in\sA^{\BD}$.

\begin{remark}\label{rem:boudeddistorionremark}
For $(f,K)\in\sA^{\BD}$, the only periodic orbits in $K$ are either hyperbolic repelling, or indifferent one-side repelling (see \cite[Remark 1.11]{PrzRiv:13}).

If $(f,K)\in \sA^3$ and all periodic orbits in $K$ are hyperbolic repelling, then $f$ can be changed outside $K$ if necessary, so that $(f,K)\in \sA^{\BD}$ (see \cite[comments after Theorem A, Remark 2.14, and Lemma A.4]{PrzRiv:13}). The $\sA^{\BD}$ assumption can be therefore considered as the weaker and the appropriate one  if only facts concerning $f|_K$ are asserted. Indeed, it assumes only  $C^2$ smoothness and allows indifferent periodic orbits in $K$.
\end{remark}

%-----------------------------------------------------------------------------------------------
\subsection{Exceptional sets}\label{exceptional}%-----------------------------------------------------------------------------------------------

Denote by $NO(f,K)$ the set of points where $f|_K$ is not open.
Observe that $NO(f,K)\subset \Crit_T(f)\cup\partial\hat I_K$ (see also~\cite[Lemma 2.2]{PrzRiv:13}).
We call $S\eqdef \Crit(f)\cup\partial\hat I_K$ the \emph{singular set} and $S'\eqdef \Crit(f)\cup NO(f,K)$ the \emph{restricted singular set}.
 In fact, $S'$ plays a more substantial role in this paper (and in \cite{PrzRiv:13}) than $S$ does.

Given an arbitrary finite set $\Sigma\subset K$, we call a nonempty set $E\subset K$ \emph{weakly $\Sigma$-exceptional}, if $E$ is non-dense in $K$ and satisfies
\begin{equation}\label{except-condition}
	{f|_K}^{-1}(E)\setminus  \Sigma \subset E
	\,.
\end{equation}	
Condition~\eqref{except-condition} can be read that all backward trajectories for $f$ which start in $E$ and go outside $E$ are ``immediately blocked'' by $\Sigma$.
We say that $E$ is \emph{$\Sigma$-exceptional} if it is forward invariant and $\Sigma$-exceptional.

Observe that the union of (weakly) $\Sigma$-exceptional sets is (weakly) $\Sigma$-exceptional. Note that, by~\cite[Proposition 2.7]{PrzRiv:13}, all weakly $\Sigma$-exceptional sets $E$ are finite and have a uniformly bounded cardinality.
Therefore, there is a finite maximal weakly $\Sigma$-exceptional set. Indeed, the union of all weakly $\Sigma$-exceptional sets clearly has the property \eqref{except-condition} and is not dense since there is a uniform bound for its cardinality.

Finally we call a point $x\in K$ \emph{(weakly) $\Sigma$-exceptional} if $x$ belongs to a (weakly) $\Sigma$-exceptional set.

We shall apply this for $\Sigma=S'$. If there does not exist any $S'$-exceptional set, we call $f$ \emph{non-$S'$-exceptional} or just \emph{non-exceptional}, otherwise we call $f$ \emph{exceptional}.

\begin{example}\label{ex:1}
 	Chebyshev maps such as, for example, $f(x)= x^2-2$ with $K=\hat I_K$ being the unit interval, are $S$-exceptional, but many other interval maps are not.
	
	The critical value of a unimodal map (for example, a quadratic polynomial)
constitutes a (one-element) weakly $S$-exceptional set.
\end{example}

%--------------------------------------------------------------------------------------------------
\subsection{Lyapunov exponents}\label{Lyapunov}
%--------------------------------------------------------------------------------------------------

We are interested  in the spectrum of Lyapunov exponents of the map $f\colon K\to K$. Given $x$ we denote by $\underline\chi(x)$ and $\overline\chi(x)$ the \emph{lower} and \emph{upper Lyapunov exponent} at $x$, respectively, where
\[
\underline\chi(x)\eqdef\liminf_{n\to\infty}\frac{1}{n}\log\, \lvert(f^n)'(x)\rvert,\quad
\overline\chi(x)\eqdef\limsup_{n\to\infty}\frac{1}{n}\log\, \lvert(f^n)'(x)\rvert.
\]
If both values coincide then we call the common value the
\emph{Lyapunov exponent} at $x$ and denote it by $\chi(x)$. We call such $x$ \emph{Lyapunov regular}.
We allow $\chi(x)=-\infty$. Given $\alpha\in\bR$, let
\[
	\cL(\alpha)\eqdef \big\{ x\in K\colon \chi(x)=\alpha\big\},
\]
that is, the set of Lyapunov regular points in $K$ which have exponent $\alpha$. For given numbers $0\le \alpha\le\beta$ we consider also the following
level sets
\[
\cL(\alpha,\beta) \eqdef
 \left\{x\in K\colon \underline\chi(x)=\alpha, \, \overline\chi(x)=\beta\right\}.
\]
If $\alpha<\beta$ then $\cL(\alpha,\beta)$ is contained in the set of so-called \emph{irregular points}
\[
\cL_{\rm irr}\eqdef\left\{ x\in K\colon \underline\chi(x)<\overline\chi(x)\right\}.
\]

Given a compact $f$-invariant set $X\subset K$, we denote by $\cM(X)$ the set of all $f$-invariant Borel probability measures supported on $X$. We denote by $\cM_{\rm E}(X)$ the subset of ergodic measures contained in $\cM(X)$. We simply write $\cM=\cM(X)$ and $\cM_{\rm E}=\cM_{\rm E}(X)$ if $X=K$.

Note that in the following definitions it is irrelevant whether we consider $\mu\in\cM_{\rm E}$ or $\mu\in\cM$.  Given $\mu\in\cM$, we denote by $h_\mu(f)$ the \emph{measure theoretic entropy} of $\mu$ and by
\[
	\chi(\mu)\eqdef\int\log\,\lvert f'\rvert\,d\mu
\]
the \emph{Lyapunov exponent} of $\mu$ (if $\mu\in\cM_{\rm E}$ then for $\mu$-a.e. $x$ it is equal to $\chi(x)$ by Birkhoff Ergodic Theorem).
Notice that for $(f,K)\in \sA$ we have $\chi(\mu)\ge 0$ for any probability $f$-invariant $\mu$ supported
on $K$ (see \cite[Theorem B]{Prz:93}, and \cite[Appendix A]{R-L:} for a simplified proof).

%--------------------------------------------------------------------------------------------------
\subsection{Pressure and Legendre-like transform}\label{Pressure}
%--------------------------------------------------------------------------------------------------

Given $t\in\bR$, we define the function $\varphi_t\colon K\to [-\infty, +\infty]$ by
\begin{equation}\label{varphit}
        \varphi_t(x) \eqdef
         -t\log\,\lvert f'(x)\rvert.
\end{equation}
We define the \emph{variational pressure} of $\varphi_t$ (with respect to $f|_K$)  by
\begin{equation}\label{wirm}
 P_{\rm var}(\varphi_t)\eqdef
      \max_{\mu\in\cM(f)}\left( h_\mu(f)+\int_\Lambda \varphi_t \,d\mu\right)
      =\max_{\mu\in\cM(f)}\Big( h_\mu(f)-t\chi(\mu)\Big).
 \end{equation}
We shall often denote this pressure (and other pressures being equal to it, see \eqref{pressures}) simply by $P(t)$ and call \emph{geometric pressure}.
The function $t\mapsto P(t)$ is convex as the maximum of convex (even affine) functions
$$
t\mapsto
h_\mu(f)-t\,\chi(\mu).
$$
It is monotone decreasing, since for all $\mu\in\cM(f)$ we have $\chi\ge 0$ (see the references above).

Let
 \begin{equation}\label{yu}
	\chi_{\rm inf}\eqdef \inf\{\chi(\mu)\colon\mu\in\cM\},\quad
	\chi_{\rm sup}\eqdef \sup\{\chi(\mu)\colon \mu\in \cM \}.
\end{equation}
The definition of the variational pressure~\eqref{wirm} implies
\[
	\chi_{\rm inf} = \lim_{t\to\infty} - \frac 1 t P_{\rm var}(\varphi_t), \quad
	\chi_{\rm sup} =  \lim_{t\to -\infty} - \frac 1 t P_{\rm var}(\varphi_t).
\]
Observe that $\chi_{\rm inf}$ and $\chi_{\rm sup}$ coincide with $\alpha^-$ and $\alpha^+$, the inclinations of respectively the right and the left asymptotes of $-P_{\rm var}(\varphi_t)$, see~\cite{GelPrzRam:10}.

Given real $\alpha\not=0$, let us denote
\begin{equation}\label{def:Fa}
F(\alpha) \eqdef \frac{1}{\lvert \alpha\rvert} \inf_{t\in\bR}
\left(P(t)+\alpha t \right)
\end{equation}
Observe that $\alpha\mapsto \inf_t(P(t)+\alpha t)$ is a Legendre-like transform of $P(t)$, more precisely $\alpha\mapsto -\inf_t(P(t)-\alpha t)$ is the Legendre transform of the function $P(t)$ (see e.g. \cite{PrzUrb:10,Pes:97}). Note that while the function $\alpha\mapsto \inf_t(P(t)+\alpha t)$ is always concave, the function $\alpha\mapsto F(\alpha)$ is not always concave (see, for example,~\cite{IomKiw:09}). However, on any compact interval it achieves its minimal value at an endpoint. Some of the possible shapes of the graphs of the function $t\mapsto P(t)$ are shown in Figure~\ref{geometric-pressure}, the corresponding graphs of $\alpha\mapsto F(\alpha)$ are shown in Figure~\ref{Fig:Legendretransform}.
Let also
\begin{equation}\label{def:Fb}
F(0) \eqdef \lim_{\alpha\to 0_+} F(\alpha).
\end{equation}
Note that $F\not=-\infty$ precisely on the interval $[\chi_{\inf},\chi_{\sup}]$, where it is non-negative.

Denote
\[
	t_0\eqdef \inf\big\{t\in\bR\colon P_{\rm var}(\varphi_t )= 0\big\}.
\]	
This number is well defined since the latter set is non-empty, see \cite[Proposition 1.19]{PrzRiv:13}.  It coincides with the so-called hyperbolic dimension, that is, the supremum of the Hausdorff dimension of expanding Cantor repellers (see Section~\ref{ss:bridges} and \cite{PrzUrb:10},~\cite{PrzRiv:13}).

\medskip

There are two significantly different behaviors of $f$:

\medskip

\noindent\emph{The case $\chi_{\inf}>0$, Lyapunov hyperbolic condition (LyapHyp).}\footnote{In~\cite{PrzRiv:13} this property is called \emph{Lyapunov}.}
This property implies that $P(t)\to-\infty$ as $t\to\infty$, with the slope bounded away from 0. Note that  for every $(f,K)\in\sA_+^{\BD}$ satisfying the weak isolation condition, LyapHyp  is equivalent to TCE together with the absence of indifferent periodic points (see Section~\ref{ss:TCE} for the definition of TCE and~\cite[Theorem C]{PrzRiv:13}).

\medskip

\noindent\emph{The case $\chi_{\inf}=0$.}
In this case we have $P(t)\equiv 0$ for all $t\ge t_0$. Indeed, by contradiction, assume that $P(t)<0$. Since for all $\mu\in \cM$ by the variational principle we have $A(\mu,t):=h_\mu(f)-t\chi(\mu)\le P(t)$ and $A(\mu,0)=h_{\mu}(f)\ge 0$, we would obtain $\chi(\mu)=(A(\mu,0)-A(\mu,t))/t\ge -P(t)>0$, a contradiction.

\medskip

This distinction is illustrated on Figure~\ref{geometric-pressure}.
For the definition of $t_+$  see Subsection~\ref{ss:phase}.

\medskip

Denote by $\chi^\ast$ the right derivative of $P(t)$ at $t=t_0$. Notice that $F(\alpha)$ attains its maximum at $\chi^\ast$ and note that this maximum is equal to $t_0$. Observe that $F(\alpha)$ is  increasing on $[\chi_{\inf},\chi^\ast]$ and non-increasing
on $[\chi^\ast,\chi_{\sup}]$.
Hence, for every $\alpha\le\beta$ we have
\begin{equation}\label{prop:FFF}
	\min\{F(q)\colon \alpha \le q \le \beta\}
	=\min\{F(\alpha),F(\beta)\}\,.
\end{equation}
In the case  $\chi_{\inf}=0$, the maximum of $F$ is attained at 0, so we have $F(0)=t_0$. If $P(t)$ is not differentiable at $t_0$ then we have $F(\alpha)\equiv t_0$ on the whole interval between the right and the left derivatives of $-P(t)$ at $t_0$.

\begin{figure}
\begin{minipage}[c]{\linewidth}
\centering
\begin{overpic}[scale=.35]{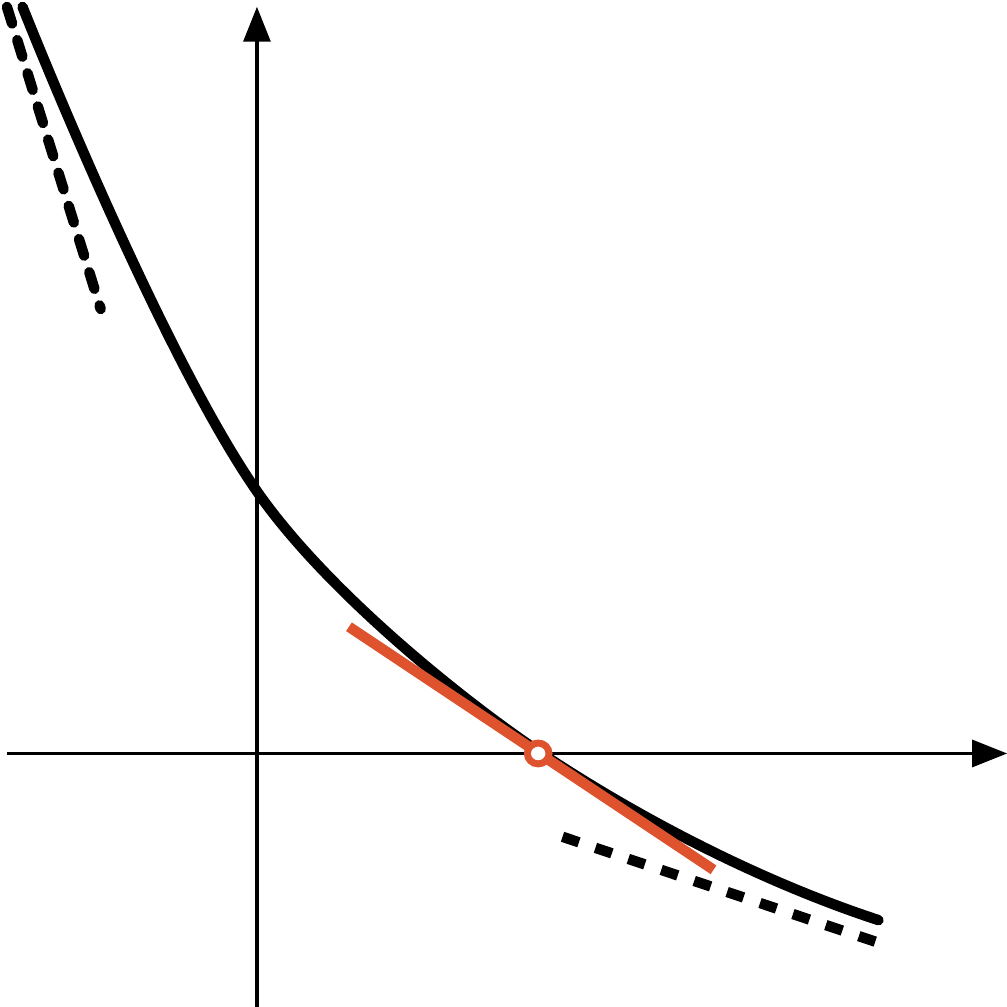}
        \put(102,23){\small$t$}
        \put(52,32){\small$t_0$}
        \put(30,93){\small$P(t)$}
        \put(-15,63){\tiny$-\chi_{\rm sup}$}
        \put(26,28){\tiny$-\chi^\ast$}
        \put(50,4){\tiny$-\chi_{\rm inf}$}
 \end{overpic}
 \hspace{0.5cm}
 \begin{overpic}[scale=.35]{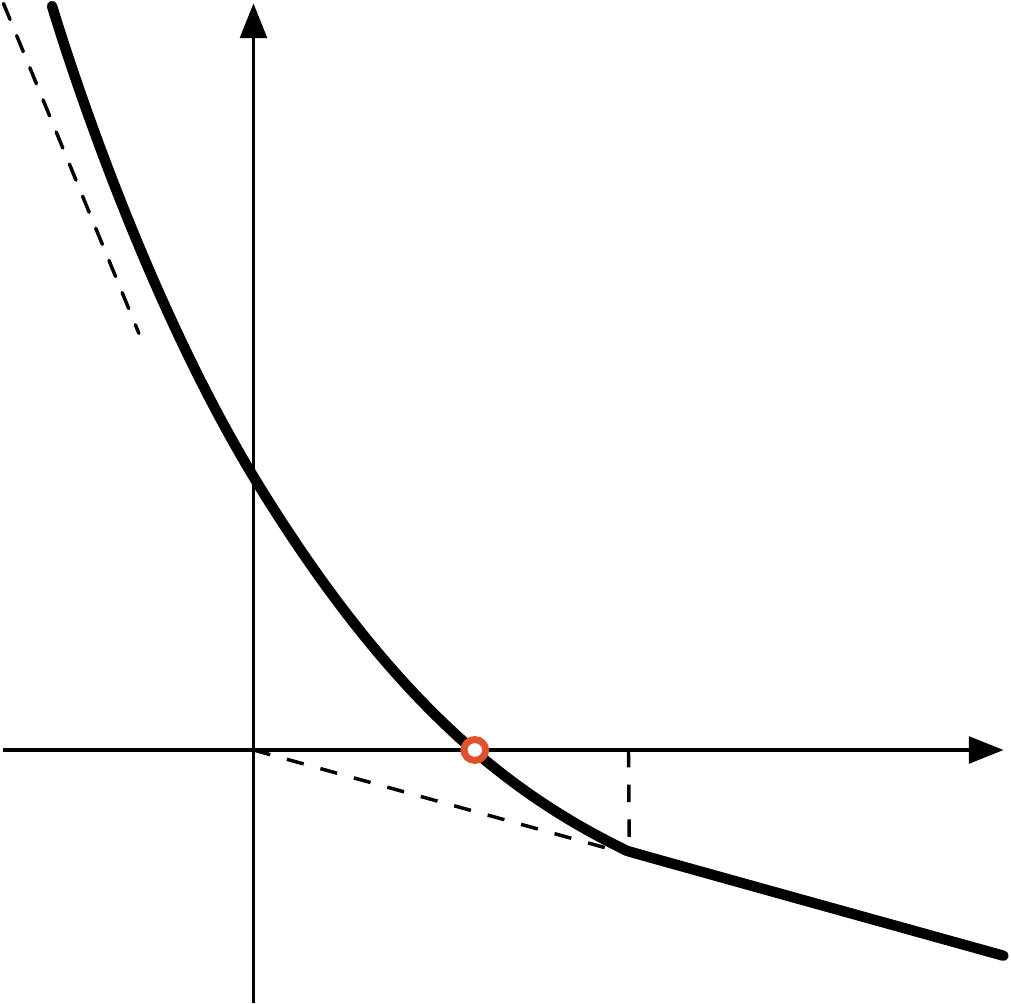}
        \put(102,23){\small$t$}
        \put(44,32){\small$t_0$}
        \put(60,32){\small$t_+$}
        \put(30,93){\small$P(t)$}
        \put(-15,63){\tiny$-\chi_{\rm sup}$}
        \put(60,4){\tiny$-\chi_{\rm inf}$}
\end{overpic}
 \hspace{0.5cm}
\begin{overpic}[scale=.35]{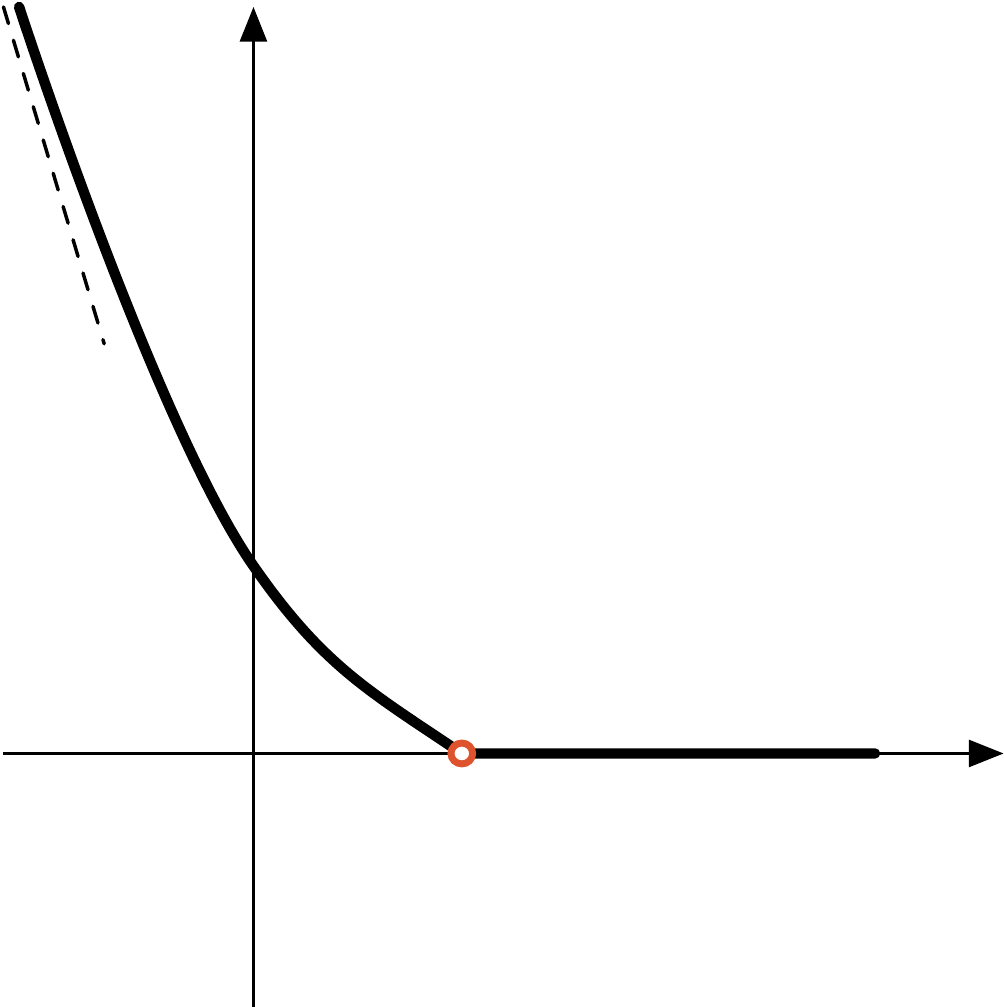}
        \put(102,23){\small$t$}
        \put(44,32){\small$t_0=t_+$}
        \put(30,93){\small$P(t)$}
        \put(-15,63){\tiny$-\chi_{\rm sup}$}
        \put(50,15){\tiny$-\chi_{\rm inf}$}
        \put(50,3){\tiny$=-\chi^\ast=0$}
\end{overpic}
\caption{The geometric pressure: LyapHyp with $t_+=\infty$, LyapHyp with $t_+<\infty$, and non-LyapHyp}
\label{geometric-pressure}
\end{minipage}
\end{figure}

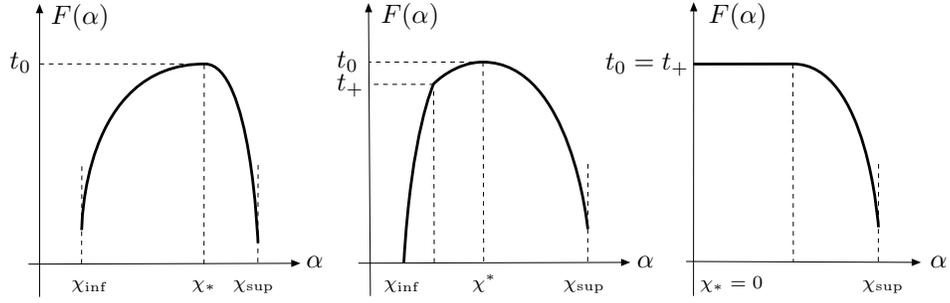
\begin{figure}
\begin{minipage}[c]{\linewidth}
\centering
\begin{overpic}[scale=.35]{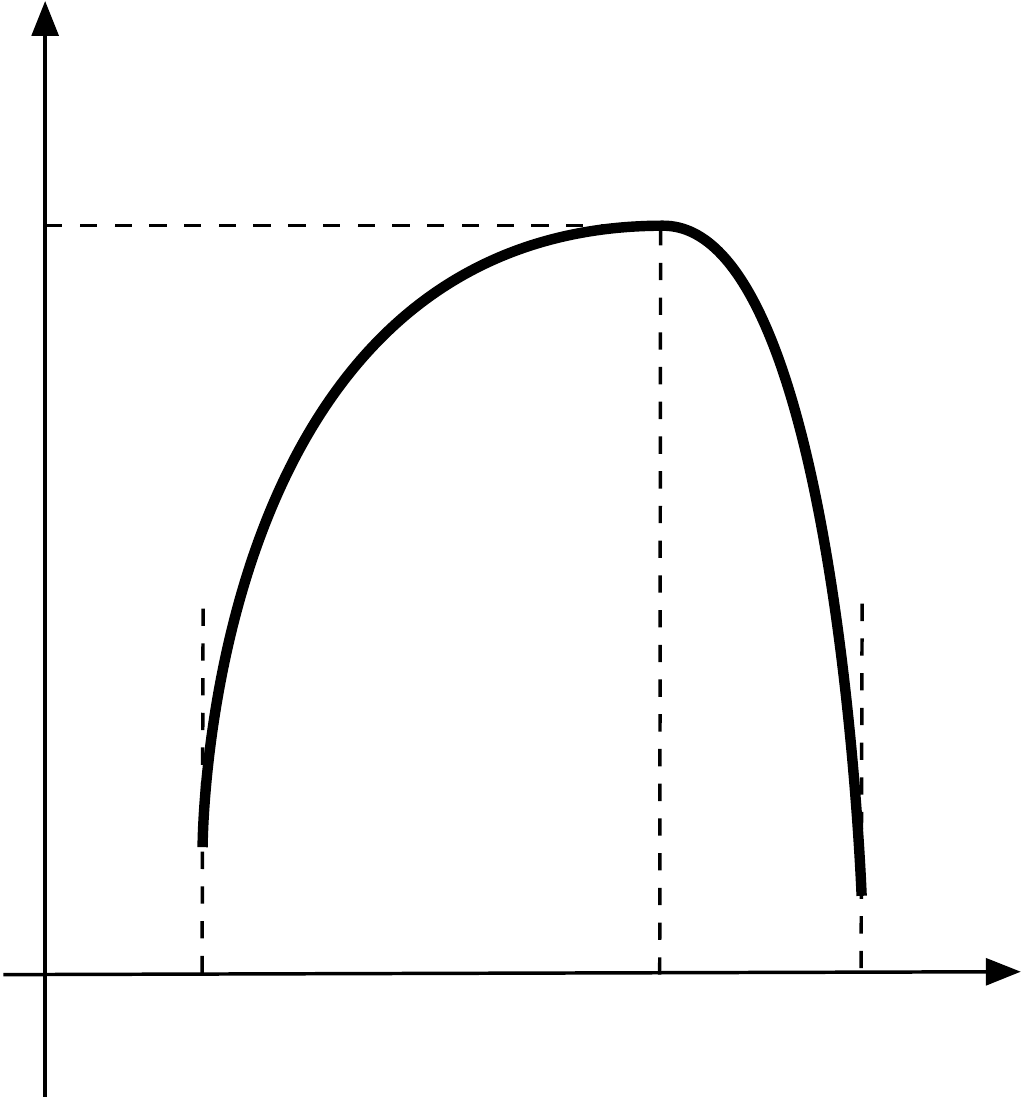}
        \put(8,93){\small$F(\alpha)$}
        \put(95,10){\small$\alpha$}
        \put(-6,78){\small$t_0$}
        \put(15,3){\tiny$\chi_{\rm inf}$}
        \put(56,3){\tiny$\chi_\ast$}
        \put(70,3){\tiny$\chi_{\rm sup}$}
 \end{overpic}
 \hspace{0.5cm}
 \begin{overpic}[scale=.35]{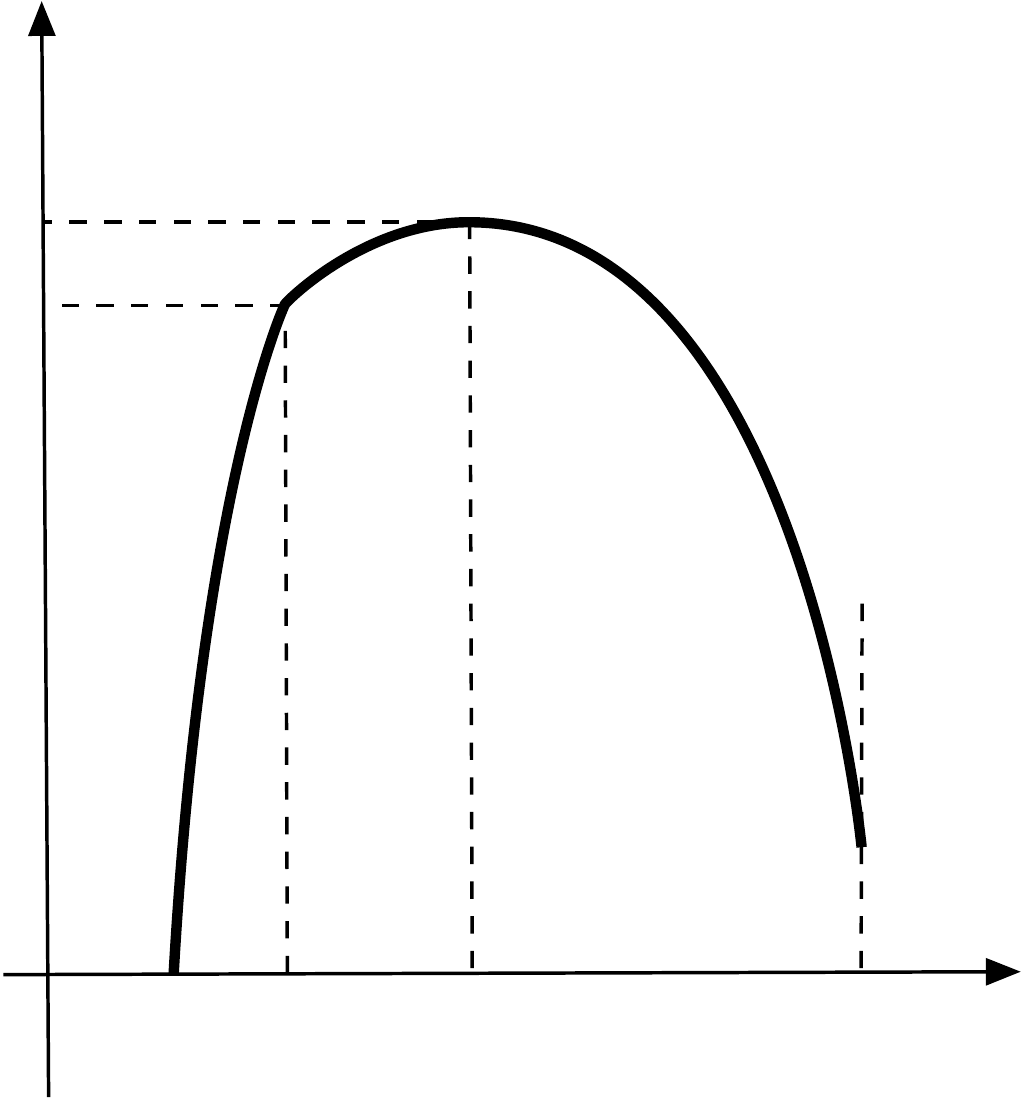}
        \put(8,93){\small$F(\alpha)$}
        \put(95,10){\small$\alpha$}
        \put(-7,78){\small$t_0$}
        \put(-7,70){\small$t_+$}
        \put(9,3){\tiny$\chi_{\rm inf}$}
        \put(39,3){\tiny$\chi^\ast$}
%        \put(25,3){\tiny$\chi^\ast$}
        \put(70,3){\tiny$\chi_{\rm sup}$}
\end{overpic}
 \hspace{0.5cm}
 \begin{overpic}[scale=.35]{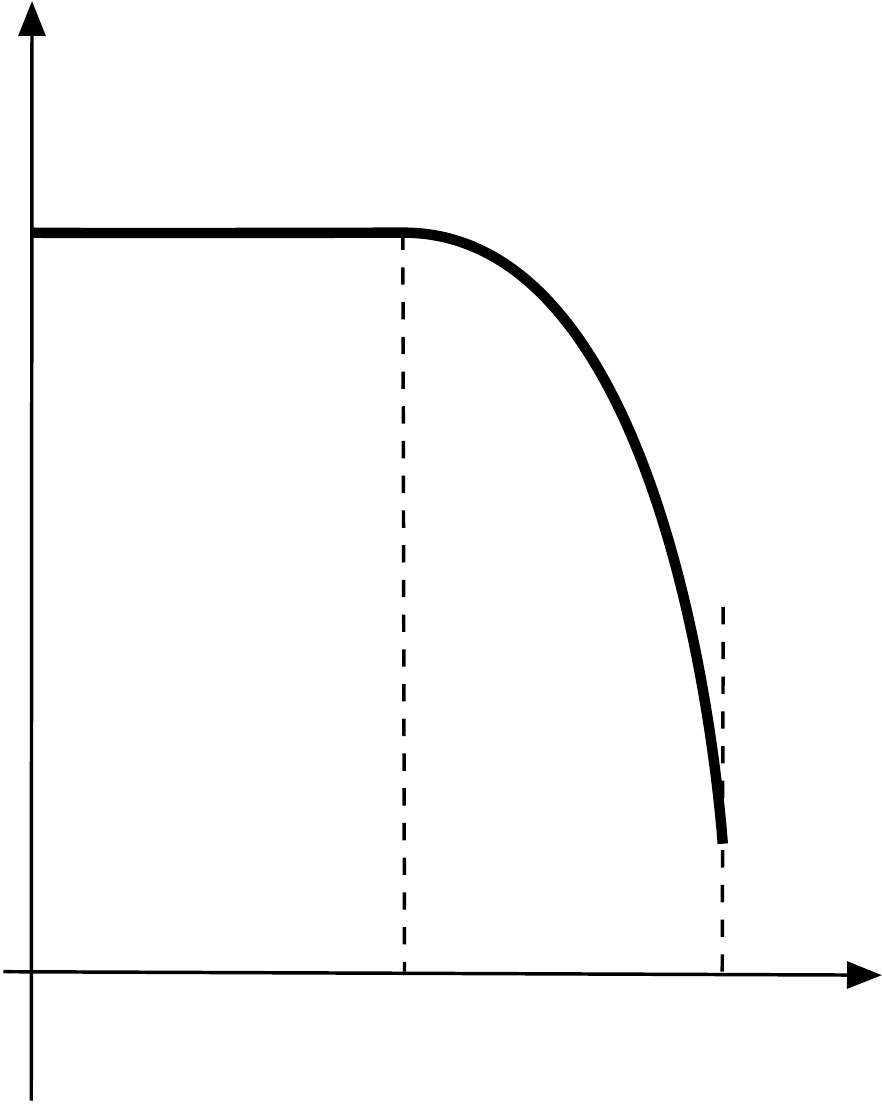}
        \put(-27,77){\small$t_0=t_+$}
        \put(8,93){\small$F(\alpha)$}
        \put(83,10){\small$\alpha$}
        \put(5,3){\tiny$\chi_\ast=0$}
        \put(60,3){\tiny$\chi_{\rm sup}$}
\end{overpic}
\caption{$F(\alpha)$: LyapHyp with $t_+=\infty$, LyapHyp with $t_+<\infty$, and non-LyapHyp}
\label{Fig:Legendretransform}
\end{minipage}
\end{figure}

%--------------------------------------------------------------------------------------------------
\subsection{Main results}\label{main}
%--------------------------------------------------------------------------------------------------
	
\begin{theorem} \label{main2}
Let $(f,K)\in \sA^{\BD}_+$ satisfy the weak isolation condition and be non-exceptional.
For any $\alpha\le \beta \le \chi_{\sup}$ with $\beta>0$, and additionally assuming $\alpha>0$ if $\chi_{\inf}=0$,  we have

\begin{equation}\label{upper-weak}
\min\{F(\alpha), F(\beta)\}\le \dim_{\rm H} \cL(\alpha, \beta)\le
\max\big\{0,\max_{\alpha\leq q \leq \beta}F(q)\big\} .
\end{equation}
In particular, for any $\alpha\in [\chi_{\inf},\chi_{\sup}] \setminus \{0\}$ we have
\[
\dim_{\rm H} \cL(\alpha) = F(\alpha).
\]
We have also
\[
\dim_{\rm H} \cL(0) \geq F(0)\,.
\]
Moreover,
\begin{equation}\label{empty1}
\left\{ x\in K\colon -\infty<\chi(x)<\chi_{\inf}\right\}
 =\emptyset,
 \end{equation}
 \begin{equation}\label{empty2}
 \left\{ x\in K\colon \overline{\chi}(x)>\chi_{\sup}\right\}
 = \emptyset,
\end{equation}
and
\begin{equation}\label{dim0}
\dim_{\rm H}\left\{ x\in K\colon 0<\overline{\chi}(x)<\chi_{\inf}\right\} = 0.
\end{equation}
\end{theorem}

The assumption of the above theorem are, for example, satisfied if $(f,K) \in\sA^3_+$ and there are no indifferent periodic orbits (see Remark~\ref{rem:boudeddistorionremark}).

\begin{remark}
Theorem \ref{main2} has the same assertion as \cite[Theorem 2]{GelPrzRam:10}
(except that we do not assume $\alpha>0$ which is in fact not needed neither here nor there).
The strategy of the proof is also the same. However, there are technical differences, and the part which has an entirely new proof, since the holomorphic one does not work, is to show that $\{ x\in K\colon 0\le \chi(x)<\chi_{\inf}\}=\emptyset$ in \eqref{empty1}.
This answers a conjecture in \cite[p.4]{IomTod:13} and \cite[Remark 4.1]{IomTod:11}.
In Theorem~\ref{strangestrange} below, we prove an even a stronger fact under the additional assumptions that $K=\hat I_K$ and that $K$ contains no indifferent periodic orbits.
\end{remark}

Let us mention that the above results extend~\cite{IomTod:11} in the following sense: our results enable us to cover also the boundary points of the spectrum $\cL(\alpha)$ for $\alpha=\chi_{\rm inf}$ and $\chi_{\rm sup}$, as well as the irregular points $\cL(\alpha,\beta)$ for $\alpha<\beta$. Moreover, additionally we are able to consider maps that have indifferent periodic orbits.

Hofbauer~\cite{Hof:10} proves results about the Lyapunov spectrum in a situation which is in some aspects considerably more general than the one discussed in the present paper (he considers piecewise monotone piecewise continuous interval maps), however which does not allow singularities $\pm \infty$ for the geometric potential.

Similarly to the complex case we do not know if $\{x\in K: \chi(x)=-\infty\}=\bigcup_{n\ge 0} f^{-n}(\Crit(f))$, unless there is at most one critical point in $K$. In this case the proof of the equality is the same as in the complex situation.

\medskip

Under additional assumptions we can improve (\ref{upper-weak}), the right hand side inequality, sometimes getting even the equality.

\smallskip

\begin{theorem}\label{main3}
Let $(f,K)\in \sA^{\BD}_+$ (or $\sA^3_+$) satisfy the weak isolation condition and be non-exceptional and without indifferent  periodic orbits.
Then given numbers  $\alpha\le \beta \le \chi_{\sup}$ with $\beta>0$, we have
	\begin{equation}\label{eq:main3ineq}
		\dim_{\rm H}\cL(\alpha,\beta) \le \max\{0,F(\beta)\}.
	\end{equation}
If we assume $K=\hat I_K$ (i.e. $K$ is the finite union of intervals), then
\begin{equation}\label{eq:main3ineq2}
\dim_{\rm H}\cL(\alpha,\beta) \le
\max\big\{0,\min\{F(\alpha^\sharp),F(\beta)\}\big\},
\end{equation}
where
\[
	\alpha^\sharp:=\frac{\beta}{1+(\beta-\alpha)/\chi_{\sup}}\,.
\]	
\end{theorem}

Note that $\alpha^\sharp$ is positive and $\alpha < \alpha^\sharp < \beta$ (except for the case $\alpha=\beta$, when $\alpha^\sharp = \alpha = \beta$).

\begin{remark}
If $F(\alpha)\ge F(\beta)$ then, together with Theorem~\ref{main2} we obtain the  equality
\begin{equation}\label{eq:main3ineqeasy}
\dim_{\rm H}\cL(\alpha,\beta)= F(\beta)\,.
\end{equation}
In particular, we obtain the equality \eqref{eq:main3ineqeasy} for non-LyapHyp systems $(f,K)$ (systems where $\chi_{\inf}=0$).

For uniformly hyperbolic systems for all $\chi_{\inf}\le\alpha\le\beta\le \chi_{\sup}$ the equality
\[
	\dim_{\rm H}\cL(\alpha,\beta)=\min \{F(\alpha) , F(\beta)\}
\]	
holds.
The question whether this equality holds in general remains open.
\end{remark}

The trouble with improving in the right hand estimate \eqref{upper-weak} the maximum by the minimum, was that the  times one can go to large scale avoiding $S'$ for a conical point $x$ (see Subsection~\ref{ss:conical points}) are not related to the times of an {\it a priori} given Lyapunov exponent between $\alpha$ and $\beta$. The new idea is to consider a larger set of (restricted Pliss hyperbolic) times where one can go to large scale with finite criticality, thus allowing us to control the exponent.

We deduce also the following strong version of a part of the property (\ref{empty1}) of Theorem~\ref{main2}.

\begin{theorem}\label{strangestrange}
Let $(f,K)\in \sA_+^{\BD}$ (or $\sA_+^3$) satisfying $K=\hat I_K$ and being non-exceptional and without indifferent periodic  orbits. Then,
given numbers $\alpha\le\beta$ with  $\alpha^\sharp<\chi_{\inf}$ and $\beta>0$, the set $\cL(\alpha,\beta)$ is empty.
\end{theorem}

\begin{remark}
The proof of the fact that
$
\left\{ x\in K\colon -\infty<\chi(x)<\chi_{\inf}\right\}=\emptyset
$
(observe that this set concerns Lyapunov regular points only) in Theorem~\ref{main2}, provided in Section~\ref{s:completeness}, is conceptually simpler: it does not use the existence of conformal measures, see the end of Subsection~\ref{ss:conformal} for this notation.

The existence of true conformal measures is the only reason for which we have to assume no indifferent periodic orbits in Theorems~\ref{main3} and~\ref{strangestrange}.
\end{remark}

Theorem~\ref{main2} is proven in Section~\ref{sectionofthetheorem}, except for relation~\eqref{empty1} which is shown in Section~\ref{s:completeness}. Theorems~\ref{main3} and~\ref{strangestrange} are proven in Section~\ref{strong}.

%--------------------------------------------------------------------------------------------------
\section{Tools}\label{s:tools}
%--------------------------------------------------------------------------------------------------
\subsection{Pressure functions}\label{ss:pressure}
%--------------------------------------------------------------------------------------------------

Recall that $f$ is said to be \emph{uniformly expanding} or \emph{uniformly hyperbolic} on a set $X\subset I$ if there exists $\lambda>1$ such that for every $n\ge 1$ we have $\lvert (f^n)'\rvert\ge \lambda^n$ for every $x\in X$.
A compact set $X\subset \bR$ is said to be \emph{isolated} if there is an open neighborhood $U$ of $X$ in the domain of $f$ such that $f^n(x)\in U$ for  all $n\ge 0$ implies $x\in X$.
A set $X\subset I$ is said to be \emph{invariant} if $f(X)\subset X$.
A compact $f$-invariant isolated uniformly expanding set $X\subset I$ is called  \emph{expanding repeller}.

Given $X\subset I$ invariant and $\varphi\colon X\to\bR$ continuous, we denote by $P_{f|X}(\varphi)$ the standard \emph{topological pressure} of $\varphi$ with respect to $f|_X$ (see, for example,~\cite{Wal:81} or~\cite{PrzUrb:10}). In the case $X=K$ we simply write $P(\varphi)$.
Furthermore, we denote by $P_{\rm hyp}(\varphi)$ the \emph{hyperbolic pressure} defined by
\[
	P_{\rm hyp}(\varphi)\eqdef \sup_{X\subset K}P_{f|X}(\varphi),
\]
where the supremum is taken over all expanding repellers $X\subset K$.

By~\cite[Theorem B]{PrzRiv:13}, for every $(f,K)\in \sA_+^{\BD}$ and $t\in\bR$ we have
\begin{equation}\label{pressures}
	P(t)=P(\varphi_t)\eqdef
	P_{\rm var}(\varphi_t)  =
	P_{\rm varhyp}(\varphi_t) =  P_{\rm hyp}(\varphi_t).
\end{equation}

The proof of the equality $P_{\rm var}(\varphi_t)  =
	P_{\rm varhyp}(\varphi_t)$ is not easy in the interval case and follows from the equality of various
definitions of Topological Collet-Eckmann maps, see \cite{R-L:} or \cite[Lemma 4.6 and Theorem C]{PrzRiv:13}. The inequality $P_{\rm varhyp}(\varphi_t) \ge P_{\rm hyp}(\varphi_t)$ follows immediately from the variational principle.
The proof of $P_{\rm varhyp}(\varphi_t) \le P_{\rm hyp}(\varphi_t)$ will be outlined in Section~\ref{Preliminary Constructions} for completeness, and to introduce another variant of
$P_{\rm hyp}(\varphi_t)$ (on Cantor sets).

%---------------------------------------------------------------------------------------------------
\subsection{Phase transitions}\label{ss:phase}
%---------------------------------------------------------------------------------------------------

Let $(f,K)\in\cA_+^{\rm BD}$. Let	
\[
t_+\eqdef
\sup\big\{t\in\bR\colon P(\varphi_t)+t\,\chi_{\rm inf}>0 \big\}.
\]

Observe that since $P(\varphi_0)=h_{\rm top}(f|_K)>0$, we have $t_+>0$. Note that both cases $t_+=\infty$ as well as  $t_+<\infty$ are possible. The latter case occurs for example
in the non-LyapHyp case, that is, if $\chi_{\inf}=0$. see Section~\ref{ss:TCE} and references therein. Observe that the case $t_+<\infty$ can happen even if  $\chi_{\inf}>0$ (and $P$ can be differentiable or nondifferentiable at $t_+$),  see \cite{CorRiv:13}.

Let also
\[t_-\eqdef
	\inf\big\{t\in\bR\colon P(\varphi_t)+t\,\chi_{\rm sup}>0 \big\}.
\]
Observe that $t_-<0$. There are examples in which $t_->-\infty$ (for example, $f(x)=x^2-2$), but this can happen only for  $f$ exceptional. This
was proved in the interval setting under some restrictions,  see \cite[Theorem B]{IomTod:10}.

That in our situation $t_-=-\infty$ for non-exceptional $f$ can be seen from the following argument.
Suppose to the contrary that $P(\varphi_t)=-t\chi_{\sup}$ for some $t<0$.
Then this holds for every $t'<t$ by the definition of $P$.
By the definition of $P$ there exists a sequence  $\mu_n\in \cM$ weakly* convergent to some $\mu_{t'}\in\cM$ such that  $-t'\chi(\mu_n)+h_{\mu_n}(f)\to P(\varphi_{t'})$, hence $-t'\chi(\mu_n)\to P(\varphi_{t'})$, as the entropies must be equal to 0.
Then by the upper semi-continuity of the function $\mu\mapsto \chi(\mu)$  on $\cM(f)$ we get
$-t'\chi(\mu_{t'})=P(t')$ and hence $\chi(\mu_{t'})=\chi_{\sup}$ (in particular $\mu_{t'}$ is na equilibrium state). Now by the Key Lemma in \cite{InoRiv:12}\footnote{$t>0$ is assumed there, but J.Rivera-Letelier has informed us that under small modifications the proof works also for $t<0$} adapted to $(f,K)\in\sA^{\BD}_+$, since we assume $(f,K)$ non-exceptional, we get $P(\varphi_{t'})>-t'\chi(\mu_{t'})$, contradiction.

Figure~\ref{geometric-pressure} illustrates various cases. In all three cases the asymptotes $-\chi_{\sup}$ are depicted. In the first case (LyapHyp with $t_+=\infty$)  the asymptote $-\chi_{\inf}$ does not always pass through $0$, it could intersect the vertical axis at a positive number. This corresponds to the case when $F(\chi_{\sup})$ and $F(\chi_{\inf})$ are positive. See \cite{Sch:99} for uniformly hyperbolic examples with H\"older potentials.

%-----------------------------------------------------------------------------------------------------
\subsection{Conformal measures}\label{ss:conformal}
%-----------------------------------------------------------------------------------------------------

Given a function $\phi\colon K\to\bR$, we say that a finite Borel measure $\mu$ supported on $K$ is $\phi$-\emph{conformal measure} if it is forward quasi-invariant, that is, sets of measure zero are sent to sets of measure zero, and  satisfies
\begin{equation}\label{def:conff}
	\mu(f(A)) = \int_A\phi\,d\mu
\end{equation}
for every $A\subset K$ on which $f$ is injective.
Conformal measures not always exist for $t<0$, in particular in the presence of critical points in $K$. {It is at the moment unclear whether conformal measures exist for $t>0$ or $t<0$ in the exceptional case. However the following weaker version will be sufficient for our setting.

We call $\mu$ \emph{conformal away from the restricted singular set $S'$} or a \emph{CaS measure}
 if~\eqref{def:conff} is satisfied under the additional assumption that $A$ is disjoint from $S'$. As usually we call $\phi$ the \emph{Jacobian} of $f$ for $\mu$.

\begin{proposition} If $f$ is non-exceptional, then
for every $t<t_+$ there exists a CaS measure with Jacobian $\lambda(t)\lvert f'\rvert^t$ for
$\log\lambda(t)=P(t)$ which is positive on open sets. We will denoted this measure by $\mu^\ast_t$.
\end{proposition}

\begin{proof}
We construct $\mu^*_t$ by Patterson-Sullivan's method, as in \cite[pages 333--334]{PrzUrb:10} (see also \cite{PrzRivSmi:04} for the $t<0$ complex case).

For
$t>0$ it is a conformal measure in the complex case, but \eqref{def:conff} can fail at points in $S'$ in the interval case,
see \cite[Appendix C]{PrzRiv:13} for a more precise description. It is proved
in \cite[Lemma C.3]{PrzRiv:13} that for all $t\in\bR$ if $\mu^*_t$ is 0 on an open set, then it is supported on a weakly $S'$- exceptional set $E$ (which is finite, as mentioned in Subsection~\ref{exceptional}). In the case that $f$ is non-exceptional this leads to contradiction.

\smallskip

By Patterson-Sullivan's construction (see
\cite[(C.2)]{PrzRiv:13}) for every $x\in K$ and every $z\in f^{-1}(x)\cap K$
we have
\begin{equation}\label{PS}
\mu^*_t(z) |f'(z)|^t  \le \mu^*_t(x)\le \sum_{y\in f^{-1}(x)\cap K} \mu^*_t(y) |f'(y)|^t,
\end{equation}
provided no expression of the form $0\cdot \infty $ is involved.

We consider three cases:
\smallskip

\noindent\textbf{Case $t<0$.}
If $\mu^*_t(x)>0$ at $x\in E$  then we have atoms of $\mu^*_t$ at all points $f^n(x)$ for $n\ge 0$ due to the left hand side inequality \eqref{PS}. In particular, since $f$ is non-exceptional, $\mu^*_t$ has an atom at some $w=f^n(x)$ not weakly $S'$-exceptional (this will be proved in Lemma~\ref{lem:denseback}, item 2 below). This contradicts the condition that $\mu^*_t$ is supported on $E$.
\smallskip

\noindent\textbf{Case $t>0$.}
If $\mu^*_t$ has an atom at $x$ then it has an atom at a point $y\in f^{-1}(x)$ by the right hand side inequality
\eqref{PS}.  Therefore it has atoms at all $f^k$-preimages of $x$ which constitute an infinite set, which is therefore not contained in $E$.
We have a contradiction again.
\smallskip

\noindent\textbf{Case $t=0$.}  We can consider a measure with maximal entropy as $\mu^*_0$.
\end{proof}

In the proof of Theorem~\ref{main3}, in Section~\ref{strong} we rely on a part of \cite[Theorem A]{PrzRiv:13}, asserting that for
$(f,K)\in \sA^{\BD}_+$ without indifferent periodic orbits, then there exists a probability measure $\mu_t$ which is truly conformal for the function
$\varphi=e^{P(t)}|f'|^t$ for every $t<t_+$,  non-atomic, positive on open sets,
supported on the set of conical points in $K$, see Subsection~\ref{ss:conical points}.

%------------------------------------------------------------------------------------------
\subsection{Topological Collet-Eckmann and related notions}\label{ss:TCE}
%------------------------------------------------------------------------------------------

In this section we will discuss some (partially equivalent) conditions for $(f,K)\in\sA$.

Given $y=f^n(x)$, denote by $\Comp_x f^{-n}(B)$ the component of the preimage of an interval $B\ni y$, which contains $x$. We sometimes call this component a \emph{pull-back} of $B$ by $f^n$. For $(f,K)\in \sA$ we consider only pull-backs intersecting $K$.
This component is an interval, to denote its length we use either $\lvert\cdot\rvert$ or $\diam$ (to avoid confusion with other meanings of $\lvert\cdot\rvert$).
\medskip

\noindent
\emph{Lyapunov Hyperbolic Condition (LyapHyp)}: $\displaystyle \chi_{\inf}>0$.
\medskip

\noindent
\emph{Topological Collet-Eckmann Condition (TCE)}:
There exist numbers $M \ge 0$, $P \ge 1$, and $r>0$ such that for every
$x\in K$ there exists a strictly increasing sequence of positive integers
$(n_j)_{j\ge1}$ such that for every $j\ge 1$ we have $n_j \le P \cdot j$ and
\begin{multline}\label{TCE}
        \#\big\{k\in\{0,\ldots,n_j-1\}\colon  \\
       \Comp_{f^k(x)}  f^{-(n_j-k)}B(f^{n_j}(x),r)
                \cap\Crit(f) \not= \emptyset \big\} \le M.
\end{multline}
This condition is considered under the absence of indifferent periodic orbits; then all components above are well insider $U_K$ for $r$ sufficiently small (see BaShrink below).
Of course the number $r$ can be taken arbitrarily small and $P$ can be taken arbitrarily close to 1 (at the cost of increasing $M$).
\medskip

\noindent
\emph{Exponential shrinking of components (ExpShrink)}:
        There exist numbers $\lambda>1$ and $r>0$  such that for every $x \in K$,
every $n > 0$, and every connected component $W$ of $f^{-n}(B(x, r))$ which intersects $K$ we have
$$
        \lvert W\rvert  \le \lambda^{-n}\,.
$$

See~\cite{PrzRiv:13} for the usage of the condition LyapHyp in the multimodal case and for example~\cite{PrzRivSmi:03} in the complex holomorphic setting.
In the interval case the proof of the fact that, assumed weak isolation, LyapHyp implies TCE is different from the holomorphic case and is given in the multimodal case in~\cite{R-L:} and later in \cite[Theorem C]{PrzRiv:13} (in the more general case of $f\in\sA_+^{\BD}$).

By~\cite[Theorem C]{PrzRiv:13}, for every $(f,K)\in\sA_+^{\BD}$ satisfying the weak isolation condition, LyapHyp  is equivalent to TCE together with the absence of indifferent periodic points which in turn is equivalent to ExpShrink.
For more conditions and their equivalences in the interval case, see \cite[Theorem C]{PrzRiv:13} or
\cite{R-L:}.

\medskip

In the sequel, instead of ExpShrink, we will use  only the following weaker  condition which is satisfied for every $(f,K)\in\sA_+$ (see~\cite[Lemma 2.10]{PrzRiv:13}).

\medskip

\noindent\emph{Backward Shrinking (BaShrink)}:
	For every $\varepsilon>0$ there exists $\delta>0$ such that  if $T$ is an open interval in $\R$ intersecting $K$, disjoint from $\Indiff(f)$, and satisfying  $\lvert T\rvert \le \delta$, then for every $n\ge 0$ and every component $T'$ of $f^{-n}(T)$ intersecting $K$ we have $\lvert T'\rvert\le \varepsilon$. Moreover, the lengths of all components of $f^{-n}(T)$ intersecting $K$ converge to 0 uniformly as $n\to\infty$.

In particular for $\delta$ sufficiently small, for every $\varepsilon'>0$ there exists a positive integer $\varkappa=\varkappa(\delta,\varepsilon')$ such that for every $T$ as above, satisfying $|T|\le \delta$, and for every component $T'$ of $f^{-n}(T)$ intersecting $K$ with $n\ge \varkappa$, we have $\lvert T'\rvert\le \varepsilon'$.

%-----------------------------------------------------------------------------------------------------
\section{Preliminary constructions}\label{Preliminary Constructions}
%-----------------------------------------------------------------------------------------------------

%----------------------------------------------------------------------------------------------------
\subsection{Outside singular set}\label{outside}
%----------------------------------------------------------------------------------------------------

 \begin{lemma}\label{l:K-homeo}
	 For $(f,K)\in\sA$ there exists $\delta>0$ such that for every interval $T\subset U_K$ intersecting $K$ and disjoint from $S'$, if $|T|<\delta$, then $f\colon T\to f(T)$ is a homeomorphism and
 \begin{equation}\label{K-homeo}
 f(T)\cap K = f(T\cap K).
 \end{equation}
 \end{lemma}

 In \cite{PrzRiv:13} such $f$ on $T$ is called {\it $K$-homeomorphism}.
 For completeness we provide the proof of this lemma.

 \begin{proof}
For $x\in\partial \hat I_K$ (there are only finitely many such points) consider the point $f(x)$. There are two possibilities: either $f(x)$ is accumulated by points in $K$ from both sides, or $f(x)$ is a boundary point of $K$. In the latter case we denote by $g_x$ the length of the component of $\bR\setminus K$ adjacent to $f(x)$.
Let $U'$ be a neighborhood of $K$ with closure in $U_K$.   Let $L$ be Lipschitz constant for $f|_{U'}$. Choose
\[
	\delta=\min\Big\{{\rm dist} (\partial U_K, \partial U'),\frac1L\max_xg_x\Big\}\,.
\]

Consider now an open interval $T$ of length shorter than $\delta$ which intersects $K$ and is disjoint from $S'$. Since $K$ is forward invariant, the only possibility that \eqref{K-homeo} fails is when
$f(T)\cap K\setminus f(T\cap K)\ne\emptyset$. In such a situation $T$ must contain a point $x\in\partial \hat I_K$.
Indeed, $T$ intersects $\hat I_K$ because it intersects $K$, and if $T\subset \hat I_K$ then \eqref{K-homeo} holds since $f^{-1}(K\cap f(T))\subset K$, by maximality of $K$.
Let $T'$ be the component of $T\setminus \{x\}$ which is disjoint from $\hat I_K$. If $f(x)$ is accumulated by points in $K\cap f(T')$, then $x\in NO(f,K)$ and hence $x\in S'$, a contradiction. Otherwise, $f(x)$ is a boundary point of $K$ and, by our choice of $\delta$, we have $f(T')\cap K=\emptyset$. Since $T'':=T\setminus T'\subset \hat I_K$, maximality of $K$ implies property
\eqref{K-homeo} for $T''$. Hence \eqref{K-homeo} holds for $T$.
\end{proof}

%--------------------------------------------------------------------------------------------------
\subsection{More on $K$-homeomorphisms}
%--------------------------------------------------------------------------------------------------

\begin{lemma}\label{no-truncation}
Let $(f,K)\in \sA_+$ be non-exceptional, satisfying weak isolation condition.
Consider any $x\in K$ not weakly $S'$-exceptional. Then for a constant $r>0$ small enough depending on $(f,K)$ and  every integer $n$ large enough the following holds.
Assume that $B(f^n(x),r/2)$ does not contain any indifferent periodic point.
Let $W_n$ be the pull-back of $B(f^n(x),r/2)$ for $f^n$ which contains $x$.
For all $j=0,\ldots,n$ denote $W_n^j:=f^{j}(W_n)$. Then for all $j=0,\ldots,n-1$ the interior of  $W_n^j$ is disjoint from  $NO(f,K)\setminus\Crit_T(f)$.
\end{lemma}

%For example, if $B(f^n(x),r/2)$ does not contain indifferent periodic orbits then the hypotheses is satisfied due to BaShrink.

\begin{proof}
One condition we impose on $r$ is that $T$ not containing any indifferent periodic point with $\lvert T\rvert\le r$ satisfies BaShrink.

Choose an arbitrary periodic point $p$ whose periodic orbit $\cO(p)$ consists of points not weakly $S'$-exceptional. This is possible since the set of weakly $S'$-exceptional points is finite, see Subsection~\ref{exceptional}, and the set of all periodic points in $K$ is infinite, see \cite[Proposition 2.5]{PrzRiv:13}.

Suppose there exists $j_0\in\{0,...,n-1\}$
and  $z\in \interior W_n^{j_0}$ not being a critical turning point, at which $f|_K$ is not open. We assume that $r$, hence $|W_n^{j_0}|$, are small enough, (smaller than the minimal mutual distance between the points in $NO(f,K)$) that there are no turning critical points in $W_n^{j_0}$. Then there exists $z'\in W_n^{j_0+1}$ arbitrarily close to $f(z)$ whose $f$-preimage $z''$ in $W_n^{j_0}$ is not in $K$.
Let $z_0=x,z_{1},...$ be a backward trajectory of $x$ in $K$,  (i.e. $f(z_j)=z_{j-1}$), converging to $\cO(p)$ omitting $S'$ (possible since we have assumed $x$ is not weakly $S'$-exceptional)
and let $p=y_0,y_{1},...$ be a backward trajectory of $p$ in $K$ such that $z'$ is its accumulation point.
Choose $n_0$ such that $y_{n_0}$ is close enough to $z'$ to be in $W_n^{j_0+1}$. In fact we want $|y_{n_0}-z'|<\e$ for an arbitrarily small $\e>0$.

Choose  $\delta>0$ small enough that all consecutive pull-backs $V_j$ for $f$ of $B(p,\delta)$ along $y_{j},j=0,1,2,..., n_0$ are disjoint from $S'$ and $V_{n_0}$ is also close enough to $z'$ to be in $W_n^{j_0+1}$ and even within an arbitrarily chosen neighbourhood, say in $B(z',\e)$.

If $n$
is large enough, then $|W_n|$ is so small that all consecutive pull-backs $Z_j$ of $W_n$ for $f$ along $z_j,j=0,-1,-2,...$ are disjoint from $S'$. Choose $n_1$ such that $Z_{n_1}\subset B(p,\delta)$.

Let summarize:
$f^{j_0+1}(W_n)=W_n^{j_0+1}$,\
$V_{n_0}\subset W_n^{j_0+1}$,\
$f^{n_0}(V_{n_0})=B(p,\delta)$,\
$Z_{n_1}\subset B(p,\delta)$,\
and $f^{n_1}(Z_{n_1})=W_n$.

 Hence there exists a periodic orbit $O$, of period $j_0+1 + n_0 + n_1$ close to $K$ and passing arbitrarily close to $z''$ for $\e$ small. Hence by the weak isolation assumption $\cO\subset K$ hence $z''\in K$, contradiction.
\end{proof}

The considerations on shadowing as above can be found for other aims in \cite[Subsection 7.3, item 1a]{PrzRiv:13}. It is also related to the `bridges' technique in Subsection~\ref{ss:bridges}.

\begin{corollary}\label{cor:no-truncation}
In the situation of Lemma~\ref{no-truncation}, in particular for $x$ not weakly $S'$-exceptional, for $r$ small enough and all $n$ large enough, $f$ is a $K$-homeomorphism on each
$\interior W_n^j$ not containing any turning critical point, onto $\interior W_n^{j+1}$.
If there is a turning critical
point  $c\in \interior W_n^j$, then $f$ is a $K$-homeomorphism on both components of $\interior W_n^j\setminus\{c\}$, one of them is mapped {onto} $\interior W_n^{j+1}$ and the other one is mapped {into}
$\interior W_n^{j+1}$.
\end{corollary}

\begin{proof}
 By Lemma~\ref{no-truncation} if $\interior W_n^j$ contains
 $z\in NO(f,K)$, then $z\in \Crit_T(f)$.
 So if there is no turning critical point in $\interior W_n^j$ there is no point belonging to $NO(f,K)$ neither. Hence $f$ is a homeomorphism on $\interior W_n^j$, and  also a $K$-homeomorphism, since $\interior W_n^j$ contains $f^j(x)\in K$, by Lemma~\ref{l:K-homeo}.

 If there is a turning critical point   $c\in \interior W_n^j$, then there is no other turning critical point
  in $\interior W_n^j$ since the interval is short enough, and $f(c)$ is an end of $f(W_n^j)$ and at least one of the ends of $W_n^j$ is mapped to the other end of $f(W_n^j)$. Since there are no points belonging to $NO(f,K)\setminus\{c\}$ in $W_n^j$ if this interval is short enough, the $K$-homeomorphism property on each component of $\interior W_n^j\setminus\{c\}$ asserted in the corollary, follows.
\end{proof}

%---------------------------------------------------------------------------------------------------
\subsection{More on topological recurrence and exceptional sets}\label{toprec}
%--------------------------------------------------------------------------------------------------

The fact that $f\colon K \to K$ is topologically transitive and has positive topological entropy implies that it is weakly topologically exact (see~\cite[Lemma A.7]{PrzRiv:13}). However, we will sometimes need a weaker conclusion, that $f|_K$ is
strongly topologically transitive, which immediately follows from weak topological exactness    and, less easily, from the topological transitivity, i.e. it is automatically true for $(f,K)\in\sA$,
see \cite{Kam:02} or \cite[Proposition 2.4]{PrzRiv:13}

We say that $f|_K$ is \emph{weakly topologically exact} (see also~\cite[Definition 1.15]{PrzRiv:13}) if  there exists a positive integer $N$
 such that for every nonempty open set $V\subset K$ there exists a positive integer $n=n(V)$ such that $f^{n}(V)\cup\ldots\cup f^{n+N-1}(V)=K$. Notice that this automatically implies the same for all $n\ge n(V)$, by acting by $f^{n-n(V)}$ on both sides of the equality.

\begin{definition}\label{def:strtoptra}
The map $f|_K$ is \emph{strongly topologically transitive}\footnote{This name was used in \cite{Kam:02}. See also the references therein. In \cite[Definition 2.3]{PrzRiv:13} it is called {\it density of preimages property}  (dp).}
if for every nonempty open $V\subset K$ there is a positive integer $m=m(V)$ such that $V\cup f(V)\cup\ldots \cup f^m(V)=K$.
Note that $f|_K$ is strong topologically transitivity if, and only if, the union of backward trajectories of every point in $K$ is dense in $K$.

Notice that for a given $\Delta$ one can choose $m(V)$ which works for all $V=B(x,\Delta)$, $x\in K$. We will denoted this number by $m_c(\Delta)$.
\end{definition}

Later we will need the following notation.  For a measure $\mu$ and $\Delta>0$ let
\begin{equation}\label{Upsilon}
\Upsilon(\mu,\Delta):=\inf_{x\in K} \mu(B(x,\Delta)).
\end{equation}
In our case, for $\mu$ being $\phi$-conformal this number is positive, which follows easily from strong topologically transitivity of $(f,K)$ due to the existence of $m_c(\Delta)$. In the case of $\mu$ being CaS measure, this positivity also holds provided this measure has infinite support.

In Section~\ref{strong}  we shall need the following very easy general fact. We will apply it to $X=K$.

\begin{lemma}\label{bunch}
Let $(X,\rho)$ be a compact infinite metric space without isolated points. Then for every $r>0$ and integer $m>0$ there exists $\e=\e(r,m)>0$ such that for every $x\in X$ there exists
$Z\subset X\cap B(x,r)$ with $\# Z\ge m$ such that for all $z_1,z_2\in Z$ with $z_1\not=z_2$ we have $\rho(z_1,z_2)\ge \e$.
\end{lemma}
\begin{proof}
For every point $x\in X$ for every $r$ and $m$ there exists $\e(x)$ satisfying the assertion because $x$ is not isolated, the problem in this lemma is to prove that one can choose one $\e$ working for all $x$. Assume that this is not the case and that there exists a sequence $(x_n)_n$ such that $\e(x_n)\to 0$. Let $x$ be an accumulation point of sequence $(x_n)$, passing to a subsequence we can assume $x=\lim x_n$. Then $B(x,r/2) \subset B(x_n,r)$ for $n$ large enough, hence for the same $m$ and for $r'=r/2$ we would have $\e(x)=0$ -- contradiction.
\end{proof}

The following lemma provides a justification of the notions of `exceptional' and `non-exceptional'.
Given $x\in K$ let
\[\begin{split}
	&\cO^-_{reg}(x)\eqdef\\
	&\big\{y=(f|_K)^{-n}(x)\colon n=0,1,\ldots, f^k(y)\notin S'\text{ for every }k=0,\ldots,n-1
	\big\}
\end{split}\]
be the union of all backward trajectories (finite or infinite) for $f|_K$ starting at $x$ that do not hit $S'$. Recall the notion of being weakly $S'$-exceptional in Section~\ref{exceptional}.

\begin{lemma}\label{lem:denseback}
	Let $(f,K)\in \sA_+$. Then
\begin{enumerate}
\item For every $x\in K$ not weakly $S'$-exceptional (we do allow $x\in S'$) the set $\cO^-_{reg}(x)$ is dense in $K$.
\item For every $x\in K$ not $S'$-exceptional  there exists $n\ge 0$ such that
$f^n(x)$ is not weakly $S'$-exceptional.
\item For every invariant measure $\mu$ which gives measure 0 to all $S'$-exceptional sets, the set of points $x\in K$ for which the set $\cO^-_{reg}(x)$ is dense in $K$ has positive measure $\mu$.
\end{enumerate}
\end{lemma}

\begin{proof}
By its definition, the set $\cO^-_{reg}(x)$ satisfies~\eqref{except-condition} for $\Sigma=S'$, hence as $x$ does not belong to any  weakly $S'$-exceptional set, $\cO^-_{reg}(x)$ must be dense in $K$ which proves the
Item 1.

Suppose that each $y=f^n(x), n\ge 0$ belongs to a weakly $S'$-exceptional set $\Sigma(y)$. Then, by definition,
$\cO^-_{reg}(y)\subset \Sigma(y)$, hence it is finite and weakly $S'$-exceptional.
Then $\Sigma:=\bigcup_{y\in \cO^+(x)}\cO^-_{reg}(y)$, for $\cO^+(x)$
 denoting the forward orbit of $x$, is finite, in particular not dense, and forward invariant. Therefore $\Sigma$ is $S'$-exceptional, hence $x$ belongs to an $S'$-exceptional set. This proves Item 2.

Item 3 follows immediately from Item 2.
\end{proof}

Formally the weak exactness property has not been needed in the proof of Lemma~\ref{lem:denseback}. However, it is needed to prove finiteness of the union of all weakly $S'$-exceptional sets, see \cite[Proposition 2.7, Part 2.]{PrzRiv:13}.
Have in mind the example of irrational rotation of the circle (though formally it is not in $\R$) or a `solenoid' $K=\omega(c)$ for an infinitely renormalizable unimodal map of the interval and $c$ its critical value, where all finite blocks $\{c,f(c),\ldots,f^n(c)\}$ are weakly $S'$-exceptional sets. These examples are strongly topologically transitive, but not weakly exact.

%-----------------------------------------------------------------------------------------------------
\subsection{Cantor repellers and `bridges'}\label{ss:bridges}
%-----------------------------------------------------------------------------------------------------

We describe a construction to `connect' two given uniformly expanding repellers by building `bridges' between them.

Let $f\colon U\to \R$ be a $C^{1+\varepsilon}$  map on an open domain $U\subset \R$.
We call a set $X\subset U$ a $f$-uniformly expanding \emph{Cantor repeller (ECR)} if $f|_X$ is a uniformly expanding repeller and a limit set of a finite graph directed system (GDS)  satisfying the strong separation condition (SSC) with respect to $f$. Recall that a GDS satisfying the SSC with respect to $f$ is a family of domains and maps satisfying the
following conditions (compare~\cite[pp.~3,~58]{MauUrb:00}):
\begin{itemize}
\item[(i)] There exists a finite family $\cU= \{U_k\}$  of open intervals with pairwise disjoint closures.
\item[(ii)] There exists a family $G=\{g_{k\ell}\}$ of branches of $f^{-1}$ mapping $\overline{U_\ell}$ into $U_k$ with bounded distortion (not all pairs
$k,\ell$ must appear here).\\
Note that a general definition of GDS allows many maps $g$ from each $\overline{ U_\ell}$ to each $U_k$. Here however there can be at most one, since we assume that $f$-critical points are far away from $X$ the intervals $U_k$ are short and that the maps $g$ are branches of $f^{-1}$.
\item[(iii)] We have
\[
	X= \bigcap_{n=1}^\infty\bigcup_{k_1,\ldots,k_n} g_{k_1k_2}\circ
	g_{k_2k_3}\circ \cdots \circ g_{k_{n-1}k_n} (U_{k_n}).
\]
We assume that we have $f(X)=X$ and hence that for each $k$ there exists $\ell$ and for each $\ell$ there exists $k$ such that $g_{k\ell}\in G$.
\end{itemize}
We can view $k$'s as vertices and  $g_{k\ell}$ as edges from $\ell$ to $k$ of a directed graph $\Gamma=\Gamma({\cU},G)$.

Observe that even if $X$ contains points in $NO(f,K)$, the resulting restricted map $f|_X$ is open as $X$ is a limit set of a GDS satisfying the SSC.

We import the following result from Pesin-Katok theory (see, for example,~\cite[Theorem 11.6.1]{PrzUrb:10}, \cite[Theorem 4.1]{PrzRiv:13} (concerning our interval case)
\cite{Gel:10}, and also~\cite{MisSzl:80} for earlier related results). Recall again the definition of weakly isolated in Definition~\ref{wi}.

\begin{lemma}\label{lem:Katok}
Consider $f\colon U\to \R$ for an open set $U\subset \R$, being a $C^2$ map with all critical points non-flat. Consider an arbitrary compact $f$-invariant $X\subset U$.
Let $\mu\in\cM_{\rm E}$ be supported on $X$ and have positive Lyapunov exponent. Let $\varphi\colon U\to\bR$ be a continuous function. Then there exists a sequence $(X_k)_{k\ge1}\subset U$ of ECRs or periodic orbits, such that for every $k$
	\begin{equation}\label{var-ineq}
		\liminf_{k\to\infty}P_{f|X_k}(\varphi)\ge h_\mu(f)+\int\varphi\,d\mu\,,
	\end{equation}
	and that every sequence $(\mu_k)_{k\ge1}$ of measures $\mu_k\in\cM_{\rm E}(X_k)$ converges to $\mu$ in the weak$\ast$-topology. Moreover, we have
	\begin{equation}\label{chi-ineq}
		\lim_{k\to\infty}\chi(\mu_k)\to\chi(\mu)\,.
	\end{equation}
If $X$ is weakly isolated, then one can choose $X_k\subset X$.
\end{lemma}

For $X=K$ for $(f,K)\in\sA_+^{\BD}$ satisfying the weak isolation condition,
Lemma \ref{lem:Katok} implies easily $P_{\rm varhyp}(\varphi_t)\le P_{\rm hyp}(\varphi_t)$, see \eqref{pressures}.
Indeed, by definition we can consider $\mu$ so that $h_\mu(f)+\int\varphi_t\,d\mu$ is arbitrarily close to  $P_{\rm varhyp}(\varphi_t)$ and next find $X_k$ as in Lemma  \ref{lem:Katok} applied to $\varphi\equiv 0$.
We get from \eqref{var-ineq} $h_{{\rm top}}(f|_{X_k})$ at least $h_\mu(f)$ (up to an arbitrarily small positive number) and applying \eqref{chi-ineq} for $\mu_k$ being measures of maximal entropy on $X_k$,  we obtain
\begin{equation}\label{P-ECR}
	P_{\rm varhyp}(\varphi_t)
	\le  P_{\rm ECR}(\varphi|_X)
		\eqdef \sup_{X\subset K}P_{f|_X}(\varphi)
		\le P_{\rm hyp}(\varphi_t),
\end{equation}
where the supremum is taken over all ECRs and periodic repelling orbits $X\subset K$.
This uses an inequality in the classical variational principle for the continuous potentials $\varphi_t|_{X_k}$, namely:
$h_{\mu_k}(f|_{X_k})+ \int \varphi_t\, d\mu_k \le P_{f|_{X_k}}(\varphi_t)$.

Compare also \cite[Lemma 4.2]{PrzRiv:13}.

\begin{definition}\label{defi:bridge}
Consider $(f,K)\in\sA_+^{\BD}$ satisfying the weak isolation condition, and set $X=K$.
Consider two ECRs or periodic repelling orbits $X_1$ and $X_2$ contained in $K$, resulting from families of domains $\cU_1=(U_{1,k})$ and $\cU_2=(U_{2,k})$ respectively.  A {\it bridge} connecting $X_1$ to $X_2$ is a backward trajectory that starts in $X_1$, goes close to $X_2$, and in-between avoids the singular set $S'$. More precisely it is a sequence of points $x_n\in K$ for $n=0,-1,\ldots,N$ such that $f(x_n)=x_{n-1}$, $x=x_0\in U_{1,k}$ for an integer $k$ and the components $W_n$ of $f^{-n}(U_{1,k})$ containing $x_n$ are disjoint from $S'$ for $n=0,\ldots,N$, disjoint from $\bigcup_k U_{i,k}$ for $n=1,\ldots,N-1$ and $W_N\subset U_{2,k'}$ for an index  $k'$.
\end{definition}

If $(f,K)$ is non-exceptional then for any two $X_1, X_2$ a bridge exists. In fact by
Lemma~\ref{lem:denseback} Item 3, in $X_1$, for every invariant measure $\mu$ on $X_1$, there is a positive measure set of points $x$ whose backward trajectories omitting $S'$ are dense in $K$. This allows to prove that for every $X_2$ a bridge to $X_2$ exists (for appropriate $\cU_i$).

Based on this idea of bridges connecting ECRs, $X_1$ to $X_2$ and $X_2$ to $X_1$, one can prove the following lemma (compare~\cite[Lemma 2]{GelPrzRam:10}).

\begin{lemma} \label{bridge}
	For any  two disjoint ECRs or periodic orbits $X_1,X_2 \subset K$  there exists an ECR set $X\subset K$ containing the set $X_1\cup X_2$. If $f$ is topologically transitive on each $X_i$, $i=1$, $2$, then $f|_X$ is topologically transitive.
\end{lemma}

In case two ECRs intersect, in a preliminary step we first have to modify one of them by the following lemma. For that we also have to take into consideration ECRs with respect to $f^n$.

In the following two results we use the standard notation  $S_\ell\phi=\phi+\phi\circ f+\cdots+\phi\circ f^{\ell-1}$ for a function $\phi\colon X\to\bR$ and natural number $\ell>0$.

\begin{lemma} \label{bridge2}
Let $X$ be a topologically transitive $f^n$-ECR. Let $\phi_i\colon X\to\bR$ be a finite number of H\"older continuous functions. Then for any open sufficiently small interval $D$ intersecting $X$ and for any $\varepsilon>0$ one can find a set $X'\subset D\cap X$ and a natural number $\ell>0$ such that $X'$ is an $f^\ell$-ECR and for every $\phi_i$ satisfies
\[
	\frac{1}{\ell} P_{f^\ell|X'}(S_\ell\phi_i)
	\geq \frac{1}{n}P_{f^n|X}(S_n\phi_i) - \varepsilon\,.
\]
\end{lemma}

Combining the Lemmas~\ref{lem:Katok},~\ref{bridge}, and~\ref{bridge2}, we obtain the following key result.

\begin{proposition}\label{decaf}
	There exists a sequence $\{a_m\}_m$ of positive integers and a sequence $X_m\subset K$ of $f^{a_m}$-invariant uniformly expanding topologically transitive repellers such that for every $t\in\bR$, we
have
\begin{equation}\label{micha}
    P(\varphi_t)=\lim\limits_{m\to\infty} \frac 1 {a_m} P_{f^{a_m}|X_m}(S_{a_m}\varphi_t)
    = \sup\limits_{m\ge 1} \frac 1 {a_m} P_{f^{a_m}|X_m}(S_{a_m}\varphi_t).
\end{equation}
For every $\alpha\in (\chi_{\rm inf}, \chi_{\rm sup})$ we have
\begin{equation}\label{fcon}
     F(\alpha)
    = \lim\limits_{m\to\infty}F_m(\alpha)=     \sup\limits_{m\ge 1}F_m(\alpha)
\end{equation}
and
\begin{equation}\label{alla}
    \lim\limits_{m\to \infty} \chi_{\rm inf}^m = \inf\limits_{m\ge 1}\chi_{\rm inf}^m =\chi_{\rm inf} ,
\quad
\lim\limits_{m\to \infty} \chi_{\rm sup}^m  = \sup\limits_{m\ge 1}\chi_{\rm sup}^m =\chi_{\rm sup} ,
\end{equation}
where $F_m$ and $\chi_{\rm inf/sup}$ are defined as in~\eqref{def:Fa}
and~\eqref{yu} but for the pressure $\frac{1}{a_m}P_{f^{a_m}|X_m}(S_{a_m}\phi)$ instead of $P(\phi)$.
\end{proposition}

The proofs of Lemmas~\ref{lem:Katok},~\ref{bridge}, and~\ref{bridge2} are almost verbatim as the proofs of~\cite[Lemmas~3,~2, and~4]{GelPrzRam:10}.
We would like to remark that to make the proof of Lemma~\ref{bridge2} work (and this should have been mentioned also in the proof of~\cite[Lemma 4]{GelPrzRam:10}) we need to choose $D$ of sufficiently small size such that the considered pull-backs (see Subsection~\ref{Telescope constructions} below) of $D$ are all univalent.

Finally notice that our sets $X_m$ are in $K$ since periodic points are dense in them and if they are sufficiently close to $K$ they are in $K$ by the weak isolation property.

\medskip
In fact one can assume all $a_m=1$. This follows from the fact that one can replace $X_m$ by $Y:=\bigcup_{j=0}^{a_m-1} X_m$ and find $Y'$ containing $Y$ also $f$-invariant uniformly expanding topologically transitive and else isolated (the property we might have lost for$Y$), see \cite[Proposition 4.5.6.]{PrzUrb:10}.

%-----------------------------------------------------------------------------------------------
\subsection{Local properties at critical points}\label{s:lp}
%-----------------------------------------------------------------------------------------------

Let $c$ be a non-flat critical point of a $C^1$ map $f$.
There exist constants $R_0>0$ and $A_0>1$ such that for every $x$ with $\lvert x-c\rvert\le R_0$ there is $d(c)\ge 2$ so that
\begin{equation}\label{e:slowl}
	A_0^{-1}\le \frac{\lvert f'(x)\rvert}{\lvert x-c\rvert^{d(c)-1}}\le A_0.
\end{equation}
In our situation there are only finitely many critical points, hence $A_0$ can be chosen uniformly.

%-----------------------------------------------------------------------------------------------------
\subsection{Telescope constructions}\label{Telescope constructions}
%-----------------------------------------------------------------------------------------------------

A number $n\in\bN$ is said to be a \emph{Pliss hyperbolic time} for a point
$x$ with exponent $\sigma>0$ if
\begin{equation}\label{hyp-times}
\lvert (f^k)'(f^{n-k}(x))\rvert \ge e^{k\sigma} \text{ for every
}1\le k \le n.
\end{equation}
The following is an immediate consequence of the Pliss lemma (see, for
example,~\cite{Pli:72} or \cite[Lemma 3.1]{ABV}. In Lemma~\ref{geometric1} below we will provide a more refined abstract result.

\begin{lemma}[Pliss lemma]\label{lemma-pliss}
Let $f$ be a $C^1$ map of a closed interval.
For every point $x$ with $\overline\chi(x)>\sigma>0$ there exist infinitely many Pliss hyperbolic times for $x$ with exponent $\sigma$. Moreover, if $H\subset\bN$ denote the set of hyperbolic times, then we have
\[
	\overline d(H):= \limsup_{n\to\infty}\frac{\#(H\cap[0,n])}{n}>0.
\]	
\end{lemma}	

We will call $\overline d(H)$ the \emph{upper density} of $H$.

\begin{proposition}[Telescope] \label{p.telescope}
	Let $f\colon I\to I$ be a $C^1$ multimodal map on a closed interval $I$.
	Given $\varepsilon>0$ and $\sigma>0$, there exist constants $A_1=A_1(\varepsilon)>0$ and $R_1=R_1(\varepsilon,\sigma)>0$ such that the following is true. Given  a point $x\in I$ with $\overline\chi(x)>\sigma>0$, for every number $r\in(0,R_1)$, for every $n\ge1$ which is a Pliss hyperbolic time for $x$ with exponent $\sigma$, and for every $k=1,\ldots, n$ we have
    \begin{align*}
    \diam \Comp_{f^{n-k}(x)} f^{-k}\left(B(f^n(x),r)\right)
    \le
    r\,A_1\, e^{k\varepsilon} \, |(f^k)'(f^{n-k}(x))|^{-1} \le \\ r\,A_1\, e^{-k(\sigma-\varepsilon)}.
    \end{align*}
\end{proposition}

This proposition was first stated and proved, with various variants,
in the complex setting in \cite{Prz:90}. Here we provide a slightly simplified proof for interval maps. Notice that only $C^1$ regularity is required and that no advanced distortion properties are needed.

We first prove a preliminary result. Recall the choice of constants $A_0,R_0$ in Section~\ref{s:lp}.

\begin{lemma}\label{l:crit}
	Given a point $c\in\Crit(f)$, let $T\subset B(c, R_0)$ be an interval and $x\in T$. Then
	\[
		\diam f(T) \geq \frac{1}{2d(c)A_0^4}\,\lvert f'(x)\rvert\,\diam T.
	\]
\end{lemma}

\begin{proof}
Denote $d=d(c)$. Let $y$ be the endpoint of $T$ furthest away from $c$ and $z$ the endpoint closest to $c$.
We can assume that $y$ is to the left of $c$.

Suppose first that $T$ is on one side of $c$ (i.e. its interior does not contain $c$).
Hence $\diam T = \lvert y-z\rvert$. By~\eqref{e:slowl}  we have
\[
	\lvert f'(x)\rvert \geq A_0^{-1} \,\lvert x-c\rvert^{d-1}\,,
\]
which implies
\[
	\diam f(T) \geq \int_{\lvert z-c\rvert}^{\lvert y-c\rvert} A_0^{-1} s^{d-1} \,ds \,.
\]
Combining this inequality with
\[
	\lvert f'(y)\rvert \leq A_0 \,\lvert y-c\rvert^{d-1}\,,
\]
we obtain
\[\begin{split}
	\diam f(T)
	&\geq \frac 1 {A_0 d} \left(\lvert y-c\rvert^d - \lvert z-c\rvert^d\right) \\
	&\geq \frac 1 {A_0 d}  \,\lvert y-c\rvert^{d-1} \cdot \lvert y-z\rvert 	\\	&\geq \frac{1}{A_0^2 d} \lvert f'(x)\rvert \diam T.
\end{split}\]

 Suppose now that $y$ and $z$ are on the opposite sides of $c$. Suppose first that $x\in [y,c]$. Then we apply the above to $T$ replaced by $[y,c]$ and obtain
 \[
 	\diam f(T)\ge \diam f([y,c])
	\ge \frac{1}{A_0^2 d}|f'(x)|\cdot |y-c|
	\ge \frac{1}{2A_0^2 d}|f'(x)|\cdot \diam T\,.
\]

 Suppose now that $x\in[c,z]$. Then we apply~\eqref{e:slowl} to the point $\widehat x= c-(x-c)$.  Then $\widehat x\in [y,c]$ and by the previous estimate and applying~\eqref{e:slowl} we obtain
 \[
 	\diam f(T)\ge \frac{1}{2A_0^2 d}|f'(\widehat x)|\cdot \diam T
	\ge \frac{1}{2A_0^4 d}|f'( x)|\cdot \diam T.
\]	
The proof is complete.
\end{proof}

\begin{remark}\label{unifyingdistortion}
If the bounded distortion condition holds, the following unifying version of Lemma~\ref{l:crit} holds.
There exists $C>0$ depending only on $f$ such that for any interval $T$ short enough and $x\in T$
\[
		\diam f(T) \geq C \lvert f'(x)\rvert\,\diam T.
	\]
 Indeed, if $T$ is close to a critical point, this follows from Lemma~\ref{l:crit}.  Otherwise if, say, $2\diamond T \cap \Crit(f)=\emptyset$, this follows from the bounded distortion condition.
 \end{remark}

Let $\lambda(r)$ be the minimum of $\lvert f'\rvert$ outside the $r$-neighborhood of $\Crit$. Let $\gamma(\cdot)$ be the modulus of continuity of the function $\lvert f'\rvert$. We have the following simple estimation.

\begin{lemma} \label{l:univ}
Given an interval $T$ of length $R'$ that is in distance at least $R$ from $\Crit(f)$, we have	
    \[
		\sup_{x,y\in T} \frac{\lvert f'(x)\rvert}{\lvert f'(y)\rvert}
		\leq \frac {\lambda(R) + \gamma(R')} {\lambda(R)}.
	\]
\end{lemma}

\begin{proof}[Proof of Proposition~\ref{p.telescope}]
Fix $\sigma$ and $\varepsilon \in(0,\min\{\sigma,1\})$. Let $n$ be a Pliss hyperbolic time for $x$ with exponent $\sigma$. For $k=0,\ldots,n$ and $r>0$ consider the sets
\[
	W_k(r) \eqdef {\rm Comp}_{f^{n-k}(x)} f^{-k}(B(f^n(x), r))\,,
\]
Let $R_1<R_\ast$ be positive numbers to be defined below. Let
\[
	M\eqdef\#\Crit,\quad d\eqdef\max d(c)\,.
\]	
Let $A(N)$ be the minimum distance $\ell$ such that there exists a critical point $c$ and a point $y\in B(c,\ell)$ such that $f^k(y) \in B(c,\ell)$ for some $k\leq N$. For a critical point $c$ we say that a number $k\leq n$ is a $c$-critical time if ${\rm {dist}}(c, W_k(R_1)) < R_\ast$.
Let $C^{crit}$ be the set of times which are $c$-critical for some $c\in {\rm Crit}$.

The following claim will immediately prove the proposition.

\begin{claim}
There exist positive numbers $R_\ast \in(0, R_0/2)$, $R_1\in(0,R_\ast)$, $A_1$, and a positive integer $N$ such that for all $k=0,\ldots,n$ and  $r\in(0,R_1)$ we have
\begin{itemize}
\item[(a)] $\displaystyle\diam W_k(R_1) \leq R_*$,\vspace{0.1cm}
\item[(b)] for every $i \in C^{crit}$, $i \leq k$, we have
	\[
		\frac{\diam W_{i -1}(r)}{\diam W_i(r) }
		\ge\frac{1}{2\,d\,A_0^4} \,\lvert f'(f^{n-i}(x))\rvert
	\]
	while
	\[
		\prod_{i \notin C^{crit}, i\leq k}  {\rm Dist}_{W_i(R_1)} f
		\leq \left(1+\frac\varepsilon2\right)^k\,.
	\]	
\item[(c)] given $c\in {\rm Crit}$, any two $c$-critical times $\le k$ differ  at least by $N$,\vspace{0.1cm}
\item[(d)] $\displaystyle\diam W_k(r) \leq r \,A_1 \,e^{k\varepsilon} \,\lvert (f^k)'(f^{n-k}(x))\rvert^{-1}$.
\end{itemize}
\end{claim}

\begin{proof}
Let us first fix some constants. Let
\begin{equation}\label{eq:choiceA}
	A_1 > (2\,d{A_0}^4)^{2M}\,.
\end{equation}
Choose a positive integer $N$ satisfying
\begin{equation}\label{eq:choiceN}
	A_1< e^{N\varepsilon/2}\,.
\end{equation}
Choose a positive number $R_\ast$ satisfying
\begin{equation}\label{eq:choiceRstar}
	R_\ast<\min\left\{\frac{R_0}{2},\frac{A(N)}{2}\right\}	\,.
\end{equation}
Choose a positive number $R_1$ satisfying
\begin{equation}\label{eq:choiceR1}
	R_1<\min\left\{\lambda(R_0) \frac{R_\ast}{2}A_1, \frac{{R_\ast}^d}{2}A_0 A_1 d\right\}
\end{equation}
and
\begin{equation}\label{eq:choiceR1b}
	\gamma(R_1) < \varepsilon \frac{\lambda(R_\ast)}{2}\,.
\end{equation}

The proof is by induction on $k$.
Note that all statements (a)-(d) hold for $k=0$. Suppose that the claim holds for $k=\ell$.

Observe that we have chosen $A_1$ and $N$ sufficiently large such that we can conclude (d) from (a)--(c). We have chosen $R_\ast$ sufficiently small such that (a) implies (c). Finally, we have chosen $R_1$ small enough that the previous lemmas and (a) imply (b). We now are going to conclude the properties (a)--(d) for $\ell+1$.
\\[0.1cm]
Property (a). Since (d) holds at time $\ell$ and since $n$ is a Pliss hyperbolic time, by the choice of $\varepsilon$ we have
\[
	\diam W_\ell(R_1) \leq R_1 A_1e^{\ell\varepsilon}\lvert (f^\ell)'(f^{n-\ell}(x))\rvert^{-1}
	\le R_1A_1e^{\ell(\varepsilon-\sigma)}
	<R_1A_1
	\,.
\]	
Thus,  together with~\eqref{eq:choiceR1}, we can conclude
\[
	\diam W_{\ell+1}(R_1)
	\leq \frac{R_1 A_1}{\lambda(R_0)}  + (A_0 \,d\, R_1 A_1)^{1/d}
	\leq R_\ast
\]
and therefore, we have (a) at time $\ell+1$.\\[0.2cm]
Property (b). We will distinguish two cases.

If $\ell+1\in C^{crit}$ then ${\rm{dist}} (c,W_{\ell+1}(R_1))< R_\ast<R_0/2$ for some $c\in\Crit$ and by Lemma \ref{l:crit} we have
\[
	\frac{\diam W_\ell(r) }{ \diam W_{\ell+1}(r)}
	\ge\frac{1}{2\,d\,A_0^4} \, \lvert f'(f^{n-(\ell+1)}(x))\rvert
\]
while the estimate of the product remains unaltered. Thus, (b) holds for $\ell+1$ in this case.

If $\ell+1\notin C^{crit}$, then ${\rm{dist}}(c,W_{\ell+1}(R_1))\ge R_\ast$ for all $c\in\Crit$ and hence, by Lemma \ref{l:univ} and by~\eqref{eq:choiceR1b} we obtain
\[
\dist_{W_{\ell+1}(R_1)} f
= \sup_{x,y\in W_{\ell+1}(R_1)}\frac{\lvert f'(x)\rvert}{\lvert f'(y)\rvert}
\leq \frac {\lambda(R_\ast) + \gamma(R_1)} {\lambda(R_\ast)}
\leq  1+\frac{\varepsilon}{2}
\,.
\]
Thus, we obtain
\[
	\prod_{i \notin C^{crit}, i\leq \ell+1}  {\rm Dist}_{W_i(R_1)} f
		\leq \left(1+\frac{\varepsilon}{2}\right)^{\ell+1}\,.
\]
On the other hand, we have
\[\begin{split}
	\frac{\diam W_\ell(r) }{ \diam W_{\ell+1}(r)}
	&\ge   \lvert f'(f^{n-(\ell+1)}(x))\rvert\cdot
	\left(\dist_{W_{\ell+1}(R_1)}f\right)^{-1}  \\
	&\ge \frac{1}{2\,d\,A_0^4}\,\lvert f'(f^{n-(\ell+1)}(x))\rvert\,,
\end{split}\]
that is, (b) holds for $\ell+1$ in this case.\\[0.2cm]
Property (c). If (a) holds at times $0,\ldots,\ell$, then it follows from~\eqref{eq:choiceRstar} that (c) holds at time $\ell+1$.
Indeed, by (a), $W_i(R_1) \subset B(c, 2R_\ast)$ holds for all $c$-critical times $i\le \ell$. The definition of $A(N)$ implies that no point $y\in W_i(R_1)$ has a preimage $f^{-(\ell+1)}(y)\in B(c, 2R_*)$ for $\ell+1<N$. Hence, times $i + \ell+1$ are not $c$-critical for $\ell+1<N$.\\[0.1cm]
Property (d). It remains to show that (d) is true at time $\ell+1$.
For all noncritical times $i\le \ell+1$ we have
\[
	\frac {\diam W_{i-1}(r)} {\diam W_i(r)}
	\geq \lvert f'(f^{n-i}(x))\rvert \cdot \left({\rm Dist}_{W_{i}(R_1)} f\right)^{-1}
	\,,
\]
while for all $c$-critical times $i\le\ell+1$, using (b), we have
\[
	\frac {\diam W_{i-1}(r)} {\diam W_i(r)}
	\geq \lvert f'(f^{n-i}(x))\rvert \cdot \frac{1}{2\,d\,A_0^4}.
\]
Thus, together with~\eqref{eq:choiceA}, we obtain
\[\begin{split}
	\diam & W_{\ell+1}(r)
	= r \frac {\diam W_{1}(r)} {\diam W_0(r)}  \cdots
		\frac {\diam W_{\ell+1}(r)} {\diam W_\ell(r)}  \\
	&\le r\,\lvert (f^{\ell+1})'(f^{n-(\ell+1)}(x))\rvert^{-1}
		\cdot\prod_{i\notin C^{crit},i\le\ell+1}\dist_{W_i(R_1)}f	
		\cdot\left(2\,d\,A_0^4\right)^{2M}\\
	&\le r\, \lvert (f^{\ell+1})'(f^{n-(\ell+1)}(x))\rvert^{-1}  e^{(\ell+1)\varepsilon}A_1
\end{split}\]
Here we used also the fact that there are only $M$ critical points and that by (c) the critical times $\le \ell+1$ for each of them happen at most once every $N$ iteration. Thus, we obtain (d) for $\ell+1$.
This finishes the proof.
\end{proof}
This completes the proof of the proposition.
\end{proof}

\begin{remark}
The proof is simpler when one assumes bounded distortion condition. In this case it is not necessary to go step by step, one can take the whole non-critical stretch of the trajectory of $x$ and use the bounded distortion there.
One proceeds as in the pull-back construction, see Subsection~\ref{pullback}.\end{remark}

%-----------------------------------------------------------------------------------------
\subsection{Pull-back construction}\label{pullback}
%-----------------------------------------------------------------------------------------

To formulate our second preliminary technical result we use the following construction, see~\cite{GraSmi:09,PrzRivSmi:03,PrzRivSmi:04} and \cite[Definition 5.1]{PrzRiv:13}.
Some complications arise in presence of indifferent periodic orbits.

Recall that $S'$ denotes the restricted singular set
\[
	S'=\Crit(f)\cup NO(f,K)\,.
\]

\noindent\textbf{Pull-back construction:}
Let $(f,K,U)\in \sA_+$.
	Given $n>0$, $R>0$, and $y\in K$,
consider  a backward trajectory $(y_i)_{i=1}^n$ of $y$, that is, $y_0=y$ and $y_{i+1}\in f^{-1}(y_i)$ for every $i=1, \ldots, n-1$. Let $k_1$ be the smallest integer for which  $\Comp_{y_{k_1}}f^{-k_1}(B(y,R))$ contains a point in $S'$.
For every $\ell\ge 1$ let then $k_{\ell+1}>k_\ell$ be the smallest integer such that
$\Comp_{y_{k_{\ell+1}}}f^{-(k_{\ell+1}-k_\ell)}(B(y_{k_\ell},R))$ contains a point in $S'$ and so on.

In this way, to each backward branch $(y_i)_i$ considered, we associate a sequence $(k_\ell)_\ell$ of numbers $1\le k_\ell\le n$ that has a maximal element that we denote  by $k$. We consider the set of all pairs $(y_k, k)$ built from all the backward branches that start from $y$ and denote by $N(y,n,R)$ the cardinality of this set.
\medskip

In order to make this construction work, we need to know that every considered preimage is well-defined.
If there are no indifferent periodic orbits in $K$ then this is easy to achieve provided that $R$ is small enough so that all pull-backs
of intervals of length at most $2R$ intersecting $K$ are well inside $U_K$ (it is sufficient if their closures are in $U_K$), so that they do not become truncated. This is possible in the absence of indifferent periodic orbits by \cite[Lemma 2.10]{PrzRiv:13}, compare the definition of BaShrink in Subsection~\ref{ss:TCE}.\footnote{Notice that this difficulty does not appear for rational mappings on the Riemann sphere.}

\smallskip

In presence of indifferent periodic orbits this construction works if we {\it a priori} know
or assume the following:

\smallskip

\begin{itemize}
\item [($\ast$)]
Let
\[
	R':=\frac 12{\rm{dist}}(S',\Indiff(f))
\]
Assume that all pull-backs  $\Comp_{y_i}f^{-i}(B(y,R))$ for $i=0,\ldots,k_1$ have diameters smaller than $R'$ and are well inside $U_K$.
\end{itemize}
Additionally we assume that $R$ is small enough such that for every $x\in S'$ the pull-backs $\Comp_{z}f^{-k}(B(x,R))$,  for all $k=0,1,\ldots$ and $f^k(z)=x$ have diameters smaller than $R'$ (in particular with $k=0$ we have $2R<R'$). This is possible
since $S'\cap \Indiff(f)=\emptyset$ ($S'$ has no periodic points) and since $S'$ is finite (refer again to \cite[Lemma 2.10]{PrzRiv:13}, by BaShrink). We also assume, as before, that $R$ is small enough that all pull-backs of intervals of length at most  $2R$ intersecting $K$ and disjoint from $\Indiff(f)$ are well inside $U_K$.

\smallskip

Under the above assumptions, the pull-back construction works. Indeed, the pull-backs  $\Comp_{y_i}f^{-i}(B(y,R))$ for $i=0,\ldots,k_1$ are short by assumption.
Then for $c\in S'\cap \Comp_{y_{k_1}}f^{-k_1}(B(y,R))$ we have $|c-y_{k_1}|<R'$, hence $B(y_{k_1}, R')$ and therefore $B(y_{k_1}, R)$ are disjoint from $\Indiff(f)$ and we can continue the construction.

\begin{remark}\label{rem:pulllback}
	As mentioned before, the condition ($\ast$) holds for $R$ small enough if $B(y,R)$ does not contain any indifferent periodic point. Notice that it also holds if $k_1$ is a Pliss hyperbolic time
%on the considered piece of orbit of the pull-back
 for $y_{k_1}$, by Proposition~\ref{p.telescope}.

 In fact due to Proposition~\ref{p.telescope} the backward shrinking as, say,  in ($\ast$) holds for any piece of a trajectory $x,f(x),...,f^n(x)$ if $n$ is a Pliss hyperbolic time for $x\in K$.
 %(i.e. for $x=y_i$ and $n=i$).
\end{remark}

\begin{proposition}[Pull-back]\label{p.pullback}
	Given $\varepsilon>0$, there exist numbers $A_2=A_2(\varepsilon)>0$ and $R_2=R_2(\varepsilon)>0$ such that for every $r\in (0,R_2)$ and every $y\in K$ we have
	\[
		N(y,n,r)<A_2\,e^{n\varepsilon}\quad\text{ for all }n\ge 1.
	\]
\end{proposition}

\begin{proof}
Given $\varepsilon>0$, let us choose a positive integer $N$ large enough such
\begin{equation}\label{c.cond}
	\frac {M_1\ln((N+1)\,M_2)} {N} \le \varepsilon\,.	
\end{equation}

Observe that there are no cycles in $S'$ that is, $f^k(x)\ne x$ for every $k\ge1$ and every $x\in S'$,
see \cite[Lemma 2.2]{PrzRiv:13}. Though, observe that a point in $S'$ may be mapped to another point in this set. So, given $N$, there is
\[
	\delta'(N)\eqdef \min_{x\in S'}\min_{1\le k\le N} |f^k(x)-x|>0.
\]
Moreover, let $\delta(N)$ be sufficiently small that for any interval $I$ of length $\delta(N)$ and for any $1\le k\leq N$ the size of each component of $f^{-k}(I)$ is smaller than $\delta'(N)/2$.

Let
\[
	M_1\eqdef\# S',\quad
	M_2\eqdef \#S' + \sum_{c\in\Crit(f)}(d(c)-1)\,.
\]

Let us consider an integer $n>N\varepsilon^{-1}$.

Let $\delta\in(0,\delta(N))$. Given a point $y\in K$ and $r\in(0,\delta)$, consider the ball $B=B(y,r)$. By definition of the number $N(y,n,\delta)$, every backward branch of $y$ determines a sequence of numbers $(k_\ell)_\ell$.
The definition of $\delta(N)$ implies that if in our construction along some backward trajectory $(y_i)_{i=1}^n$ of $y$ there are two `singular times' $k_\ell, k_m$, $1\le k_\ell<k_m\le n$, corresponding to the \emph{same} singular point $x\in S'$ then $k_m-k_\ell>N$.
As a consequence, given a stretch of this backward trajectory with times in $\{j+1,\ldots,j+N\}$, there are at most $M_1$ `singular times' $k_\ell$ corresponding to a point in $S'$.

Finally, given $1\le i\le n$, considering all possible backward trajectories starting in $y$,  at most $M_2$ of them have a \emph{smallest} `singular time' $k_1=i$, that is, the corresponding component of $f^{-i}(B)$ will intersect $S'$ for the first time exactly after $i$ backward iterates. Observe that any of those components will be disjoint. This implies, in particular, that at most $M_2$ of them can intersect $S'$. The same argument will apply not only for backward branches starting at $y$, but also the branches starting at $f^{-k_i}(y)$ that we use in the construction.

Lets summarize the construction in the statement of the lemma. When we divide the time interval $\{1,\ldots, n\}$ into intervals $I_k=\{kN+1, kN+1,\ldots, (k+1)N\}$, for any backward trajectory of $y$ at most $M_1$ singular times can be in $I_k$. Hence, there are at most $(N+1)^{M_1}$ possibilities for $\{k_j\}\cap I_k$. Moreover, given the sequence $\{k_i\} \cap \{1,\ldots,  kN\}$ each of possible subsets $\{k_j\}\cap I_k$ can be realized for at most $M_2^{M_1}$ subbranches. This lets us estimate the number of backward branches for which there is a singular time in $I_k$ by

\[
[(N+1)M_2]^{(k+1)M_1}
\]
and the number of all singular branches up to a time $n$

\[
N(y, n, r) \leq \sum_{k=0}^{\lceil n/N\rceil} [(N+1)M_2]^{(k+1)M_1} \leq A_2(\varepsilon) e^{n\varepsilon}
\]
This proves the proposition.
\end{proof}

%--------------------------------------------------------------------------------------------------
\subsection{Conical points}\label{ss:conical points}
%--------------------------------------------------------------------------------------------------

After~\cite{DenMauNitUrb:98}, for $(f,K,U_K)\in\sA$, a point $x\in K$ is said to be {\it conical} if there exist a number $r>0$, a sequence of positive integers $n_i\nearrow \infty$, and a sequence $(U_i)_i$ of neighborhoods $U_i\subset U_K$ of $x$, such that $f^{n_i}$ is defined on $U_i$ and $f^{n_i}(U_i)=B(f^{n_i}(x),r)$
and that $\dist f^{n_i}|_{U_i}$ is bounded uniformly in $i$. We say that $x$ is \emph{$S'$-conical} if additionally for all $i$ and $0\le k<n_i$ we have $f^k(U_i)\cap S'=\emptyset$.

Notice that $|U_i|\to 0$ as $n_i \to\infty$, otherwise there would exist a wandering interval or an attracting (or indifferent attracting from the side of $K$) periodic orbit in $K$, which is not possible, see \cite[Chapter 4, Theorem A]{deMvanS:}, compare \cite[Subsection 2.3]{PrzRiv:13}.

\begin{remark}
 	We observe that the concept of conical points is in fact very much related to the concept of points that are `seen' by an inducing scheme (see, for example~\cite{IomTod:11} for an application of this technique to questions similar to those discussed here). Similarly to Proposition~\ref{prop:wander} below, in~\cite[Theorem 3.3 c)]{IomTod:11} the authors show that the inducing scheme captures all but a set of points of zero Hausdorff dimension. See also \cite{PrzRiv:07} in the holomorphic case.
\end{remark}

\begin{proposition}[Non-conical hyperbolic points]\label{prop:wander}	
	The set of points $x\in K$ that are not $S'$-conical and satisfy $\overline\chi(x)>0$ has Hausdorff dimension zero.
\end{proposition}
\begin{proof}

Let us choose some numbers $\sigma >0$ and $\varepsilon>0$. Let
\[
	r\eqdef\frac 1 {2(A_1+2)}\min\{R_1,R_2\},
\]
where $A_1$ and $R_1$ are constants given by Proposition~\ref{p.telescope}
and where $R_2$ is given by Proposition~\ref{p.pullback}. Let $R\eqdef2^{-1}\min\{R_1,R_2\}$.

We fix a finite family of balls $\{B_i\}_{i=1}^N$ of radius $3r$ such that any ball of radius $2r$ which intersects $K$ must be contained in one of the balls in this family.

Given $m\ge1$ and $\sigma>0$, let $G(m,\sigma)$ be the set of points $x\in K$ with $\overline\chi(x)>\sigma$ for which for all $\ell>m$ being a Pliss hyperbolic time with exponent $\sigma$, the pull-back of  $B(f^\ell(x),2r)$ for $f^{\ell}$ containing $x$ meets a  point in $S'$, that is, we have
\begin{equation}\label{def:bbran}
	\Comp_{f^k(x)}f^{-(\ell-k)}(B(f^n(x), 2r))\cap S' \ne\emptyset
\end{equation}
for some $k\in\{0,\ldots, \ell-1\}$.
Since $\ell$ is a Pliss hyperbolic time, by Proposition~\ref{p.telescope}, all the above pull-backs are small, and in particular well inside $U_K$. Hence the assumption ($\ast$) enabling the pull-back construction holds. Compare Remark~\ref{rem:pulllback}.

We claim that $\dim_{\rm H}G(m,\sigma)=0$ for any $m\ge1$ and $\sigma>0$.
Indeed, given $n>m$, let us denote by $G(m,\sigma, n)$ the set of points in $G(m,\sigma)$ for which $n$ is a Pliss hyperbolic time with exponent $\sigma$. Recall that
$\overline\chi(x)>\sigma>0$ implies that there exist infinitely
many Pliss hyperbolic times for $x$ with exponent $\sigma$. Hence, we
have
\begin{equation} \label{eqn:prs}
G(m,\sigma) = \bigcap_{k\geq m} \bigcup_{n>k} G(m, \sigma, n).
\end{equation}
Let $x\in G(m, \sigma)$. Then $x\in G(m,\sigma,n)$ for  some $n>m$ arbitrarily large.
Let $B_j$ be a ball from the above chosen family which contains $B(f^n(x),2r)$ and  pick $y$ such that $B_j=B(y,3r)$. We will apply the pull-back
construction of Proposition~\ref{p.pullback} to the point $y$, the number
$n$, the radius $R$, and $f^{-n}_x$. Let $(y_k,k)$ be given by the pull-back construction.

We first note that Proposition~\ref{p.telescope} and $\lvert f^n(x)-y\rvert \leq
r\le R_1$ imply that
\[
	\lvert (f^{n-k}(x))-y_k\rvert \leq
	r\,A_1e^{-k(\sigma-\varepsilon)}
	< r\,A_1.
\]
By our choice of constants we have $r\,A_1+2r=R$, which implies
\[
	B(f^{n-k}(x), 2r)\subset B(y_k,R).
\]	
Since $x\in G(m,\sigma)$, the pull-back of $B(f^\ell(x),2r)$ must meet a  point in $S'$ whenever $\ell>m$. On the other hand $k$ is the maximal element and hence the pull-back of $B(y_k,R)$ does not meet any critical point. Thus, we conclude $m\ge n-k$.
Thus, fixed $n$ and $B_j$, by Proposition~\ref{p.pullback} the point $x$ must belong to one of at most $A_2 e^{n\varepsilon} (\deg f)^m$ preimages of $B_j$.

Also observe that $B_j=B(y,3r)\subset B(f^n(x),4r)$ and $4r<R_1$. Since $n$ is a Pliss hyperbolic time for $x$ with exponent $\sigma$, by Proposition~\ref{p.telescope}
this pre-image of $B_j$ has diameter not greater than $4r\,A_1
e^{-n(\sigma - \varepsilon)}$.

Hence we showed that every point in $G(m,\sigma,n)$ belongs to
the $n$th pre-image of some ball $B_j$ along a backward branch of $y$ and the maximal element $(y_k,k)$ as defined in the pull-back construction satisfies $k\ge m-n$.
Thus, by Proposition~\ref{p.telescope}, the set $G(m,\sigma, n)$ is contained in a union of at most $A_2e^{n\varepsilon} (\deg f)^m$
sets of diameter not greater than $4r\,A_1 e^{-n(\sigma - \varepsilon)}$. We apply $G(m, \sigma) \subset \bigcup_{n>m}G(m, \sigma,n)$ to obtain

\[
\dim_{\rm H}G(m,\sigma)
\leq \frac{ \varepsilon }{\sigma -\varepsilon} ,
\]
recall~(\ref{eqn:prs}).
As $\varepsilon$ can be chosen arbitrarily small, this proves $\dim_{\rm H}G(m,\sigma)=0$.

Finally let us consider the set
\[
	G\eqdef\bigcup_{m=1}^\infty\bigcup_{n=1}^\infty G\big(m,\frac 1 n\big).
\]
By the above we have $\dim_{\rm H}G=0$.
Let $x\in K\setminus G$ be a point with $\overline\chi(x)>0$. Let us now show that $x$ must be $S'$-conical. By the above, there is a sequence of Pliss hyperbolic times $(n_i)_i$ such that $\Comp_{f^k(x)}f^{-(n_i-k)}(B(f^{n_i}(x),2r))$ does not intersect $S'$ for any $k=0,\ldots,n_i-1$. Hence, in particular $f^{n_i}$ maps $I_1^i=\Comp_{f^k(x)}f^{-(n_i-k)}(B(f^{n_i}(x),2r))$ diffeomorphically onto $I_2^i=B(f^{n_i}(x),2r)$. Note that $(1-\varepsilon)I_2^i\subset B(f^{n_i}(x),r)$. Hence, as $f$ satisfies bounded distortion near $K$ there exist $C(\varepsilon)>1$ such that for every $i\ge1$
\[
	\dist g|_{B(f^{n_i}(x),r)}\le C(\varepsilon),\quad\text{ where }
	g=(f^{n_i}|_{I_1^i})^{-1}.
\]
Thus, $x$ is $S'$-conical with $r$, $(n_i)_i$, and $U_i=I_1^i$.
This finishes the proof.
\end{proof}

%++++++++++++++++++++++++++++++++++++++++++++++++++++++++++
\section{Proof of Theorem \ref{main2}}\label{sectionofthetheorem}
%++++++++++++++++++++++++++++++++++++++++++++++++++++++++++

In this section we shall prove Theorem~\ref{main2}, except for \eqref{empty1} which is postponed to Section~\ref{s:completeness}.

%++++++++++++++++++++++++++++++++++++++++++++++++++++++++++
\subsection{No conical points with small exponent}\label{ss: no small exponent}
%++++++++++++++++++++++++++++++++++++++++++++++++++++++++++

We start with a corollary, which, together with Proposition~\ref{prop:wander} in the previous section, proves in particular~\eqref{dim0}  in Theorem~\ref{main2}
(compare also \cite[Remark 4.1]{IomTod:11}).
In the next sections we shall provide improvements, skipping the assumption points are conical, in particular proving \eqref{empty1} and its generalizations.

\begin{corollary}\label{cor:conicalupperlower}
	Let $(f,K)\in \sA^{\BD}_+$ be non-exceptional and satisfy the weak isolation condition. Then the set
$\{x\colon x  \text{ is conical}, \,\,\overline{\chi}(x)<\chi_{\inf}\}$ is empty.

\end{corollary}
In particular we can write above $S'$-conical, which together with Proposition~\ref{prop:wander} is also sufficient to prove (\ref{dim0}). Moreover, we point out that the assumption of weak isolation is not needed in this proof.

\smallskip

Let us summarize that we get a decomposition of $K$ into three sets:
\begin{itemize}
\item points $x$ being non-$S'$-conical with $\overline\chi(x)\le 0$,
\item points $x$ being non-$S'$-conical with $\overline\chi(x) > 0$, and
\item $x$ being $S'$-conical.
\end{itemize}
We do not know much about the first set. By Proposition~\ref{prop:wander}, the second one has Hausdorff dimension $0$. Any $x$ in the third set satisfies
$\overline{\chi}(x)\ge\chi_{\inf}$ by Corollary~\ref{cor:conicalupperlower}.

\begin{proof}[Sketch of Proof of Corollary~\ref{cor:conicalupperlower}]
For $x$ conical  we find an arbitrarily small neighborhood $U$ of $x$ and $n$ such that $f^n$ maps $U$ onto $B(f^n(x),r)$ with bounded distortion. By strong transitivity $f^{-m}(x)$ is $r/2$-dense for sufficiently large $m$.  As $\lvert U\rvert$ is small its respective pull-back $U'$ is in $B(f^n(x),r)$.
Then one finds a shadowing periodic orbit and distributes an invariant measure $\mu$ on it. We obtain $\chi(\mu)\le \overline\chi(x)+\e$ with $\e$ arbitrarily close to 0. Observe that capturing of critical points when going forward from $U'$ to $U$ does not hurt as small derivative of $f$ close to them works to our advantage.
\end{proof}

%*********************************************************************************************
\subsection{Upper bound}\label{ss:upper_bound}
%*********************************************************************************************

The content of this subsection, like Sections~\ref{ss: no small exponent} and~\ref{ss: lower bound interior}, is very similar to the complex case in \cite{GelPrzRam:10}. However, the modified versions of the proof of Proposition \ref{prop:conicalupper} will be considered in Sections \ref{s:completeness} and \ref{strong}, so we include the details.

Thus, a part of Theorem~\ref{main2} consists of

\begin{proposition}\label{prop:conicalupper}
Given numbers $\alpha\le\beta\le \chi_{\rm sup}$ with $\beta>0$, and additionally assuming $\alpha>0$ if $\chi_{\inf}=0$, we have
\[
	\dim_{\rm H}\cL(\alpha,\beta) \le
		\max\Big\{0,\max_{\alpha\le q\le\beta}F(q)\Big\}\,.
	\]
\end{proposition}

The possibility to restrict to $S'$-conical points due to Proposition~\ref{prop:wander} makes the proof easier than the proofs in Sections~\ref{ss:TCE} and \ref{strong}, where one deals with finite criticality (TCE) and has to deal with true conformal measures.

\begin{proof}%[Proof of Proposition~\ref{prop:conicalupper}]
	Let $\alpha$ and $\beta$ be as in the assumption and $x\in\cL(\alpha,\beta)$. By Proposition~\ref{prop:wander} it is enough to assume that $x$ is $S'$-conical with a corresponding number $r>0$, a sequence $n_i\nearrow\infty$, and a sequence $(U_i)_i$ of neighborhoods of $x$
such that $f^{n_i}(U_i)= B(f^{n_i}(x),r)$ ($U_i$'s are pull-backs in another terminology) and that $\dist f^{n_i}|_{U_i}$ is bounded uniformly in $i$ by some constant $C$. By passing to a subsequence we can assume that $\frac 1 {n_i} \log |(f^{n_i})'(x)|$ has a limit $q\in [\alpha, \beta]$. If $\chi_{\inf}>0$ then
as in the proof of Corollary~\ref{cor:conicalupperlower} we can find a periodic point $x$ with $\chi(x)$ less than $q+\e$ for $\e$ arbitrarily close to 0, for $n_i$ respectively large.
Hence $q\ge \chi_{\inf}>0$. If $\chi_{\inf}=0$ we use the assumption $\alpha>0$ to assure that $q>0$. 	
	
\smallskip

Every $f^k(U_i)$ and thus $f^{k-n_i}(B(f^{n_i}(x),r))$, for every $k=0,\ldots,n_i-1$, are disjoint from $S'$.
So when we consider $\mu^*_t$, the measure defined in Subsection~\ref{ss:conformal} with Jacobian $\lambda\lvert f'\rvert^{t}$ for $\lambda=e^{P(t)}$, applying the conformality relation for $f$ repeated times,  we have, due to Lemma~\ref{l:K-homeo},
\[
	 \int_{f_x^{-n_i}(B(f^{n_i}(x),r))}  \lvert (f^{n_i})'\rvert^{t}
			\,d\mu^*_t
	= \mu^*_t\big(B(f^{n_i}(x),r)\big)\,\lambda^{-n_i}	\,.
\]	
Given $r$, there is a constant $c=c(r,t)>1$ such that $c^{-1}\le \mu^*_t(B(y,r))\le c$ for every $y\in K$, see Subsection~\ref{toprec} (where the notation $\Upsilon$ was used).
This, together with the distortion results, implies
\begin{equation}\label{eq:distoconfmeau}
	\mu^*_t(f_x^{-n_i}(B(f^{n_i}(x),r)))
	\ge
	\lvert (f^{n_i})'(x)\rvert^{-t} C^{-1}c^{-1} e^{-n_i(P(t)+\varepsilon)}\,.
\end{equation}
Now we conclude as in the proof of~\cite[Proposition 2]{GelPrzRam:10}, see also the end of Section~\ref{strong}, that the lower local dimension of $\mu^*_t$ is bounded from above by $t+(P(t)/q$  at $x$. Then, continuing as in this proof, and applying the Frostman lemma, we obtain
\begin{equation}\label{confestim}
	\dim_{\rm H}\cL(\alpha,\beta)\le
	\max\Big\{0,
		\max_{\alpha\le q\le\beta}F(q)
	\Big\}\,.
\end{equation}
In the case $\chi_{\inf}>0$ we have got a formally better result, with $ \max\{\alpha,\chi_{\inf}\}\le q \le \beta$ in the above estimate, but this does not make difference since $F(\alpha)=-\infty$ for $\alpha<\chi_{\inf}$.
This finishes the proof.
\end{proof}

Notice, that in particular for $0< q\le \beta < \chi_{\inf}$
for $t$ large enough we have $P(t)\le P(t)<-q\,t$. Hence, for $\varepsilon$ small enough the measures in \eqref{eq:distoconfmeau} diverge. So such $S'$-conical point $x$ cannot exist.
Observe that this way we have obtained an alternative proof of the partial result in Corollary~\ref{cor:conicalupperlower}  stating that the set of $S'$-conical points is empty (the emptyness of the  conical set does not follow, since to use CaS measure we need to omit $S'$).

Summarizing, Corollary~\ref{cor:conicalupperlower} together with Proposition~\ref{prop:wander} prove~\eqref{dim0}.
This approach will be used again in Section~\ref{strong}.

%===========================================================
\subsection{Lower bound -- interior of the spectrum}\label{ss: lower bound interior}
%===========================================================

The proof of the lower bound for the regular spectrum in Theorem \ref{main2} goes exactly like in \cite[Section 5]{GelPrzRam:10}, with the only difference that we use $S'$ instead of the set of critical points.
The main idea is to construct ECRs and use bridges, as recalled in Subsection~\ref{ss:bridges}, and to apply Proposition~\ref{decaf}.

We use the formula for the Lyapunov spectrum on a uniformly hyperbolic repeller $X_m$ (see for example \cite{Pes:97})
which gives in particular
$$
\dim_{\rm H} \{x\in X_m\colon \chi(x)=\alpha\} \ge F_m(\alpha).
 $$
Hence, using \eqref{fcon} for any $\alpha\in (\chi_{\inf},\chi_{\sup})$ we get
$$
\dim_{\rm H} \{x\in K\colon \chi(x)=\alpha\} \ge \sup_m F_m(\alpha)=F(\alpha)\,,
$$
which is the lower bound asserted in Theorem~\ref{main2}.

%##########################################################
\subsection{Lower bound -- boundary and irregular part}
%##########################################################

It is sufficient to have the following result.

\begin{theorem} \label{closing}
	Let  $(X_i)_i$ be a sequence of ECRs invariant for $f^{a_i}$, in $K$ (Cantor sets, we do not allow individual periodic orbits here), disjoint from $\partial \hat I_K$.
Let $(\phi_i)_i$ be a sequence of H\"older continuous potentials and let $(\mu_i)_i$ be  a sequence of equilibrium states for $\phi_i$ with respect to $f^{a_i}|_{X_i}$.
    Then we get
    \[
    \dim_{\rm H}\big\{x\in K\colon
    \underline\chi(x)=\liminf_{i\to\infty}\chi(\mu_i),\,
    \overline\chi(x) = \limsup_{i\to\infty}\chi(\mu_i)\big\}
    \ge \liminf_{i\to\infty} \dim_{\rm H}\mu_i
    \]
    and for the packing dimension
    \[
    \dim_{\rm P}\big\{x\in K\colon
    \underline\chi(x)=\liminf_{i\to\infty}\chi(\mu_i),\,
    \overline\chi(x) = \limsup_{i\to\infty}\chi(\mu_i)\big\}
    \ge \limsup_{i\to\infty} \dim_{\rm H}\mu_i.
    \]
\end{theorem}

\begin{proof}
This proof goes exactly like the proof of~\cite[Theorem 3]{GelPrzRam:10}.
However in our setting, even assuming the weak isolation condition, it is not {\it a priori} clear that the Cantor set $C$ of point $x$ constructed in~\cite{GelPrzRam:10} is in  $K$.
So, we should care about this in the construction.
The points in $C$ are the points whose forward trajectory spends time $n_1$ close to a trajectory in $X_1$ Birkhoff for $\phi_1$ (i.e. with finite time Birkhoff average close to $\int \phi_1\,d\mu_1$), next it goes along  a bridge which joins $X_2$ to $X_1$ (in the forward direction), then it spends time $n_2$ close to a  trajectory in $X_2$ Birkhoff for $\phi_2$, next goes to $X_3$, etc.. The sequence $(n_i)_i$ growths sufficiently fast to guarantee the required properties.

Since the measures $\mu_i$ are non-atomic, the (countable) set $\bigcup_{m=0}^\infty
(\partial \hat I_K)$ has measure 0. Hence, we can start
bridges (backward) from points not in this set and hence these bridges are disjoint from
$\partial \hat I_K$, If
$\cU_i$ consist of  components of length small enough, then
$\bigcup_{m=0}^\infty f^m(C)\subset \hat I_K$. Hence, by the maximality of $K$ in
$\hat I_K$, we have $C\subset K$ , see Subsection~\ref{ss:generalized}.
\end{proof}

To prove lower bounds for, say, $\dim_{\rm H} \cL(\alpha,\beta)$ with $\alpha,\beta \in (\chi_{\inf},\chi_{\sup})$ we consider $X_i$ as in
Proposition~\ref{decaf} (indexed $X_m$ there). Consider also
\[
	\phi_i:=
	\phi_{t_i}|_{X_i}
	\text{ where }t_i \text{ is such that }
	\frac{\partial P(f|_{X_i},\phi_t)}{\partial t}|_{t_i}
	=\begin{cases}\alpha&i\text{ odd},\\\beta&i\text{ even},\end{cases}
\]	
and let $\mu_i$ be the equilibrium for $\phi_i$ with respect to $f|_{X_i}$.

Finally notice that $X_i$ are Cantor sets what is asserted already in Proposition~\ref{decaf} (they cannot be individual periodic orbits because they contain sets corresponding to equilibria $\mu$
in Lemma~\ref{lem:Katok}
of positive entropy).

Now we can subtract from $X_i$ a finite set of small intervals covering $\partial\hat I_K$ which, for example, can be taken to be small
`cylinders' of the form $g_{k_1k_2}\circ
	g_{k_2k_3}\circ \cdots \circ g_{k_{n-1}k_n} (U_{k_n})$, see Subsection~\ref{ss:bridges}. This will not change the topological pressures on $X_i$ much.

\smallskip

If at least one of the numbers $\alpha,\beta$ is in the boundary of the interval $[\chi_{\inf},\chi_{\sup}]$, then
we consider $\mu_i$ for a sequence of numbers $t_i$ corresponding to $\gamma_i\in (\chi_{\inf},\chi_{\sup})$ which accumulates at $\beta$ and at $\alpha$. This proves in particular $\dim_{\rm H} \cL(0) \ge F(0)$ in Theorem~\ref{main2}.

%------------------------------------------------------------------------------------------
\section{On the completeness of the spectrum}\label{s:completeness}
%------------------------------------------------------------------------------------------
Analogously to \cite[Section 3]{GelPrzRam:10} we investigate which numbers can occur at all as upper/lower Lyapunov exponent. The proof of (\ref{empty2}) that
\[
 \left\{ x\in K\colon \overline{\chi}(x)>\chi_{\sup}\right\}
 = \emptyset
\]
in Theorem \ref{main2} is the same as in \cite[Proof of Lemma 5]{GelPrzRam:10}.
The proof of (\ref{empty1}) that
\[
\left\{ x\in K\colon -\infty<\chi(x)<\chi_{\inf}\right\} = \emptyset
 \]
  is more difficult and in part different from the one in \cite{GelPrzRam:10}. This section is devoted to this proof.

%------------------------------------------------------------------------------------------
\subsection{No points with exponent less than $\chi_{\inf}$}\label{nosmallexp}
%------------------------------------------------------------------------------------------

We prove the completeness part of Theorem \ref{main2}, promised at the beginning of this section.

\begin{theorem}\label{chi}
Let $(f,K)\in\sA_+^{\BD}$ and satisfy the weak isolation condition.
Then for every $x\in K$ such that finite $\chi(x)$ exists, we have $\chi(x)\ge\chi_{\inf}$.
\end{theorem}

\begin{proof}
The proof of $\chi(x)\ge 0$ under the hypothesis $\lim_{n\to\infty} \log \,\lvert f'(f^n(x))\rvert=0$ is the same as in~\cite[Lemma 6]{GelPrzRam:10}. However, to prove $\chi(x)\notin [0,\chi_{\inf})$ we cannot use complex methods as in \cite{GelPrzRam:10}, so we provide an alternative proof.

Let us suppose that $\chi(x)\ge 0$.
We can also freely assume that $x$ is not eventually periodic. Indeed, otherwise $\chi(x)=\chi(p)$ for some periodic $p$ and $\chi(p)\ge\chi_{\inf}$ because we can distribute an invariant measure on the orbit of $p$.
If $\chi_{\inf}=0$ then automatically $\chi(x)\ge \chi_{\inf}$ and the assertion holds. Thus, it remains to consider the case $\chi_{\inf}>0$. By the properties listed in Subsection \ref{ss:TCE}, we know that $(f,K)$ satisfies TCE and lacks indifferent periodic points in $K$.
\smallskip

Fix any $n=n_j$ from the sequence $(n_j)_j$ in TCE and consider the intervals
\[
	\widehat W=B(f^n(x),r)
	\quad\text{ and }\quad
	W\eqdef\frac12\diamond \widehat W=B(f^n(x),r/2),
\]	
that is, $\widehat W$ is the $1/2$-scaled neighborhood of $ W$ (recall~\eqref{notescaled}).
Given $k=0,\ldots,n$, denote by $\widehat W_k$  the pull-back of $\widehat W$ for $f^k$ which contains $f^{n-k}(x)$,
\[
	\widehat W_k\eqdef \Comp_{f^{n-k}(x)}f^{-k}(\widehat W),
\]
 and denote by $W_k$ the corresponding pull-back of $W$. Denote by
\[
	\widehat W_k\setminus W_k=L_k\cup R_k
\]	
the union of the left and right interval (margin) in this difference.

We shall now prove  the following fact.

\begin{claim}
	There exists a number $\tau>0$ (which depends only on $f$, but not on $x,n$) such that each $\widehat W_k$ contains an $\tau$-scaled neighborhood of $W_k$.
\end{claim}	
Compare for example \cite[Lemma 1.4]{Prz:98} in the complex case.

\begin{proof}
	Recall the choice of $\delta$ in~\eqref{disttt}. Assume that in TCE the number $r$ is small enough so that in BaShrink all pull-backs of $\widehat W$ for iterates of $f$ have lengths less than $\delta$.
	
For $k\in\{0,\ldots,n\}$ denote by $\tau_k$ the maximal number, for which the claim still holds. By construction, $\tau_0 \geq 1/2$.

Let $0\le k_1<\ldots<k_N\le n-1$  be the subsequence of all consecutive integers $k$ such that $\widehat W_{k}\cap\Crit(f)\ne\emptyset$. By TCE, we have $N\le M$. Note that $M$  depends neither on $x$ nor on $n$.

Let us divide the interval $\{0,\ldots,n-1\}$ into critical times $\{k_i\}$ and intercritical intervals $\{k_i+1,\ldots, k_{i+1}-1\}$, $i=1,\ldots,N-1$. By \eqref{e:slowl}, at any critical time the number $\tau_j$ can decrease at most by a constant factor. By \eqref{disttt}, in an intercritical interval the number $\tau_j$ can decrease at most by a constant factor. As the number of critical times and intercritical intervals are bounded by $M$ and $M+1$, respectively,   the claim follows.
\end{proof}

Notice that $f^n(\partial W_n)\subset \partial W$ (it can happen that both ends of $W_n$ are mapped to one end of $W$). Since $f^n( W_n)$ contains also the point $f^n(x)$, we get $\lvert f^n(W_n)\rvert \ge r/2$.

We will show in the following that there is a constant $a\in(0,1)$ such that
\begin{equation}\label{a}
	a\diamond(f^n(W_n))\cap K \not=\emptyset\,.
\end{equation}
Observe that this condition says that there is a point belonging to K which is ``well inside'' $f^n(W_n)$.
We will show this fact not for all $n_j$ but for some subsequence, slightly corrected if necessary, satisfying \eqref{TCE} (with the same $r$, but the constant $M$ may be altered), and else assuming that $x$ is not pre-periodic.  But let us postpone the proof for a while.

\medskip
\noindent\textbf{1. Proof $\chi(x)\ge \chi_{\rm inf}$ assuming~\eqref{a}.}

\smallskip
\noindent\textbf{1.a) Capture a shadowing periodic orbit.}
Let $a'=(a+1)/2$. Pick some $z\in a\diamond(f^n(W_n))\cap K$.
Recall the function $m_c$ in Definition~\ref{def:strtoptra}. There exists some $m\in\{0,\ldots, m_{c}((a'-a)r/4))\}$ and $y\in f^{-m}(x)$ such that $\lvert y-z\rvert < (a'-a)r/4$. Hence, as $\lvert f^n(W_n)\rvert \ge r/2$, we have $y\in a'\diamond(f^n(W_n))$.

Hence, due to BaShrink,  the pull-back interval $W_{n+m}$ of $W_n$ for $f^m$ which contains $y$, has length smaller than $(1-a')r/4$, for $n$ large enough. Hence, it is strictly contained in $f^n(W_n)$. Therefore, there exists a periodic point $p\in W_{n+m}$ of period $n+m$. By the weak isolation condition, $p\in K$.

\begin{figure}
\begin{minipage}[c]{\linewidth}
\centering
\begin{overpic}[scale=.60]{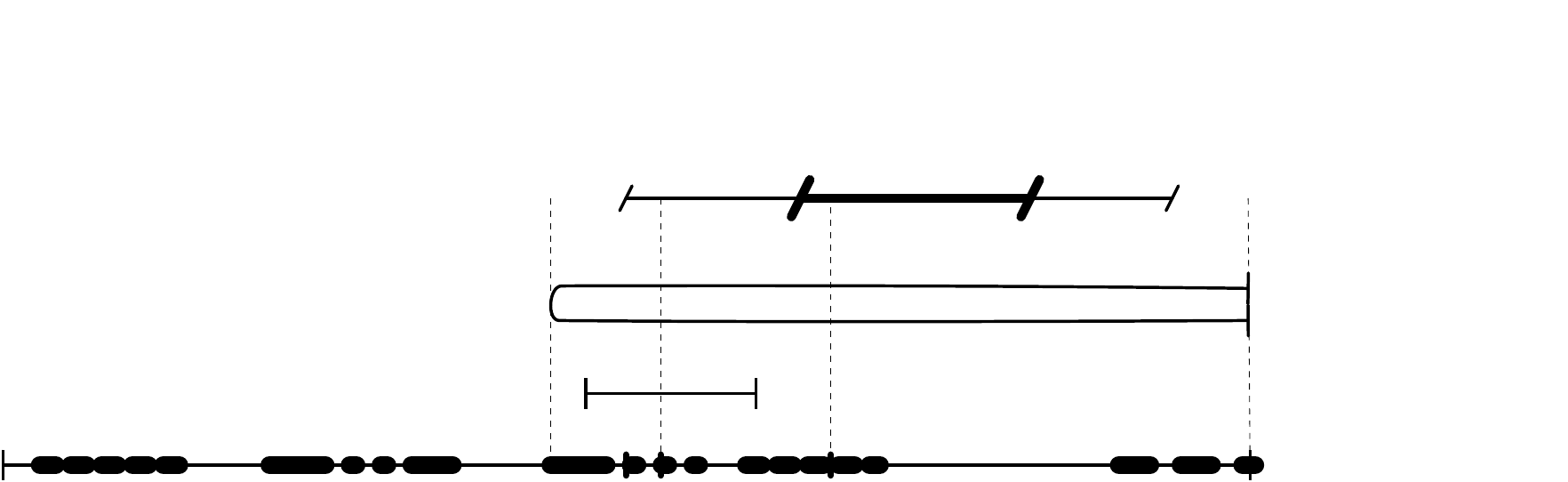}
        \put(3,-4){$K$}
        \put(51,21){$\overbrace{\hspace{1.5cm}}^{} $}
        \put(31,-4){\tiny $f^n(x)$}
        \put(41.5,-4){\tiny $y$}
        \put(52.5,-4){\tiny $z$}
        \put(90,10){\small$f^n(W_n)$}
        \put(90,5){\small$W_{n+m}$}
        \put(90,17){\small$a'\diamond f^n(W_n)$}
        \put(52,26){$a\diamond f^n(W_n)$}
 \end{overpic}
\caption{}
\label{large-scale}
\end{minipage}
\end{figure}

\smallskip
\noindent\textbf{1.b) Comparison of derivatives.}
Let us now estimate $|(f^{n+m})'(p)|$  from above.  Denote by $\Lip$ the Lipschitz constant of $f|_K$.
First observe  existence and finiteness of $\chi(x)$ imply that $ \lim_{n\to\infty}\frac1n \log |f'(f^n(x))|=0$. Hence, given $\xi>0$, for large enough $n$ we have $\lvert (f^n)'(x)\rvert\le \exp(\xi Nn)$.
For $n$ large enough, we have
\begin{multline*}
	|(f^{n+m})'(p)|=\\
	|(f^{m})'(p)|\cdot \prod_{i=1}^N |(f^{k_{i}-k_{i-1}-1})'(f^{m+n-k_i+1}(p))
	\cdot	\prod_{i=1}^N |f'(f^{m+n-k_i}(p)| \,.
\end{multline*}
Let now $\tau$ be as in the Claim and $C(\tau)$ as in~\eqref{disttt}. We can further estimate
\begin{equation}\label{dist-est}\begin{split}
|(f^{n+m})'(p)|&\le
\Lip^{m+N} C(\tau)^N  \cdot \prod_{i=1}^N |(f^{k_{i}-k_{i-1}-1})'(f^{n-k_i+1}(x))|          \\
&\le \Lip^{m+N} C(\tau)^N  \cdot |(f^n)'(x)|\exp (\xi N n)\,.
\end{split}\end{equation}
We conclude
$$
\frac1{n+m}\log |(f^{n+m})'(p)|\le N\xi + \lim_{n\to\infty} \frac1n \log |(f^n)'(x)|
	+O(\frac1n)\,.
$$
Considering the $f$-invariant probability measure
\begin{equation}\label{equidistribute}
\mu_{n,\xi}
:=\frac{1}{n+m}\sum_{k=0}^{n+m-1} \delta_{f^k(p)},
\end{equation}
its Lyapunov exponent is $\chi(p)$.
Hence, letting $\xi\to 0$ and $n\to\infty$, we obtain a sequence of measures $\mu_{n,\xi}$ such that $\inf_{n,\xi}\chi(\mu_{n,\xi})\le \chi(x)$.
This proves the assertion  assuming that~\eqref{a} holds.

\medskip

What remains to prove is~\eqref{a}.
\medskip

\noindent
\textbf{2. Proof of~\eqref{a}, changing $n$ and $M$ if necessary.}
Now let us prove the postponed (\ref{a}). We shall specify the constant $a$ later on.
For any $\rho>0$ define
\[
	\partial_\rho
	:=\{z\in K\colon (z,z+\rho)\cap K=\emptyset \text{ or }
			(z-\rho,z)\cap K=\emptyset\}\,.
\]	

Considering $a>1/2$, set $b=2a-1$ (hence $1-b=2(1-a)$) and $\rho=br/2$. Assume first that
\begin{equation}\label{ball}
f^n(x)\notin B(\partial_\rho, (1-b)r/4).
\end{equation}
In words this means that $f^n(x)\in K$ is not closer than $(1-b)r/4$ to a ``big hole in K" which is of length at least $\rho$. Then $K$ intersects both intervals
$R=b\diamond (f^n(x),f^n(x)+ r/2)$ and
$L=b \diamond (f^n(x)-r/2,f^n(x))$.
Since $f^n(W_n)$ covers either $(f^n(x),f^n(x)+ r/2)$ or $(f^n(x)-r/2,f^n(x))$, then (\ref{a})
follows. Indeed, notice that $(f^n(x), f^n(x)+r/2)$ is a $(1-b)r/4$-scaled neighborhood of $R$ (analogously for $L$), so
$f^n(W_n)$ contains a $(1-a)r/4$-scaled neighborhood of $R$ (or $L$).

\medskip

Notice that if $0<r_1<r_2$ then $\partial_{r_2}\subset \partial_{r_1}$.
If $r/3\le \rho$ and thus $b\ge 2/3$, then~\eqref{ball} and thus~\eqref{a}
  hold for all $n$ such that
\begin{equation}\label{ball2}
f^n(x)\notin B(\partial_{r/3}, (1-b)r/4)=:B.
\end{equation}
Notice that our subscript at $\partial$ is independent of $b$ now.
Notice that the set $\partial_{r/3}$ is finite, with the number of elements depending only on $r$. This set can contain periodic points, which must be hyperbolic repelling. Let $\theta>0$ be a constant such that for each periodic point $q\in\partial_{r/3}$ with period $m(q)$, for every $y\in B(q,\theta)$ we have $|(f^{m(q)})'(y)|>1$.

\medskip

Let us consider the upper density
\[
	\overline d:=\limsup_{m\to\infty} \frac1m\#\{0< k \le m\colon f^k(x)\notin B\}\,.
\]	

\medskip
\noindent
\textbf{2.a) The case  $\overline d>0$.}
Taking $P$ in the definition of TCE small enough that $1/P + \overline d>1$, possibly passing to a subsequence, we can assume that  there is a strictly increasing sequence $(n_j)_j$ such that TCE and~\eqref{ball2} hold. Hence, this shows~\eqref{a} and finishes the proof of Theorem \ref{chi} in this case.

\medskip
\noindent
\textbf{2.b) The case  $\overline d=0$.}
Set
\[
	\alpha:={\rm{dist}}( f(\partial_{r/3})\setminus \partial_{r/3}, \partial_{r/3})>0\,.
\]	
Then for each $z\in \partial_{r/3}\setminus f^{-1}(\partial_{r/3})$
\begin{equation}
	f\big(B(z,\Lip^{-1}\alpha/2)\big)\cap B(\partial_{r/3},\alpha/2)=\emptyset\,.
\end{equation}
Recall that $\Lip$ denotes the Lipschitz constant of $f|_K$.
Denote
\[
	\Theta:=\min\{|w-w'|\colon w\not=w', \, w,w'\in  \partial_{r/3}\}\,.
\]	
We set
\[
	\alpha':=\min\Big\{\frac{\alpha}{2\Lip}, \frac\Theta2, \frac\theta2\Big\}\,.
\]	
Let $b$ be defined by $(1-b)r/4=\alpha'$. This in turn eventually specifies $a$.
Finally define, recalling the definition of $\varkappa(\cdot)$ in BaShrink,
\begin{equation}\label{M'}
	M':= M + \#\partial_{r/3}+\varkappa(2r,\alpha').
\end{equation}

Recall again that $f^{k}(x)\notin B$ implies (\ref{ball2}) and hence (\ref{a}). As we have $\overline d=0$, for all $m$ large enough an arbitrarily large proportion of $\{1,2,\ldots,m\}$ is covered by  blocks $A_j$ of consecutive integers $k$ so that $f^k(x)\in B(z_k,\alpha')$ for an apropriate $z_k\in \partial_{r/3}$, such
$z_k$ is unique by the choice of $\alpha'$. Moreover, it also follows from the choice of $\alpha'$ that for integers $k,k+1$ belonging to one block $ A_j$ we have
\begin{equation}\label{end-orbit}
f(z_k)=z_{k+1}.
\end{equation}
Hence for each block $A_j=\{\ell_j,\ell_j+1,\ldots,\ell'_j\}$, if $\#A_j>\#\partial_{r/3}$ then the point $z_{\ell_j}$ is eventually periodic, namely there exists a periodic point $q\in \partial_{r/3}$ with period not greater than $\#A_j-1$, which we denote by $t_j$.\footnote{Notice that we can assume that $z_{\ell_j}\in f^{\ell_j}(W_{\ell_j})$. Otherwise $f^{\ell_j}(W_{\ell_j})$ would cover the other side of $f^{\ell_j}(x)$ than the one containing $z_{\ell_j}$ and hence~\eqref{a} would hold immediately.
However we will not use this observation and assumption later on.}

Note that we can exclude the case $f^{\ell_j}(x)=z_{\ell_j}$, because then $x$ would be eventually periodic (we considered this case in the beginning of the proof of Theorem~\ref{chi}).
Increasing $m$ if necessary, we can assume that the last block ends at $m$ and that all the  blocks $A_j$ have been chosen maximal, that is, $f^{\ell'_j+1}(x)\notin B$ for all $j$. Observe that each block must indeed be finite by definition of $\alpha'$.

So (\ref{ball2}) and hence (\ref{a}) holds for $n_j=\ell'_j+1$ for each $j$. What remains to be checked is that this  number satisfies also (\ref{TCE}) in TCE.

We will prove that~\eqref{TCE} holds for $n_j=\ell'_j+1$ with the constant $M'$, provided it holds for some $m_j \in A_j$ with the constant $M$.
Indeed,  the pull-backs $\Comp_{f^k(x)} f^{-(\ell'_j+1-k)}  B(f^{\ell'_j+1}(x),r)$ have lengths smaller than $\alpha'$ for all $k\ge \varkappa(2r,\alpha')$ by the definition of the latter constant in BaShrink, and hence their consecutive pull-backs do not capture critical points for $\ell_j+ t_j\le k \le \ell'_j-\varkappa(2r,\alpha')$;
they can capture critical points only for $k\in\{m_j,\ldots, m_j + t_j-1\}$ and the number of such $k$'s is bounded by $\#\partial_{r/3}$.
In conclusion, when passing from $m_j$ to $\ell'_j+1$, TCE still holds with $M$ replaced by $M'$ defined in \eqref{M'}.

Finally recall that the union of the blocks $A_j$ covers an arbitrarily large proportion of $\{0,\ldots,m\}$ for $m$ large enough,
hence there exist arbitrarily large $m_j$ belonging to some $A_j$ satisfying also TCE, hence there exist arbitrarily large $n_j=\ell'_j+1$ satisfying TCE with $M'$, and simultaneously satisfying (\ref{a}).
\smallskip

This finishes the proof of the theorem.
\end{proof}

%----------------------------------------------------------------------------------------------
\section{Strong upper bound for irregular part}\label{strong}
%----------------------------------------------------------------------------------------------

We now shall prove Theorem \ref{main3} and Theorem~\ref{strangestrange} from Subsection~\ref{main}. We start with three easy technical lemmas.

\begin{lemma}\label{geometric1}
	Let $\Phi\colon\R_+\to \R$ be a continuous function, and $L>0$ such that for each $0\le t_1<t_2$ we have $\Phi(t_2)-\Phi(t_1)\le L(t_2-t_1)$. Let $0\le\sigma\le q_1<q_2$ be real numbers. Let $C:=L^{-1}(q_2-q_1)$ and denote
\[	
	\Phi^\sigma(t)\eqdef \sup_{s<t} \left(\Phi(s)+\sigma (t-s)\right)
\]
and
\[
	H_{\R}:=\{x\in \R_+\colon \Phi(x)=\Phi^\sigma(x)\}
	\quad \text{and}\quad
	H_{\N}:=\{n\in \N\colon \Phi(n)=\Phi^\sigma(n)\}\,.
\]
Let $\tau_1<\tau_2$ be positive numbers  satisfying $\Phi^\sigma(\tau_1)=q_1\tau_1$, $\Phi^\sigma(\tau_2)=q_2\tau_2$.
 Then, with $\ell$ denoting the length measure, we have
\begin{equation}\label{denseR}
\frac{1}{\tau_2}\ell(H_{\R}\cap [\tau_1,\tau_2]) \ge C\,.
\end{equation}
If $\Phi$ is affine between each pair of consecutive integers and $\tau_2$ is an integer then
\begin{equation}\label{denseN}
\frac{1}{\tau_2}\#(H_{\N}\cap [\tau_1,\tau_2]) \ge C.
\end{equation}
\end{lemma}

\begin{proof}
Notice that the function $\Phi^\sigma$ is affine on certain open intervals, on each such interval its graph has slope $\sigma$, and these intervals are all in the complement of $H_\bR$.
Since by our assumption $\sigma<L$,  the function $\Phi^\sigma$ is monotone increasing Lipschitz continuous with Lipschitz constant $L$. Hence
\[
	\Phi^\sigma(\tau_2) - \Phi^\sigma(\tau_1)
	\le L \cdot \ell \big(H_{\R}\cap [\tau_1,\tau_2]\big)
		+ \sigma \cdot \ell \big([\tau_1,\tau_2]\setminus H_{\R}\big).
\]
On the other hand
\[
	\Phi^\sigma(\tau_2) - \Phi^\sigma(\tau_1)
	= q_2\tau_2 - q_1\tau_1 = (q_2 - q_1)\tau_2 + q_1 (\tau_2 - \tau_1).
\]
From $\sigma\le q_1$ we get
\[
 	(q_2 - q_1)\tau_2 \le L \cdot \ell \big(H_{\R}\cap [\tau_1,\tau_2]\big)
\]
 and (\ref{denseR}) follows.

Finally notice that if $\Phi$ is affine between each two consecutive integers, then $t\in H_{\R}$ implies $\lceil t\rceil\in H_{\R}$ where $\lceil t\rceil$ denotes the smallest  integer $\ge t$.
Hence $\#(H_{\N}\cap [\tau_1,\tau_2]$ for integer $\tau_2$ is not less than the number of intervals of the form $(m,m+1]$ intersecting $H_{\R}\cap [\tau_1,\tau_2]$. Hence
\[
	\#(H_{\N}\cap [\tau_1,\tau_2])
	\ge \ell (H_{\R}\cap (\tau_1,\tau_2))\ge L^{-1}(q_2-q_1)\tau_2\,,
\]
which proves~\eqref{denseN}.
\end{proof}

Below we will apply this lemma to numbers $0\le \sigma \le q_1<q_2<L$ so that
\[
	\Phi^\sigma(\tau_1) = q_1 \tau_1,\quad
	q_1 t \le \Phi^\sigma(t) \le q_2t \quad\forall t\in [\tau_1,\tau_2],\quad
	\Phi^\sigma(\tau_2) = q_2 \tau_2\,.
\]	
We call such an interval $[\tau_1,\tau_2]$ a $(q_1,q_2)$-{\it interval} for $\Phi^\sigma$. If it is the minimal interval with this property, we call it a \emph{$(q_1,q_2)$-crossing interval} (compare Figure~\ref{crossing}).

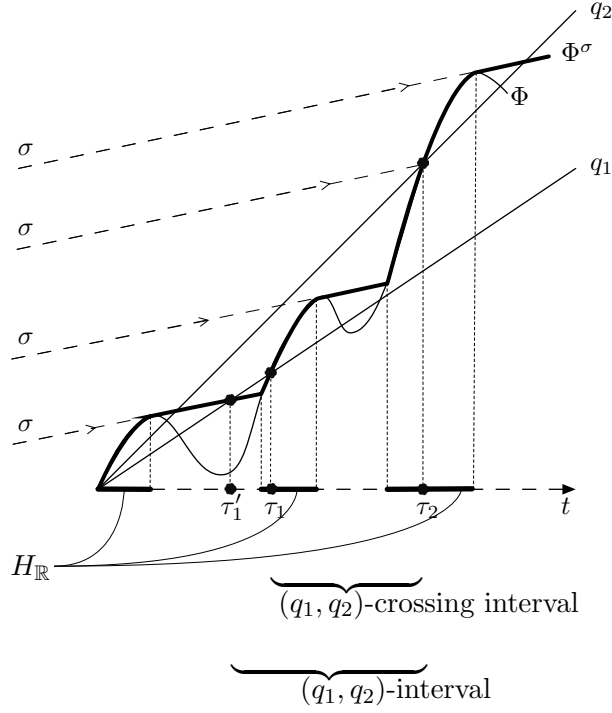
\begin{figure}
\begin{minipage}[c]{\linewidth}
\centering
\begin{overpic}[scale=.5]{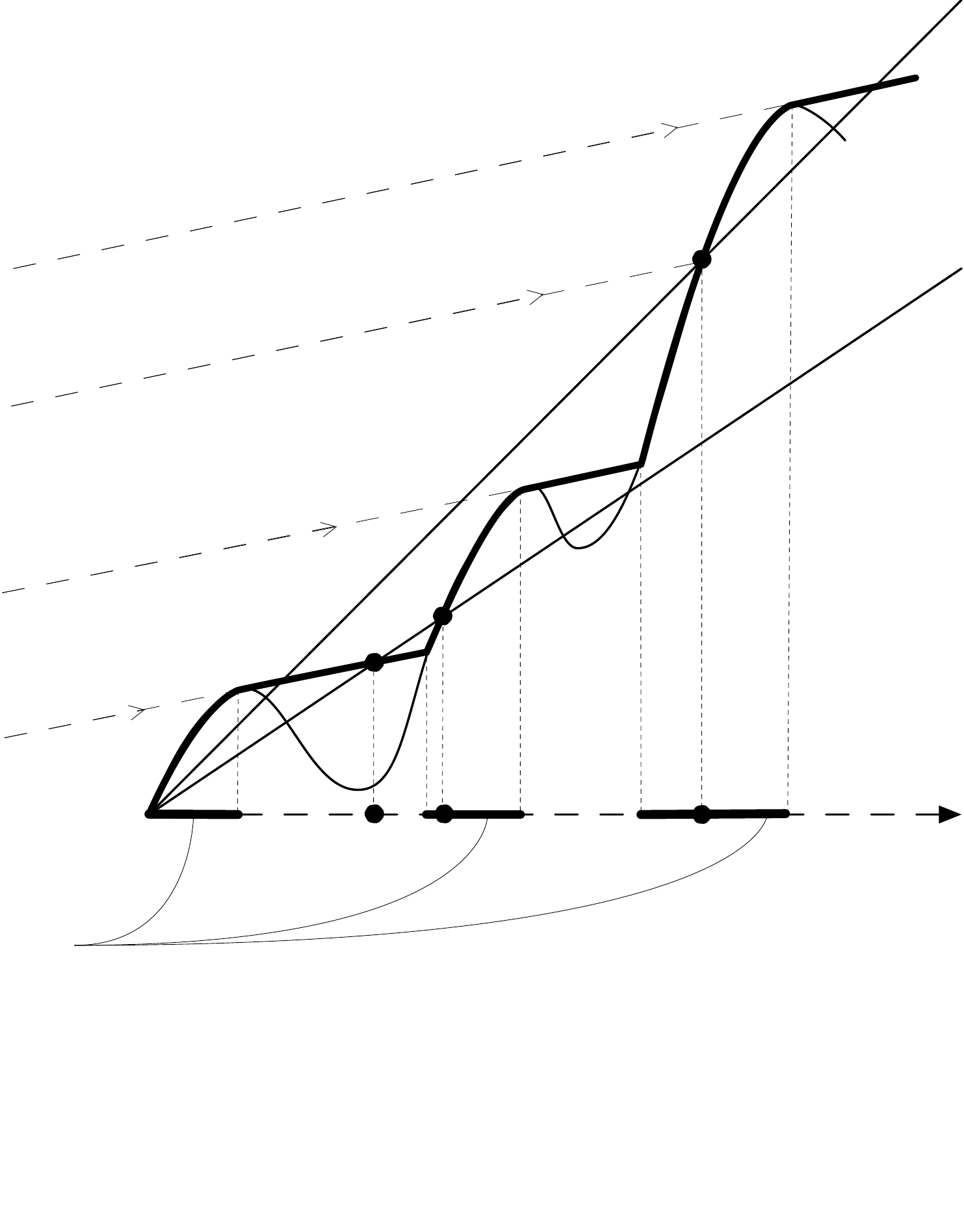}
        \put(56,30.5){\small $\tau_2$}
        \put(35,30.5){\small $\tau_1$}
        \put(29,30.5){\small $\tau_1'$}
        \put(76,30.5){\small $t$}
        \put(80,99.5){\small $q_2$}
        \put(80,78){\small $q_1$}
        \put(76,93){\small $\Phi^\sigma$}
        \put(69,87){\small $\Phi$}
        \put(1,80){\small $\sigma$}
        \put(1,69){\small $\sigma$}
        \put(1,54){\small $\sigma$}
        \put(1,42){\small $\sigma$}
        \put(0,22){$H_{\mathbb R}$}
        \put(36,22){$\underbrace{\hspace{2cm}}^{}$}
        \put(37,17){$(q_1,q_2)$-crossing interval}
        \put(30.5,10){$\underbrace{\hspace{2.6cm}}^{}$}
        \put(40,5){$(q_1,q_2)$-interval}
 \end{overpic}
\caption{$(q_1,q_2)$- and crossing intervals}
\label{crossing}
\end{minipage}
\end{figure}

\begin{lemma}
\label{lemmaalphasharp}
	Let $\Phi, L$ be as in Lemma~\ref{geometric1}, $\beta=\limsup_{t\to\infty}\Phi(t)/t$, and $\alpha=\liminf_{t\to\infty}\Phi(t)/t$ (we allow $\alpha$ to be negative). Let
\[
L_0 := \limsup_{N\to \infty} \sup_{k} \frac 1N \lvert\Phi(k+N)-\Phi(k)\rvert.
\]
Then for $\sigma\in[0,\beta)$ we have
\begin{equation}
	\liminf_{t\to\infty}\frac{\Phi^\sigma(t)}{t}
	\le \frac{\beta+\sigma(\beta-\alpha)/(L_0-\sigma)}{1+(\beta-\alpha)/(L_0-\sigma)}.
\end{equation}
In particular, for $\sigma=0$ we have
\[
	\liminf_{t\to\infty}\frac{\Phi^0(t)}{t}
	\le \frac{\beta}{1+(\beta-\alpha)/L_0}\,.
\]	
\end{lemma}
Applying this lemma we shall consider $\sigma>0$ arbitrarily close to 0, so we will be allowed to use the latter formula.

\begin{proof}
We always have
\begin{equation} \label{eqn:help}
\alpha < \alpha^\sharp(\sigma) := \frac {\beta+\sigma(\beta-\alpha)/(L_0-\sigma)} {1+(\beta-\alpha)/(L_0-\sigma)}.
\end{equation}
Fix a small $\varepsilon < \min\{\beta - \sigma, (\alpha^\sharp(\sigma)-\alpha)/2\}$.
Let $(n_i)_i$ be a sequence of times for which $\Phi(n_i) \leq (\alpha + \varepsilon)n_i$. We can freely assume (passing to a subsequence if necessarily) that $\Phi^\sigma(n_i) \geq \Phi(n_i)$ for every $i$, otherwise the assertion follows immediately. Hence, each $n_i$ lies on plateaux of $\Phi^\sigma$, denote by $k_i$ and $\ell_i$ the beginning and the end of this plateaux. We have
\[
	\Phi(k_i) + (\ell_i - k_i)\sigma = \Phi(\ell_i)
\]
and, for $i$ large enough,
\[
	\Phi(k_i) \leq (\beta + \varepsilon) k_i\,.
\]

We can freely assume that $\Phi^\sigma(\ell_i) \geq (\alpha^\sharp(\sigma) - \varepsilon)\ell_i$, otherwise the assertion follows. As
\[
	\Phi(\ell_i) \leq \Phi(n_i) + (\ell_i-n_i)L,
\]
this implies that
\[
	\frac {\ell_i}  {n_i} \geq
	\frac {L - \alpha - \varepsilon} {L - \alpha^\sharp(\sigma)+\varepsilon} >1,
\]
and hence for $i$ large enough
\[
	\Phi(\ell_i) \leq
	\Phi(n_i) + (L_0 + \varepsilon)(\ell_i-n_i) \leq
	(\alpha + \varepsilon)k_i + L_0(\ell_i-k_i).
\]
Hence,
\[
	\frac {\ell_i} {k_i}
	\geq 1+ \frac 1 {L_0+\varepsilon-\sigma}
		\left(\frac {\Phi(k_i)} {k_i} - (\alpha+\varepsilon)\right)
\]
and
\[
	\frac {\Phi(\ell_i)} {k_i} - \sigma \frac  {\ell_i} {k_i} = \frac {\Phi(k_i)} {k_i} - \sigma.
\]

Dividing side by side we get

\[
\frac {\Phi(\ell_i)} {\ell_i} - \sigma \leq \frac {\left(\frac {\Phi(k_i)} {k_i} - \sigma\right)(L_0+\varepsilon-\sigma)} {\frac {\Phi(k_i)} {k_i} + L_0 - \sigma - \alpha}.
\]
The function $x\to (x-\sigma)/(x-\sigma + L_0-\alpha)$ is increasing. As $\Phi(k_i)/k_i \leq \beta + \varepsilon$,

\[
\frac {\Phi(\ell_i)} {\ell_i} \leq \sigma + \frac {(\beta + \varepsilon - \sigma)(L_0+\varepsilon-\sigma)} {L_0-\sigma+\beta - \alpha+\varepsilon}
\]
Passing with $\varepsilon$ to 0 we get the assertion.
\end{proof}

In the remaining section we consider the following piecewise affine function
\begin{equation}\label{def:Phii}
	\Phi(t):=\begin{cases}
		\log\,\lvert (f^\ell)'(x)\rvert&\text{ if }t=\ell\in \bN,\\
		(\ell+1-t)\cdot\Phi(\ell)+ (t-\ell)\cdot \Phi(\ell+1)
		& \text{ if }\ell<t<\ell+1.
		\end{cases}
\end{equation}
Observe that by~\cite[Proposition 4.7]{PrzRiv:13} when considering the orbit of a point $x$ and the above defined function then we have $L_0\le\chi_{\sup}$.
Hence, in the following for this function we let
\[
	\alpha^\sharp=\alpha^\sharp(\beta) := \frac{\beta}{1+(\beta-\alpha)/\chi_{\sup}}\,.
\]

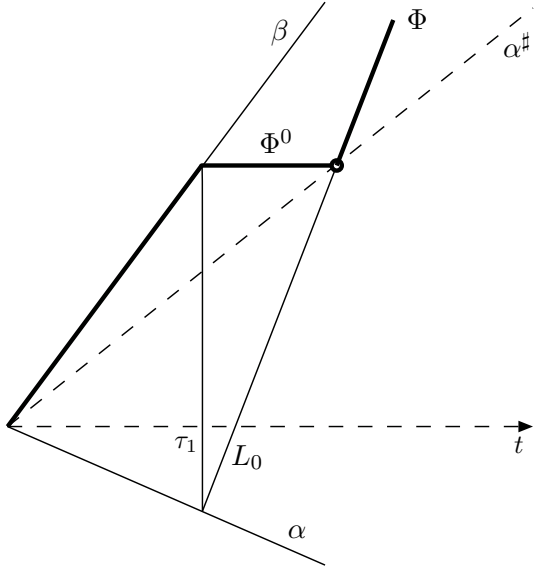
\begin{figure}
\begin{minipage}[c]{\linewidth}
\centering
\begin{overpic}[scale=.55]{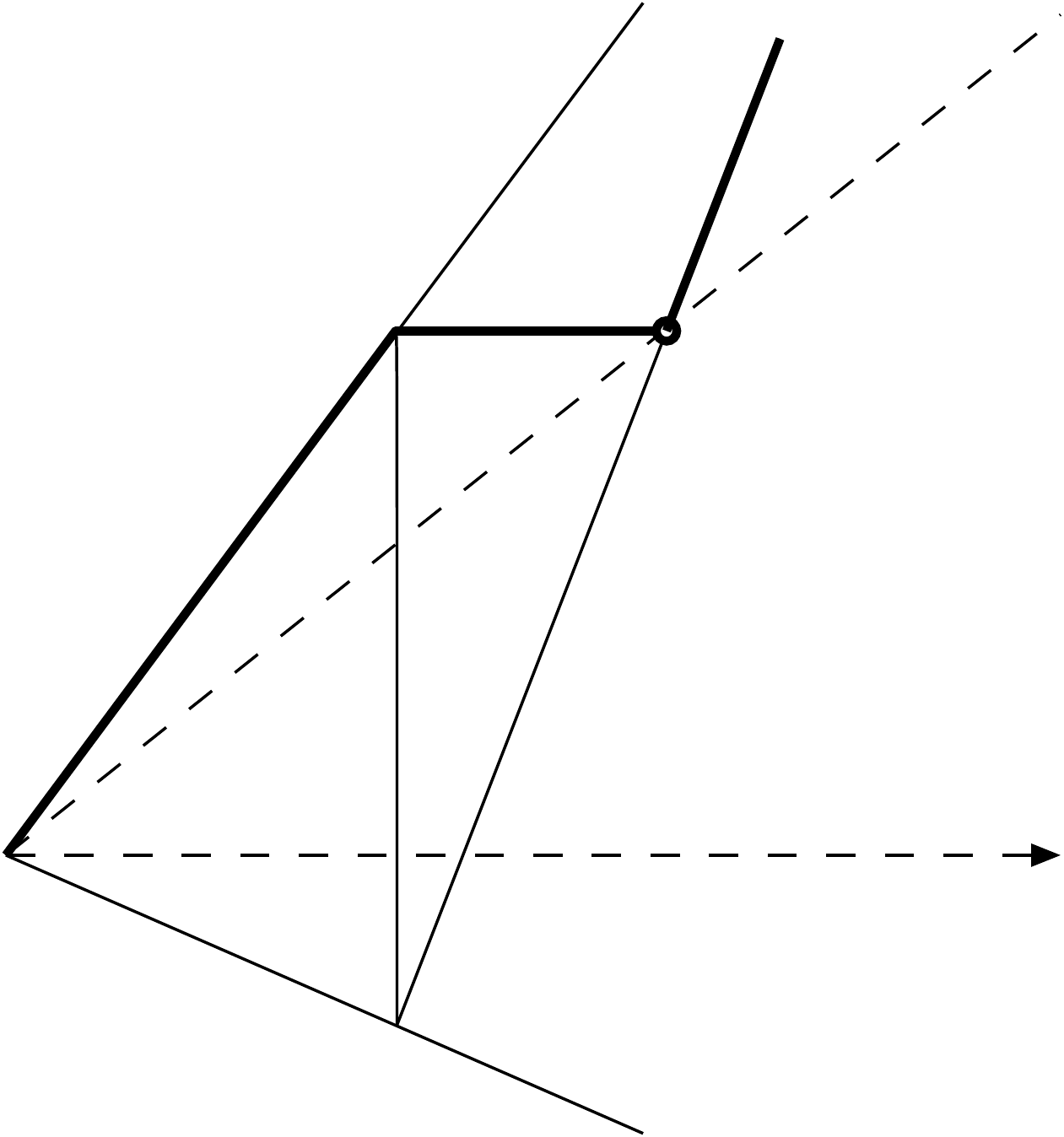}
        \put(90,20){\small $t$}
        \put(30,21){\small $\tau_1$}
        \put(50,5){$\alpha$}
        \put(88,90){$\alpha^\sharp$}
        \put(47,93){$\beta$}
        \put(71,95){$\Phi$}
%        \put(36,36){$h$}
        \put(40,18){$L_0$}
        \put(45,73){$\Phi^0$}
%        \put(39,66){$d$}
%        \put(63,69){$P$}
 \end{overpic}
\caption{Sketch of the proof of Lemma~\ref{lemmaalphasharp} ($\sigma=0$)}
\label{alphasharp}
\end{minipage}
\end{figure}

\begin{lemma}\label{combinatorial}
For every $d\in(0,1]$ and every positive integer $k$ there exist an integer $m=m(d,k)$ and a number $n_0$ such that for every $n\ge n_0$ and every $J\subset \{1,\ldots,n\}$ such that $\#J \ge d\,n$ there exist numbers  $a_1<a_2<\cdots<a_k$ in $J$ such that $a_k-a_1\le m$.
\end{lemma}

\begin{proof}
Let $m= \lfloor 2k/d \rfloor +1$, where $\lfloor \cdot \rfloor$ denotes the integer part.
Then for $n$ large enough
\[
	\sum_{j\in J}\big\lvert [j,j+m]\cap [1,n]\big\rvert
	\ge m(\#J-m)
	\ge m d n - m^2\ge \frac12 m n d.
\]
Hence, there exists $i\in \{0,\ldots,n\}$ such that
\[
	A(i):=\big\{j\in J\colon i\in [j,j+m]\cap[1,n]\big\}
	\quad\text{ satisfies }\quad
	\#A(i)\ge \frac12m d\,.
\]
We now let $a_1<a_2<\cdots<a_k$ be the consecutive indices $j$ in $A(i)$.
\end{proof}

The following lemma will provide a substantial technical ingredient to prove Theorem~\ref{main3}

\begin{lemma}\label{Levin}
	Let $(f,K)\in \sA$. Given $\alpha\le\beta$ with $\beta>0$, for every $q>\sigma>0$ such that $\alpha^\sharp\le q\le \beta$ and every $\e>0$, there exist $M\ge 0$, $r>0$, $\e'>0$ such that for every $x\in \cL(\alpha,\beta)$, there exists a subset $H$  of $\N$ of  upper density
 \[
 	\overline d(H):=\limsup_{n\to 0}\frac{\# (H\cap [1,n])}{n} > 0
\]	
such that for every $n\in H$ and every $k=0,\ldots,n$ for the pull-back $W_k$ of $B(f^n(x),r)$ for $f^k$ which contains $f^{n-k}(x)$ the following properties hold
\begin{equation}\label{Shrink1a}
	\lvert (f^k)'(f^{n-k}(x))\rvert^{-1}\le e^{-k\sigma},
\end{equation}
\begin{equation}\label{Shrink1}
	\lvert W_k\rvert\le e^{-k(\sigma-\e)},
\end{equation}
\begin{equation}\label{Shrink1b}
	n(q-\e)\le \log \,\lvert (f^n)'(x)\rvert \le n(q+\e),
\end{equation}
\begin{equation}\label{Shrink2}
        \#\{k\colon 0 \le k < n, W_k \cap\Crit(f) \not= \emptyset \} \le M.
\end{equation}

Moreover, for $\Phi$ defined in~\eqref{def:Phii}, if $\alpha<\beta$ then there exists $n^{\rm up}>n$ such that
\begin{equation}\label{auxiliary}
	\frac{\Phi(n^{\rm up})}{n^{\rm up}}
	\ge \frac{\Phi(n)}{n}+\e'
	\,\,\text{ and }\quad
	\frac{\Phi(t)}{t}\ge \frac{\Phi(n)}{n}-\e
	\,\,\text{ for all }
	n<t<n^{\rm up}.
\end{equation}
If $\alpha=\beta$ then only the second inequality holds for all $t>n$.
\end{lemma}

Observe that~\eqref{Shrink2} is as in~\eqref{TCE} in TCE, however we conclude this property for the subset $\cL(\alpha,\beta)\subset K$ only.

The property (\ref{auxiliary}) is technical, not needed in the case $K=\hat I_K$.
The possibility to find a positive upper density set $H$ of integers for which Pliss' type properties
(\ref{Shrink1a}), (\ref{Shrink1}) and simultaneously (\ref{Shrink2}) hold,  has been recently noticed in \cite{Levin}, using `shadows' from \cite{PrzRo:98}.

\begin{proof}
\textbf{Step 1.}
Let $H_1$ be the set of all Pliss hyperbolic times for $x$ with exponent $\sigma$. As $\sigma<\overline\chi(x)=\beta$, by the Pliss Lemma~\ref{lemma-pliss} (see Subsection~\ref{Telescope constructions}) $H_1$ has positive upper density $\overline d(H_1)$. Now the property~\eqref{Shrink1a} follows by definition of the Pliss hyperbolic times and~\eqref{Shrink1} holds by Telescope Lemma~\ref{p.telescope}.

To prove (\ref{Shrink1b}) and (\ref{auxiliary}) we can assume $\alpha<\beta$, since for $\alpha=q=\beta$ these properties hold immediately.

Let $q_1<q'_2$ and $[q_1,q'_2]\subset (q-\e,q+\e)\cap (\alpha^\sharp,\beta)$. Assume also $\sigma<q_1$. Let $q_2=\frac{q_1+q'_2}{2}$ (hence $q_2-q_1\le\e$). We define
$\e':=q'_2-q_2$.

To apply Lemma~\ref{geometric1}, let us consider the piecewise affine function $\Phi$ defined in~\eqref{def:Phii} which extends $\Phi$ defined above to $\bR^+$.
Now let $H_{\N}$ be the set from Lemma~\ref{geometric1} applied to this function.
Recall that every integer $n\in H_\bN$ satisfies $\Phi^\sigma(n)=\Phi(n)$ and is a Pliss hyperbolic time for $x$ with exponent $\sigma$, that is, $H_\bN\subset H_1$.
Continue to consider the set $H_2$ being the set $H_\bN$ intersected with the union of all the $(q_1,q_2)$-crossing intervals (recall the definition right after the proof of Lemma~\ref{geometric1}).
If $[\tau_1,\tau_2]$ is now such a crossing interval then for $n=\tau_2$ we have
\[
	\frac{\#(H_{\N}\cap [\tau_1,n])}{n}\ge C>0\,.
\]	
Moreover, by
\[
	\liminf_{n\to\infty} \frac1n \log\,\lvert(f^n)'(x)\rvert \le \alpha^\sharp <q_1
	\,\text{ and }\,
	q'_2<\limsup_{n\to\infty}\frac1n\log\,\lvert(f^n)'(x)\rvert=\beta
\]	
there are infinitely many crossing intervals, hence infinitely many such $n$'s. Thus, the set $H_2$ has positive upper density and also satisfies~\eqref{Shrink1a} and~\eqref{Shrink1b}.
The crossing property implies easily (\ref{auxiliary}).
\medskip

\noindent\textbf{Step 2.}
Now we shall concentrate on proving (\ref{Shrink2}), further restricting $H_2$.
By \cite[Lemma 3.1]{PrzRiv:13} (see also \cite{NowPrz:98}, or~\cite{DPU:96} in the holomorphic case) there is a constant $Q>0$ such that for every $x\in K$ for every $n$ and an arbitrary $c\in\Crit(f)$ we have
\begin{equation}\label{average distance}
	{\sum_{j=0}^n}{}'  -\log\,\lvert f^j(x)-c\rvert \le Qn,
\end{equation}
here $\Sigma'$ means that in the sum we omit  an index $j$ of smallest distance  $\lvert f^j(x)-c\rvert$ (only one if there is more than one such $j$).
For normalization (to have only non-negative numbers in
\eqref{average distance}) we assume $K\subset [0,1]$; this influences only the constant $Q$.

We continue, using the ideas from \cite{PrzRo:98} following the proof of the fact that ExpShrink implies TCE (see also \cite{NowPrz:98}). Given $j\ge1$, let
\[
	a(j):=-\log {\rm{dist}}(f^j(x),\Crit)\,.
\]	
Fix a number $\varkappa>1/(\sigma-\e)$.
For each $j$ consider the interval (``shadow'') $S_j:=(j, j+\varkappa a(j)]$.
Denote by $\mathds 1_{S_j}$ the characteristic function of $S_j$.
Then, by~\eqref{average distance} we have
\[
\sum_{k=1}^n  \sum_{j=0}^n  \mathds 1_{S_j}(k)
\le\sum_{j=0}^n \varkappa a(j)
\le \#\Crit \cdot n +  \varkappa Qn
= n(\#\Crit+\varkappa Q )=:nQ_\varkappa,
\]
where in the term $\#\Crit\cdot n$ takes care of the indices which have been omitted in the sum in~\eqref{average distance}.
So, given an arbitrary $M>0$ the cardinality $n_M$ of the set
of $k\in\{1,\ldots, n\}$ covered by more than $M$ shadows is
less than $n(\#\Crit+\varkappa Q )/M$.
Hence, there is $\eta=\eta(M,\varkappa,Q)$ such that
$$
1-n_M/n \ge 1-(\#\Crit+\varkappa Q )/M=\eta.
$$
In particular the lower density satisfies
$$
\liminf_{n\to\infty} \big(1- \frac{n_M}{n}\big) \ge \eta,
$$
and choosing $M$ sufficiently large $\eta$ is arbitrarily close to 1.

Let now $H$ be the set of numbers in $H_2$ which do not belong to more than $M$ shadows.
Thus, its upper density is still positive for sufficiently large $M$.

What remains to show is that every $n\in H$ satisfies~\eqref{Shrink2}.
If $n\in H$ and $n\notin S_j$ for $j<n$ then, by definition $n-j> \varkappa a_j$ and hence
\[
	\exp( -(n-j)/\varkappa) < {\rm{dist}}(f^j(x),\Crit)\,.
\]
Thus, by~\eqref{Shrink1} we obtain $\lvert W_j\rvert\le e^{-(n-j)(\sigma-\e)}$. Since $\varkappa>1/(\sigma-\e)$ we conclude that $W_j$ is disjoint from $\Crit$. Thus the set $H$ defined above satisfies (\ref{Shrink2}).
\end{proof}

We remark that the used term ``shadow'' origins in its analogy to the ``shadows'' in the proof of existence of Pliss hyperbolic times. Indeed, compare Figure~\ref{crossing} where the shadows correspond to the lines with slope $\sigma$ to catch hyperbolic times with such exponent.

The following proposition is the key step in the proof of Theorem \ref{main3}.
In the remainder of the section we will assume that $(f,K)\in \sA^{\BD}_+$ is non-exceptional and we will consider the following family of measures $\mu_t$.
In the case that there are no indifferent periodic orbits in $K$, then for every $t<t_+$ we  denote by $\mu_{t}$ the unique (nonatomic) $e^{P(t)}\lvert f'\rvert^t$-conformal measure which is well-defined by \cite[Theorem A]{PrzRiv:13}.
For $t<0$ if such measure does not exist we consider a CaS measure $\mu^*_{t}$ instead, positive on open sets (see Subsection~\ref{ss:conformal} and Subsection~\ref{ss:upper_bound}) and denote it also by $\mu_t$.

\begin{proposition}\label{upper}
Let $(f,K)\in \sA^{\BD}_+$  satisfy the weak isolation condition and be non-exceptional and without indifferent periodic orbits.
Given numbers $\alpha\le\beta$  and $q$ such that $\alpha^\sharp\le q \le \beta$,
for every $x\in \cL(\alpha,\beta)$ not weakly $S'$-exceptional and
there exist a sequence of integers $n_i\to\infty$, a sequence $(q_i)_i$ of positive numbers $q_i\in [q,\beta]$, and $r>0$ such that for every
$\varepsilon>0$ and $i\ge1$ we have
\begin{equation}\label{Shrink3}
 	\diam\Comp_x f^{-n_i}(B(f^{n_i}(x),r))\le e^{-n_i(q_i-\varepsilon)}.
\end{equation}

For every $t<t_+$, the measure $\mu_t$ satisfies
\begin{equation}\label{Shrink4}
 	\mu_t \bigr(\Comp_x f^{-n_i}(B(f^{n_i}(x),r))\bigl)
	\ge e^{ -n_i P(t)}
		e^{-n_iq_i-n_i\varepsilon\lvert t\rvert}\Upsilon(t),
\end{equation}
where $\Upsilon(t):=\Upsilon(\mu_t,\Delta)$ defined in (\ref{Upsilon}), for $\Delta$ depending only on $r$, $f$, and $\text{sgn}(t)$. In the case $K=\hat I_K$, one can take $q_i=q$ for all $i$.
\end{proposition}

\begin{proof}%[Proof of Lemma \ref{upper}]
 The property (\ref{Shrink3}) for $n=n_i\in H$ and $q_i\equiv q$ follows immediately from (\ref{Shrink1a}), (\ref{Shrink1b}), and
Proposition \ref{p.telescope}. (Though later on there will occur a case we need to work with different choice of $q_i$.)
Therefore the rest of the proof will be devoted to (\ref{Shrink4}).

\smallskip
We start by repeating the construction in the proof of Theorem~\ref{chi}. Consider $\widehat W$ and $W=1/2\diamond \widehat W=B(f^n(x),r/2)$ and denote by $\widehat W_k$ the pull-back of $\widehat W$ for $f^k$ which contains $f^{n-k}(x)$.
\smallskip

The general strategy to prove (\ref{Shrink4}) is the same as Proposition~\ref{prop:conicalupper} using the conformal measure $\mu=\mu_t$. We will start considering a particular case.

\smallskip
\noindent\textbf{1. Assume $f^n(W_n)=W$.}\footnote{This is so in the iteration of a rational function case.
Hence for rational maps the same proof will give us (\ref{Shrink4}), see Appendix A.}
\smallskip

\noindent\textbf{1a. The case $t\ge 0$.}
We will pay particular attention to the encounter of critical points. As in the proof of Theorem~\ref{chi}, let $0\le k_1<\ldots<k_{M'}\le n-1$ be the subsequence of all consecutive integers $k$ such that $\widehat W_k\cap\Crit\ne\emptyset$.
Here, by Lemma~\ref{Levin} property~\eqref{Shrink2} we capture critical points only $M'\le M$ times.

We pull back $W=B(f^n(x),r/2)$ in place of $B(f^n(x),r)$ and use bounded distortion for $f^{k_{i}-k_{i-1}-1}$ on the pull-back $W_{k_i-1}\ni f^{n-(k_i-1)}(x)$ for $f^{k_i-1}$ as in proof of Theorem~\ref{chi} (the formula (\ref{dist-est})).

\smallskip

In the proof of Theorem~\ref{chi} in the part 1.b) we only needed to compare derivatives.
Here however we pull back $\mu_t$ on $K$ so we must be careful.
Let us explain it in detail.
As in Lemma~\ref{no-truncation} we use the notation $W_n^j=f^j(W_n)$.
By Corollary~\ref{cor:no-truncation} the map $f^{k_i-k_{i-1}-1}$ is a $K$-homeomorphism of $W_n^{n-k_i+1}$ onto $W_n^{n-k_{i-1}}$. These are in fact  diffeomorphisms with bounded distortion, since we take into account, while defining $k_i$, also captures of inflection critical points.
Therefore, there exists some constant $C$ coming from the distortion relation~\eqref{disttt} so that
\begin{equation}\label{noncriticalpullback}
	\frac{\mu_t(W_n^{n-k_i+1})}{ \mu_t(W_n^{n-k_{i-1}})}
	\ge C^t e^{-(k_i-k_{i-1}-1)P(t)}
		\Bigl(\frac {|W_n^{n-k_i+1}|}{ |W_n^{n-k_{i-1}}|}\Bigr)^t.
\end{equation}

If $W_n^{n-k_i}$ contains an inflection critical point $c$, \
$f$ is a $K$-homeomorphism onto
$W_n^{n-k_i +1}$ again by Corollary~\ref{cor:no-truncation} and denoting  $(f|_{W_n^{n-k_i}})^{-1}$ by $g$
we have, (as $\mu_t$ is non-atomic, in particular not having an atom at $c$ we can write even equality in place of the first inequality)
\[\begin{split}
\mu_t(W_n^{n-k_i})
&\ge e^{-P(t)}\int_{W_n^{n-k_i +1}} |g'(y)|^t\,d\mu_t(y)\\
&\ge C^t e^{-P(t)}\Bigl(\frac{|W_n^{n-k_i}|}{|W_n^{n-k_i+1}|}\Bigr)^t \mu_t(W_n^{n-k_i+1}),
\end{split}\]
with a constant $C$ arising from Lemma~\ref{l:crit}, applied for every $y$ under the integral.
We again get \eqref{noncriticalpullback} (for $n=1$).

If $W_n^{n-k_i}$ contains a turning critical point $c$, then we apply the above
estimate by choosing as $g$ the branch of $f^{-1}$ on $W_n^{n-k_i +1}$ mapping it onto $W_n^{n-k_i}$ according to Corollary~\ref{cor:no-truncation}.

Notice finally that the case $W_n^{n-k_i}$ contains no critical points can happen, since
this set can be strictly contained in $W_{k_i}$. Then $f$ is a $K$-homeomorphism on its image and we also get the above estimate, referring to Remark~\ref{unifyingdistortion}.

Composing those inequalities we obtain
\begin{equation}\label{givemealabel}
\mu_t(W_n)\ge C^{t(2M'+1)}e^{-nP(t)}
	\Bigl(\frac{|W_n|}{|f^n(W_n)|}\Bigr)^t \mu_t(f^n(W_n)).
\end{equation}
This  ends the proof as we assumed $f^n(W_n)=W$ and hence
\[
	\mu_t(B(f^n(x),r/2))\ge \Upsilon(\mu_t, r/2)>0
\]	
see \eqref{Upsilon} and~\eqref{Shrink3}.

\medskip

\noindent\textbf{1b. The case $t< 0$.}
By Lemma~\ref{bunch} applied to $X=K$ there is some $r'=\varepsilon(r/2,M+1)/2$ and $z\in K\cap B(f^n(x),r/2)$ such that $B(z,r')$ is disjoint with $\bigcup_{j=1}^n f^j(\Crit\cap W_j)$.
Then $f^n$  has no critical points and has bounded distortion on the  pull-back $W'_n$ of $W'=B(z,r'/2)$ contained in $W_n$. Hence, by~\eqref{disttt} and the telescope estimate in Proposition~\ref{p.telescope}, since $B(z,r')$ is $1/2$-scaled neighborhood
of $B(z,r'/2)$, for every $y\in W'_n$
$$
|(f^n)'(y)| \ge \frac{1}{C(1/2)} \frac{r'}{|W'_n|}
	\ge \frac{C_1}{ |W_n|}
	\ge C_2 |(f^n)'(x)| e^{-n\varepsilon}
$$
for arbitrarily small $\varepsilon>0$ and some constants $C_1,C_2$. As $f^n$ is a $K$-homeomorphism on $W_n'$, by Corollary~\ref{cor:no-truncation}, for $t<0$ we get
\[\begin{split}
\mu_t \bigr(\Comp_x f^{-n}(B(f^{n}(x),r))\bigl)
&\ge\mu_t (W'_n) \\
&\ge C_2^{\lvert t\rvert}
e^{ -n P(t)}  |(f^n)'(x)|^{-t} e^{-n\lvert t\rvert\varepsilon}
\cdot\Upsilon(\mu_t,r'/2)
\end{split}\]
giving again \eqref{Shrink4}, with $\Upsilon(t)=\Upsilon(\mu_t,r'/2)$.

\smallskip

In the Steps 1a and 1b we showed the assertion under the assumption that  $f^n(W_n)=W$. However this equality, which is always true for rational maps, is not true in general. To prove the general case, we need to modify the construction.

The problem is, that when $f^n\colon W_n\to W$ is not onto, during the pull-back we can lose part of the conformal measure.
To have $f^n$ onto, we replace in this construction $W$ by
$f^n(W_n)$. Then however we lose the uniform bound $\mu_t(f^n(W_n))\ge c>0$.
(Notice that in Subsection \ref{ss:upper_bound} we did not have this trouble.)

A priori $\mu_t(f^n(W_n))$ (or its piece if $t<0$ as in 1b.) could be arbitrarily close to 0. Notice however that if there is a constant $\delta>0$ such that there is $y\in K$ with $B(y,\delta)\subset f^n(W_n)$ then as before $\mu_t(f^n(W_n))\ge\Upsilon(\mu_t,\delta)$ and we are done. (Compare \eqref{ball2} to be used also later on.)

 \smallskip

\noindent\textbf{2. Assume $f^n(W_n)\ne W$, but $K=\hat I_K$.}
In this case the above holds. Indeed, we assume that $r$ is smaller than the lengths of all gaps in $\bR\setminus \hat I_K$ and all components of $\hat I_K$. For $n$ large enough $|W_n|<r$ by BaShrink, see Subsection~\ref{ss:TCE}.  Hence, as $x\in K$, at least one end, say $z$,
 of $W_n$ is in $K$. Hence $f^n(z)\in K$. Since $\lvert f^n(x)-f^n(z)\rvert=r/2$ and $f^n(x)\in K$, so letting $y$ being the middle point between $f^n(x)$ and $f^n(z)$ and $\delta:=r/4$ we have $\mu_t(f^n(W_n))\ge\Upsilon(\mu_t,\delta)$ and we are done.

 \smallskip

\noindent\textbf{3. The general $(f,K)\in\sA^{\BD}_+$.}
We shall overcome the difficulty with estimating $\mu_t(f^n(W_n))$
from below by acting similarly as in the last part of the proof of Theorem~\ref{chi}. Unfortunately it can happen that we need to replace $q$ by some $q'>q$.

\smallskip

We first restrict our considerations to a particular subset of the Pliss hyperbolic times $H$, see Lemma~\ref{lemma-pliss}.
Fix a number $\hat n\in \N$ sufficiently large such that $\#(H\cap \{1,\ldots,\hat n\})\ge \overline d(H)/2$.
Then, for $d=\overline d(H)/4$ we have $\#(H \cap \{\lfloor \hat n\, d \rfloor, \ldots, \hat n\})\ge d \hat n$.
Applying now Lemma~\ref{combinatorial} to $d$, $k=\#\partial_{r/3}+1$ (see the notation in the proof of Theorem~\ref{chi}), $\hat n$, and the set $H_{\hat n}:= H\cap \{1,\ldots,\hat n\}$, we find a number  $m=m(d,k)>0$ and numbers $a_1<a_2<\cdots <a_k$ all in $H_{\hat n}$ such that $a_k-a_1\le m$.

Consider now an arbitrary  constant $b\in(0,1)$ (to be specified later on), close to 1.
Denote
\[
	B:=B(\delta_{r/3},(1-b)r/4)\,.
\]
We consider two cases.
\smallskip

\noindent\textbf{3a. Assume there exists  $n\in H_{\hat n}$ with $f^n(x)\notin B$.}
Then, comparing (\ref{ball2}) and the notation in the proof of Theorem~\ref{chi}, there is some $a\in(0,1)$ such that some point $x'\in K$ is well inside $f^n(W_n)$, that is, $x'\in a\diamond(f^n(W_n))\cap K$, and thus $f^n(W_n)$ contains the ball $B(x',(1-b)r/4)$. So, as above, we conclude
\[
	\mu_t(f^n(W_n))\ge \Upsilon(\mu_t,(1-b)r/4)>0\,.
\]
This proves the assertion in the case $t\ge0$.

For $t<0$ we proved already that there is some point $x''\in K$ such that $B(x'',(1-b)r/4)\subset f^n(W_n)$. As in Step 1b, applying Lemma~\ref{bunch} to $X=K$ there is $r''=\varepsilon((1-b)r/4,M+1)/2$  and $z\in K\cap B(x'',(1-b)r/4)$ such that the pull-back of $B(z,r'')$ along the trajectory of $x$ does not encounter critical points. As before we obtain
\[
		\mu_t(B(x'',(1-b)r/4))\ge \Upsilon(\mu_t,r'')>0\,.
\]
This proves the assertion in the case $t<0$.
\smallskip

\noindent\textbf{3b. Assume $f^n(x)\in B$ for all $n\in H_{\hat n}$.}
In particular, this holds for all $n=a_1<a_2<\cdots<a_k$ in $H_{\hat n}$ as chosen above.
For any integer $N>0$ define
\[
	\alpha_N:=
	{\rm{dist}}\Big(\bigcup_{j=1}^N f^j(\partial_{r/3})\setminus \partial_{r/3},
		\partial_{r/3}\Big)>0\,.
\]	
Then for each $j=1,\ldots,N$ we have for each $z\in \partial_{r/3}\setminus\bigcup_{j=1}^N f^{-j}(\partial_{r/3})$
\begin{equation}
	f^j(B(z,\Lip^{-N}\alpha_N/2))\cap B(\partial_{r/3},\alpha_N/2))=\emptyset,
\end{equation}
where $\Lip$ denotes the Lipschitz constant of $f|_K$. Hence, comparing the definition of $\alpha'$ preceding (\ref{M'}), for
$$
\alpha_N':=\min\Big\{ \frac{\alpha_N}{2\Lip^N}, \frac\Theta2, \frac\theta2\Big\},
$$
 for each $0<s_1<s_2$,  if $f^{s_1}(x)\in B(z_1,\alpha'_N)$ and
$f^{s_2}(x)\in B(z_2,\alpha'_N)$ for $z_1,z_2\in\partial_{r/3}$, then either $s_2-s_1>N$ or
$$
f^{s_2-s_1}(z_1)=z_2,
$$
compare (\ref{end-orbit}).

Now we fix $N:=m$ and let $s_1,s_2$ be any $a_i,a_j$ for $i<j$.  We fix $b$ such that $(1-b)r/4=\alpha'_m$.
Hence $f^{a_i}(x)\in B(\partial_{r/3},\alpha'_m)=B$ for all $i=1,\ldots,k$.

Therefore there is an index $i$ such that $y:=f^{a_i}(x)$ is
within the distance at most $\alpha'_m$ from a periodic point $p\in \partial_{r/3}$, of period at most $m$. Let us take $n:=a_i$.
Note that $p$ is hyperbolic repelling.\footnote{The remainder of the proof would also work if $p$ were  indifferent, repelling to one side, but in this proof we had to exclude indifferent periodic orbits to refer to  \cite{PrzRiv:13} for the existence of $\mu_t$.}

If $y=p$ then $\chi(x)=\chi(p)$, and in particular $\chi(x)$ exists, contrary to our assumption that $x\in  \cL(\alpha,\beta)$ for $\alpha\not=\beta$.
So we can assume $y\not=p$.
Let $\ell\ge n$ be the least integer, such that $f^{\ell}(x)\notin B$.
It exists by the definition of $\theta$.

Considering the pull-back $W_{\ell}$ of $B(f^{\ell}(x),r/2)$ containing $x$, we repeat the reasoning in Case 3a. This provides us lower bounds for $\mu_t$.

Unfortunately we do not know now whether $q(\ell):=\log\,\lvert(f^{\ell})'(x)\rvert/ \ell$ is in $[q-\e,q+\e]$. This depends on $\chi(p)$.
We deal with this difficulty using the item (\ref{auxiliary}) in Lemma~\ref{Levin}. For simplicity, we assume the period of $p$ to be equal to 1.

\smallskip
\noindent\textbf{Case $1$. $\chi(p)\ge q-2\e$.}
Then $q(\ell)\ge \min\{q-\e,\chi(p)-\e\} \ge q - 3\e$. We obtain (\ref{Shrink3}) and (\ref{Shrink4}) with $\ell$ in place of $n$, with $3\e$ in place of $\e$ and with  $q(\ell)\ge q(n)$.

\smallskip
\noindent\textbf{Case $2$. $\chi(p)< q-2\e$.}
In this case
\[
q(t) \leq \frac 1t (nq(n) + (t-n)(\chi(p)+\e)) < q(n)
\]
for $t\in (n,l]$. In particular, $n^{\rm up} >\ell$. Hence, by \eqref{auxiliary}, $q(\ell) \geq q(n) - \e \geq q - 3\e$.
Hence we obtain  (\ref{Shrink3}) and (\ref{Shrink4}) with $\ell$ in place of $n$, $3\e$ in place of $\e$ and with $q(\ell) \geq q - 3\e$.
\smallskip

To define the sequence $(n_i)_i$ in the proposition, we repeat the above steps infinitely many times, for some increasing sequence of restricted Pliss hyperbolic times obtaining  (\ref{Shrink3}) and (\ref{Shrink4}) for some numbers $n_i$ and $q_i:=q(n_i)\in[q,\beta]$.
\end{proof}

Now we can finally prove Theorems~\ref{main3} and~\ref{strangestrange}.
\begin{proof}[Proof of Theorem~\ref{main3}]\

\noindent\textbf{1. Special case.} Suppose that in Proposition~\ref{upper}  all $q_i$'s are equal to $q$, which occurs, for example,  if $K=\hat I_K$ (and also in the case of a rational function on the Riemann sphere as in Appendix A).

Let $x\in \cL(\alpha,\beta)$.
We can assume $x$ is not weakly $S'$-exceptional by replacing it by $f^n(x)$ for some $n>0$, see Lemma~\ref{lem:denseback}.
(%the local dimension of each $\mu_t$ is the same at both points
Indeed, for each $n$ we conside $X_n:=\{x\in\cL(\alpha,\beta)\colon f^n(x)\text{ is weakly regular}\}$ and any upper estimate of $\dim_{\rm{H}}(f^n(X_n))$ holds automatically for $X_n$. Notice that $x\in \cL(\alpha,\beta)$ implies that $x$ is not precritical, hence $f^n(x)\in\cL(\alpha,\beta)$.%
). For every $q\in [\alpha^\sharp,\beta]$, for every $t<t_+$ for the measure $\mu_t$ as in the proposition, the lower local dimension at $x$ satisfies
\[\begin{split}
	\underline{d}_{\mu_t}(x)
	&\le \liminf_{i\to\infty}\frac{\log \mu_t \bigr(\Comp_x f^{-n_i}(B(f^{n_i}(x),r))}
		{\log \diam\Comp_x f^{-n_i}(B(f^{n_i}(x),r))}\\
	&\le \liminf_{i\to\infty} \frac{ -n_i (P(t) +
		(q+\varepsilon)t)+\log\Upsilon(t)}{ -n_i(q-\varepsilon)}
	\le \frac{P(t)+qt}{q}+\e',
\end{split}\]
with $\e'>0$ arbitrarily close to 0 for $\e$ appropriately small.
We conclude that
\begin{equation}\label{local}
\underline{d}_{\mu_t}(x)\le \frac{P(t)+qt}{q}=\frac{P(t)}{q}+t\,.
\end{equation}
\smallskip

\noindent\textbf{2. General case.}
In general, Proposition \ref{upper} provides us only with a sequence of $q_i\in [q, \beta]$ instead of $q_i \equiv q$.
This leads to the estimate
\begin{equation} \label{local2}
\underline{d}_{\mu_t}(x)\le \max_{q_i \in [q,\beta]}\frac{P(t)}{q_i}+t\,.
\end{equation}
\smallskip

\noindent\textbf{3. Conclusion.}
To obtain the estimates~\eqref{eq:main3ineq} and~\eqref{eq:main3ineq2} we will use the Frostman Lemma.  To obtain~\eqref{eq:main3ineq} we consider $q=\beta$ and then~\eqref{local} and~\eqref{local2} coincide. We can choose $t<t_+$ such that the right hand side in~\eqref{eq:main3ineq} takes value arbitrarily close to $F(\beta)$, recall \eqref{def:Fa}, and the assertion follows from Frostman Lemma.

To obtain the additional upper bound for \eqref{eq:main3ineq2} (note that our assumptions for this inequality guarantee that we are in the special case here) we will use $q=\alpha^\sharp$. As we are in the special case, \eqref{local} holds and we can find $t<t_+$ such that the right hand side of \eqref{local} is arbitrarily close to $F(\alpha^\sharp)$, and then apply the Frostman Lemma.

This proves Theorem~\ref{main3}.
\end{proof}

\begin{proof}[Proof of Theorem~\ref{strangestrange}]
We are in the proof of Theorem \ref{main3}, in the special case. Let $\alpha$ and $\beta$ be such that $\alpha^\sharp(\alpha, \beta)<\chi_{\inf}$. Let $x\in \cL(\alpha, \beta)$. We have \eqref{local} for any $t<t_+$. If $\alpha^\sharp < \chi_{\inf}$ then $F(\alpha^\sharp)=-\infty$, hence there exist $t<t_+$ such that
\[
\frac {P(t)} t + t < 0\,.
\]
By \eqref{local}, this implies that the corresponding measure $\mu_t$ has negative local dimension at $x$ (that is, the measure $\mu_t(B(x,r))$ escapes to infinity as $r\to 0$), which is impossible. Hence, no such point $x$ can exist.
\end{proof}

\appendix
\section{Strong upper bound. Holomorphic case}
We consider a rational map $f\colon\overline\bC\to\overline\bC$ of degree at least 2. Denote by $J$ its Julia set (it corresponds to the set $K$ in our interval maps notation).
The sets considered below are subsets of $J$.

The following results provide the complex counterparts of our main results. The first one
is a slightly strengthened version of \cite[Theorem 2]{GelPrzRam:10} (observe that there we only assumed that $\alpha>0$). The other two are new. For their proofs we refer to the remarks and footnotes in Section~\ref{strong}.

\begin{theo-app}\label{Jmain2}
Let $f$ be a non-exceptional rational map of degree $\ge2$. For any $\alpha\le \beta \le \chi_{\sup}$ with $\beta>0$, and additionally with $\alpha>0$ if $\chi_{\inf}=0$,  we have
\begin{equation}\label{Jupper-weak}
\min\{F(\alpha), F(\beta)\}\le \dim_{\rm H} \cL(\alpha, \beta)\le
\max\big\{0,\max_{\alpha\leq q \leq \beta}F(q)\big\} .
\end{equation}
In particular, for any $\alpha\in [\chi_{\inf},\chi_{\sup}] \setminus \{0\}$ we have
\[
\dim_{\rm H} \cL(\alpha) = F(\alpha)
\quad\text{ and }\quad
\dim_{\rm H} \cL(0) \geq F(0)\,.
\]
Moreover,
\[
\left\{ x\in K\colon -\infty<\chi(x)<\chi_{\inf}\right\}
 = \left\{ x\in K\colon \overline{\chi}(x)>\chi_{\sup}\right\}
 = \emptyset
\]
and
\[
\dim_{\rm H}\left\{ x\in K\colon 0<\overline{\chi}(x)<\chi_{\inf}\right\} = 0\,.
\]
\end{theo-app}

\begin{theo-app}\label{Jmain3}
Let $f$  be a non-exceptional rational map of degree $\ge 2$ without indifferent periodic orbits.
For any $\alpha\le \beta \le \chi_{\sup}$ with $\beta>0$, we have
	\begin{equation}\label{Jeq:main3ineq2}
\dim_{\rm H}\cL(\alpha,\beta) =
\max\,\bigl\{0,	\min  \{F(\alpha^\sharp), F(\beta)\}\,\bigr\},
\end{equation}
where
\[
	\alpha^\sharp:=\frac{\beta}{1+(\beta-\alpha)/\chi_{\sup}}.
\]	
\end{theo-app}

%\begin{theo-app 
\begin{theo-app}\label{Jstrangestrange}
Let $f$  be a non-exceptional rational map without indifferent periodic orbits.
 Assume $\chi_{\inf}>0$ ($f$ is Topological Collet-Eckmann). Then, for any $\alpha\le\beta\le\chi_{\rm sup}$ with $\beta>0$, if  $\alpha^\sharp<\chi_{\inf}$, then
 $\cL(\alpha,\beta)=\emptyset$.
\end{theo-app}

%----------------------------------------
\section{On Rivera-Letelier's proof of $\chi(x)\ge\chi_{\inf}$}
%------------------------------------------

After this paper was written, J. Rivera-Letelier explained to us that~$\chi(x)\ge \chi_{\inf}$, more
precisely \eqref{empty1} in Theorem~\ref{main2}, follows easily from his paper \cite{R-L:}.
%his inequality $\lim_{n\to\infty} \frac1n\log\bigl(\max\{|W|: W \hbox{connected component of}\; f^{-n}(T)\}\bigr)$ for $T$ intervals in $\hat I$ small enough, proved in \cite{R-L:} (under additional, but not substantial, assumptions on $f$ ). See Appendix B for details.}
We devote this Appendix to a brief explanation and comparison to our proof.

In \cite{R-L:} multimodal maps of the unit interval $I$ are considered with $K=J(f)$ being Julia set (see Subsection 1.2: Periodic orbits), of class $C^3$, with all periodic orbits in $K$ hyperbolic repelling and
$f|_K$ topologically exact (which is formally stronger than having positive entropy; see \cite{PrzRiv:13} for a discussion concerning a comparison of these notions). His theory is however working in our more general
setting of $(f,K)\in \sA^{\BD}$.

%For every open interval $T\subset\bR$ of length at most $r$, intersecting $K$,
Define
\begin{equation}\label{ExpShrinkExp}
\chi_{\rm{ExpShrink}}:= -\limsup_{n\to\infty} \frac1n\log\bigl(\max\{|T_n|\}\bigr),
\end{equation}
where the maximum is taken over all connected components $T_n$ of $f^{-n}(T)$ intersecting $K$, over all open intervals $T\subset\bR$ of length at most $r$, not intersecting $G$ being the union of all components of $\bR \setminus K$ longer than some constant $r'$ (`large gaps'). The number $r>0$ is also a constant small enough, in particular not bigger
 than
$\delta$ in the definition of Backward Shrinking in Subsection~\ref{ss:TCE}. (Then $\chi_{\rm{ExpShrink}}$ occurs  independent of $r$ by a variant of the Telescope Proposition~\ref{p.telescope}.)

Then a part of the Main Theorem in \cite[Section 4]{R-L:} asserts (in our setting)
that
%if $\chi_{\inf}>0$ then
%for all $T$ of length $r$ small enough and not intersecting $\partial$, denoting the union of the boundaries of the components of $\bR\setminus K$ longer than a constant,
\begin{equation}\label{comparison}
\chi_{\rm{ExpShrink}}\ge \chi_{\inf}.
\end{equation}

\begin{corollary}
Let $(f,K)\in \sA^{\BD}_+$ satisfy the weak isolation condition. For every regular $x\in K$ (i.e. such that $\chi(x)$ exists), if $\chi(x)>-\infty$, then
\begin{equation}\label{111}
\chi(x)\ge \chi_{\inf}.
\end{equation}
\end{corollary}

\begin{proof}
First notice, as in Proof of Theorem~\ref{chi}  which refers to \cite{GelPrzRam:10}, that  $\chi(x)\ge 0$.
%We can assume that $\chi_{\inf}>0$ since otherwise there is nothing to prove.
Notice also that by the existence of a finite $\chi(x)$ we have
\begin{equation}\label{subexponential_distance}
\liminf_{n\to\infty}{\rm{dist}}(f^n(x),\Crit(f))=0.
\end{equation}
Using this, one can prove, compare \cite[Claim in Section 3]{GelPrzRam:10}, that there exists $r>0$ (maybe smaller than the previous one) such that for any $\e>0$ arbitrarily  close to 0, for every $n$ large enough the map $f^n$ is a diffeomorphism of
$\Comp_x f^{-n}(B(f^n(x),r\exp(-\e n)/2))$ onto $B(f^n(x),r\exp(-\e n)/2)$ with distortion bounded by, say, $1/2$. Therefore $\chi(x)\ge \chi_{\rm{ExpShrink}}-\e$, hence taking $\e\to 0$
\begin{equation}\label{chichi}
\chi(x)\ge \chi_{\rm{ExpShrink}}.
\end{equation}
Hence, by \eqref{comparison} applied to $T=B(f^n(x),r/2)$,   we get \eqref{111},
provided there is a sequence $n_j$ such that all $f^{n_j}(x)$ are within the distance at least $r/2$ from $\partial G$.
This is not the case only if $x$ is pre-periodic, where $\chi(x)\ge \chi_{\inf}$ holds directly.
\end{proof}

\begin{remark}
In fact we use above only an easy case of \eqref{comparison}, where the distance of $T$ from $\partial G$ is bounded away from 0.
This is explained in more detail below.
\end{remark}

\

The strategy above, together with the proof of \eqref{comparison} is in fact similar to ours. Here is a comparison of both.

\medskip

$\bullet$ Our proof is direct. We shadow $x,...,f^n(x)$, extended by a block of a backward trajectory of $x$  of length bounded by a constant independent of $n$, ending close to $f^n(x)$ (we call it `closing the loop'),
by a periodic trajectory of a point $p_n$ close to $K$ thus in $K$ by the weak isolation assumption.
TCE allows to easily compare the derivatives, though on the other hand we can consider only $n=n_j$, see
\eqref{TCE}, which causes, fortunately minor, difficulties. Hence
$\chi(x)\ge\liminf_{n\to\infty}\chi(p_n)\ge\chi_{\inf}$.

We can `close the loop' whenever we can choose $n$ so that $f^n(x)$ is not too close to $\partial G$, see Section~\ref{s:completeness}, \eqref{a}. If we cannot then $x$ is pre-periodic as above and we are done.
%, similarly to \cite{R-L:}.

\medskip

$\bullet$ To follow \cite{R-L:} one considers an auxiliary number $\chi_{CE2}(z)$ and proves
that for $z$ being safe and hyperbolic, see definitions in
\cite[Definition 1.22 and Definition 1.23]{PrzRiv:13} and \cite[Definition 12.5.7]{PrzUrb:10},
\begin{equation}\label{4chi}
\chi(x) \ge \chi_{\rm{ExpShrink}} \ge \chi_{CE2}(z)\ge \chi_{\inf}
\end{equation}

\smallskip

{\bf 1.} The first inequality has been already proven above. It is sufficient to consider $x$ not pre-periodic, so we can consider in the
definition of $\chi_{\rm{ExpShrink}}$ only $T$ within the distance at least $r$ from $G$, getting the same inequality.

%by applying~\eqref{subexponential_distance} we replace $\chi(x)$ by $\chi_{\rm{ExpShrink}}(r)$ as above.

\medskip

{\bf 2.}
Define  $\chi_{CE2}(z)$ for a point $z\in K$ as logarithm of $\lambda_{CE2}(z)$ being the supremum of such
$\lambda$ that there exists $C=C(z,\lambda)>0$ such that for every $n \ge 1$ and
every $w \in f^{-n}(z)$,
\begin{equation}\label{CE2}
|(f^n)'(w)|\ge C \lambda^n.
\end{equation}
Here $w$ is not necessarily in $K$ but we consider all $w$ such that $\Comp_w f^{-n}(B(z,r_1))$ intersects $K$, for a constant $r_1$.

\smallskip

Safe and hyperbolic points exist, see\cite[Lemma 4.4]{PrzRiv:13}.
%It is easy to see that $z$ safe implies $\chi_{CE2}(z)\ge 0$.

\smallskip

%Recall the definitions, see \cite[Definition 12.5.7]{PrzUrb:10} and \cite{PrzRiv:13}:

%$\bullet$ We call $z\in K$
%\textit{safe} (or $\Crit$-\textit{safe} if $z\notin \bigcup_{j=1}^\infty f^j(\Crit(f))$ and for every $\delta>0$ and all $n$ large enough $B(z, \exp (-\delta n))\cap \bigcup_{j=1}^n(f^j(\Crit(f)))=\emptyset$.

%$\bullet$ We call $z\in K$  \textit{hyperbolic} if there exist $\Delta>0$ and $\lambda=\lambda_z>1$ such that
%for all $n$ large enough $f^n$ maps 1-to-1 with bounded distortion the interval $\Comp_{z}(f^{-n}(B(f^n(z), \Delta)))$to $B(f^n(z), \Delta)$ and $|(f^n)'(z)|\ge \lambda^n$.

The conditions `safe' and `hyperbolic" allow to find via shadowing, as in \cite[Lemma 3.1]{PrzRivSmi:03} or e.g. \cite[Section 3]{PrzRiv:13},
compare also Corollary~\ref{cor:conicalupperlower} (though we do not care in this Corollary about the times of going to large scale), periodic points $p_n$ such that
$\chi_{CE2}(z)\ge\liminf_{n\to\infty}\chi(p_n)\ge\chi_{\inf}$, thus proving the last inequality
in \eqref{4chi}

%of that for every $1<\lambda<\exp\chi_{\inf}$  The measures we consider to estimate $\chi_{\inf}$ are just measures equidistributed on periodic orbits resulting from the shadowing.

(The property $\chi_{CE2}(z)>0$  is called $CE2^*(z)$ (or backward Collet-Eckmann condition at $z$ for preimages close to $K$), see \cite[Section 3]{PrzRiv:13}. So the last inequality yields LyapHyp $\Rightarrow CE2^*(z)$ for safe hyperbolic $z$.)

%Thus, if we denote $\chi_{CE2(z)}=\liminf_{n\to\infty} \frac1n\log \inf_w\{|(f^n)'(w)|$ for $w$ as above, we obtain
%\begin{equation}\label{UHPer}\chi_{CE2}(z)\ge \chi_{\inf}.\end{equation}

\medskip

{\bf 3.} The novelty in \cite{R-L:} is proving the missing inequality
\begin{equation}\label{missing}
\chi_{\rm{ExpShrink}} \ge \chi_{CE2}(z)
\end{equation}

As remarked above we need here only its easy case where $T$ is not close to $\partial G$. Then also the proof in
\cite[Proof of Theorem C. Case 1]{PrzRiv:13} works. We consider a finite set $Y$
of $f^j$-preimages $y\in K$ of $z$, $j=0,1,...$
which constitute a $\delta$-dense set for $K$ for $\delta$ small enough. By finiteness for every $\lambda$ as in the definition of $\chi_{CE2}(z)$ we can find a common $C$. If $T$ is not close to $\partial$ we can assume it is between two consecutive points $y,y'$ of $V$. Then pulling $[y,y']$ back we obtain \eqref{missing} for $z$ being $y$ or $y'$.

\smallskip

Note the following.

The pull-backs of $T$ become short after a bounded time by Backward Shrinking, see Subsection~\ref{ss:TCE}.
Note that we can assume $\chi_{CE2}(z)>0$, otherwise there is nothing to prove, since $\chi_{\rm{ExpShrink}}$ is obviously non-negative.
Then we deal with short intervals and capture critical points rarely. One proves by induction that increasingly  rarely. This yields at most subexponential growth of $|T_n|\cdot |(f^n)'(x_n)|$. The proof resembles the proof of the Telescope Proposition~\ref{p.telescope}.
%(the considered captures are  by pull-backs of $1/2$-scaled neighbourhoods of pullbacks
One need not use TCE, one allows to capture critical points unbounded number of times.

In the complex situation this proof does not work. One does not have the Backward Shrinking property
and the pull-backs already at the beginning can become uncontrollable large. Fortunately
one has a different method: joining $f^n(x)$ to $z$ satisfying CE2 with a curve $\gamma$ in sufficient distance from
critical values (of subexponential quasi-hyperbolic length, see \cite{GelPrzRam:10} and \cite{Prz:99}).

Surprisingly our proof works in the complex case too, since the Backward Shrinking (even ExpShrink) holds due
to TCE. Since $f$ is open in the complex case, only the easy part of the proof of Theorem 5.1 is needed (\eqref{a} holds automatically). So we get a new proof of $\chi(x)\ge\chi_{\inf}$, easier than the one in \cite{GelPrzRam:10}, just by
shadowing by periodic orbits,
without joining $f^n(x)$ with $\gamma$ to a remote safe hyperbolic $z$.

%--------------------------------------------------------------------------------------------------------
\bibliographystyle{amsplain}

\begin{thebibliography}{11}

\bibitem{ABV} J. Alves, C. Bonatti, M. Viana,  \emph{SRB measures for partially hyperbolic systems whose central direction is mostly expanding}, Invent. Math. \textbf{140} (2000), 351--398.
%
%\bibitem{Bow:73b} R.~Bowen, \emph{Topological entropy for noncompact sets}, Trans. Amer. Math. Soc. \textbf{184} (1973), 125--136.
%
\bibitem{ColLebPor:87} P.~Collet, J.~Lebowitz, and A.~Porzio, \emph{The dimension spectrum of some dynamical systems}, J. Statist. Phys.~\textbf{47} (1987), 609--644.
%
\bibitem{CorRiv:13} D. Coronel and J. Rivera-Letelier, \emph{Low-temperature phase transitions in the quadratic family}, Adv. Math. \textbf{248} (2013), 453�-494.
%
\bibitem{DenMauNitUrb:98} M. Denker, R. D. Mauldin, Z. Nitecki, and M. Urba\'nski, \emph{Conformal measures for rational functions revisited}, Fund. Math. \textbf{157} (1998), 161--173.
%
\bibitem{DPU:96} M.~Denker, F.~Przytycki, M.~Urba\'nski, \emph{On the transfer operator for rational functions on the Riemann sphere}, Ergodic Theory and Dyn. Sys. \textbf{16} (1996), 255-266.
%
 \bibitem {EckPro:86} J.-P.~Eckmann and I.~Procaccia, \emph{Fluctuations of dynamical scaling indices in non-linear systems}, Phys. Rev. A \textbf{34} (1986), 659--661.
%
\bibitem{Gel:10} K.~Gelfert, \emph{Expanding repellers for non-uniformly expanding maps with singularities and criticalities}, Bull. Braz. Math. Soc. (2) \textbf{41} (2010), 237--257.
%
\bibitem{GelPrzRam:10} K.~Gelfert, F.~Przytycki, and M.~Rams, \emph{Lyapunov spectrum for rational maps}, Math. Ann. \textbf{348} (2010), 965--1004.
%
\bibitem{GraSmi:09} J.~Graczyk and S.~Smirnov, \emph{Non-uniform hyperbolicity in complex dynamics}, Invent. Math. \textbf{175} (2009), 335--415.
%
\bibitem{Hof:10} F.~Hofbauer, \emph{Multifractal spectra of Birkhoff averages for a piecewise monotone interval map}, Fundam. Math. \textbf{208} (2010), 95--121.
%
\bibitem{InoRiv:12}   I.~Inoquio-Renteria and J.~Rivera-Letelier, \emph{A Characterization of hyperbolic potentials of rational maps}, Bull. Braz. Math. Soc. (N. S.)\textbf{43}  (2012),  99--127.
%
\bibitem{IomKiw:09} G.~Iommi and J.~Kiwi, \emph{The Lyapunov spectrum is not always concave}, J. Stat. Phys. \textbf{135} (2009), 535--546.
%
\bibitem{IomTod:10} G.~Iommi and M.~Todd, \emph{Natural equilibrium states for multimodal maps}, Comm. Math. Phys. 300 (2010), 65--94.
%
\bibitem{IomTod:11} G.~Iommi and M.~Todd, \emph{Dimension theory for multimodal maps}, Ann. Henri Poincar\'e \textbf{12} (2011), 591--620.
%
\bibitem{IomTod:13} G.~Iommi and M.~Todd, \emph{Thermodynamic formalism for interval maps: inducing schemes} Dyn. Sys. Issue edited by Quas and Vaienti. \textbf{28.3} (2013), 354--380.
%
\bibitem{Kam:02} A.~Kameyama, \emph{Topological transitivity and strong transitivity}, Acta Math. Univ. Comenianae (N.S.) 71.2 (2002), 139--145.
%
\bibitem{Levin} G.~Levin, F.~Przytycki, W.~Shen, \emph{Lower Lyapunov exponent}, manuscript.
%
%\bibitem{Man:81} A.~Manning, \emph{A relation between Lyapunov exponents, Hausdorff dimension and entropy}, Ergodic Theory Dynam. Systems \textbf{1} (1981), 451--459.
%
\bibitem{MauUrb:00} D. Mauldin and M. Urba\'{n}ski, \emph{Graph Directed Markov Systems: Geometry and Dynamics of Limit Sets}, Cambridge University Press, 2003.
%
\bibitem{deMvanS:} W.~de Melo and S.~van Strien, \emph{One Dimensional Dynamics}, Springer-Verlag, 1993.
%
\bibitem{MisSzl:80} M.~Misiurewicz and W.~Szlenk, \emph{Entropy of piecewise monotone mappings}, Studia Math.  \textbf{67}  (1980), 1, 45--63.
%
\bibitem{NowPrz:98} T.~Nowicki and F.~Przytycki, \emph{Topological invariance of the Collet-Eckmann property for $S$-unimodal maps}, Fund. Math. \textbf{155} (1998), 33--43.
%
%\bibitem{NowSan:98} T.~Nowicki and D.~Sands, \emph{Non-uniform hyperbolicity and universal bounds for $S$-unimodal maps}, Invent. Math. \textbf{132} (1998), 633--680.
%
%%\bibitem{NowPrz:89}  T. Nowicki, F.~Przytycki, \emph{The conjugacy of Collet-Eckmann's map of the interval with the tent map is H\"older continuous}, Ergodic Theory Dynam. Systems \textbf{9} (1989), 379--388.
%%\marginpar{\tiny do we refer to it? F\\ \textcolor{red}{No, K.}. So I remove NowPrz:89}
%
\bibitem{Ols:95} L.~Olsen, \emph{A multifractal formalism}, Adv. Math. \textbf{116} (1995), 82--196.
%
%
\bibitem{Pli:72} V.~Pliss, \emph{On a conjecture due to Smale}, Differ. Uravn. \textbf{8} (1972), 262--268.
%
\bibitem{Prz:90} F.~Przytycki, \emph{On the Perron-Frobenius-Ruelle operator for rational maps on the Riemann sphere and for H\"older continuous functions}, Bull. Braz. Math. Soc. (N. S.) \textbf{20} (1990), 95--125.
%
\bibitem{Prz:93} F.~Przytycki, \emph{Lyapunov characteristic exponents are nonnegative}, Proc. Amer. Math. Soc.  \textbf{119}  (1993), 309--317.
%
\bibitem{Prz:98} F.~Przytycki, \emph{Iteration of holomorphic Collet-Eckmann maps: Conformal and invariant measures. Appendix: On non-renormalizable quadratic polynomials}, Trans. Amer. Math. Soc. \textbf{350} (1998), 717--742.
%
\bibitem{Prz:99} F.~Przytycki, \emph{Conical limit set and Poincar\'e exponent for iterations of rational functions}, Trans. Amer. Math. Soc. \textbf{351} (1999), 2081--2099.
%
%
\bibitem{Pes:97} Ya.~Pesin, \emph{Dimension Theory in Dynamical Systems: Contemporary Views and Applications},  The University of Chicago Press, 1997.
%
\bibitem{PrzRiv:07} F. Przytycki, J. Rivera-Letelier, \emph{Statistical properties of Topological Collet-Eckmann maps}, Annales Scientifiques de l'\'Ecole Normale Superieure
$4^e$ s\'erie, \textbf{40} (2007), 135--178.
%
%\bibitem{PrzRiv:11} F.~Przytycki, J.~Rivera-Letelier, \emph{Nice inducing schemes and the thermodynamics of rational maps}, Comm. Math. Phys. \textbf{301}  (2011), 661--707.
%
\bibitem{PrzRiv:13} F.~Przytycki and J.~Rivera-Letelier, \emph{Geometric pressure for multimodal maps of the interval}, {\tt arXiv:1405.2443}.
%
\bibitem{PrzRivSmi:03} F.~Przytycki, J.~Rivera-Letelier, and S.~Smirnov, \emph{Equivalence and topological invariance of conditions for non-uniform hyperbolicity in the iteration of rational maps}, Invent. Math. \textbf{151} (2003), 29--63.
%
\bibitem{PrzRivSmi:04} F.~Przytycki, J.~Rivera-Letelier, and S.~Smirnov, \emph{Equality of pressures for rational functions}, Ergodic Theory Dynam. Systems \textbf{24} (2004) 891--914.
%
%
%%\bibitem{PrzRivSmi:03}
%%F.~Przytycki, J.~Rivera-Letelier, and S.~Smirnov,
%%\emph{Equivalence and topological invariance of conditions for
%%non-uniform hyperbolicity in the iteration of rational maps},
%%Inventiones Mathematicae~\textbf{151} (2003), 29--63.
%
%%\bibitem{PrzRivSmi:04}
%%F.~Przytycki, J.~Rivera-Letelier, and S.~Smirnov, \emph{Equality
%%of pressures for rational functions}, Ergodic Theory Dynam.
%%Systems \textbf{24} (2004), 891--914.
%
\bibitem{PrzRo:98} F.~Przytycki and S.~Rohde, \emph{Porosity of Collet-Eckmann Julia sets}, Fund. Math. 155 (1998), 189--199.
%
\bibitem{PrzUrb:10} F.~Przytycki and M.~Urba\'nski, \emph{Conformal Fractals: Ergodic Theory Methods}, London Mathematical Society Lecture Note Series 371, Cambridge University Press, 2010.
%
\bibitem{Ran:89} D.~Rand, \emph{The singularity spectrum f(a) for cookie-cutters}, Ergodic Theory Dynam. Systems \textbf{9} (1989), 527--541.
%
\bibitem{R-L:} J.~Rivera-Letelier, \emph{Asymptotic expansion of smooth interval maps}, {\tt arXiv:1204.3071v2}.
%
\bibitem{Sch:99} J. Schmeling, \emph{On the completeness of multifractal spectra}, Ergodic Theory Dynam. Systems \textbf{19} (1999), 1595�-1616.
%
%\bibitem{vSV} S.~van Strien and E.~Vargas, \emph{Real bounds, ergodicity and negative Schwarzian for multimodal maps}, J. Amer. Math. Soc. \textbf{17} (2004),  749--782, J. Amer. Math. Soc. 20 (2007), no. 1, 267--268.
%
\bibitem{Wal:81} P.~Walters, \emph{An Introduction to Ergodic Theory}, Graduate Texts in Mathematics 79, Springer, 1981.
%
\bibitem{Wei:99} H.~Weiss, \emph{The Lyapunov spectrum for conformal expanding maps and axiom-A surface diffeomorphisms}, J. Statist. Phys.~\textbf{95} (1999), 615--632.
%
\end{thebibliography}

\end{document}